\documentclass[12pt,a4paper]{article}
\usepackage[T1]{fontenc}
\usepackage{lmodern}
\usepackage[margin=27mm]{geometry}
\usepackage{amsmath,amsfonts,amssymb,amsthm,mathrsfs}
\usepackage{graphicx}
\graphicspath{{Fig/}}
\usepackage{booktabs,multirow,float}
\usepackage{subfig}

\usepackage{pdflscape}
\usepackage{algorithm}
\usepackage{algpseudocode}
\usepackage[doublespacing]{setspace}
\usepackage{color}
\usepackage[authoryear,round]{natbib}
\usepackage{bibunits}
\usepackage[hidelinks,hypertexnames=false]{hyperref}
\allowdisplaybreaks[2]
\theoremstyle{plain}
\newtheorem{theorem}{Theorem}

\newtheorem{lemma}[theorem]{Lemma}
\newtheorem{corollary}[theorem]{Corollary}
\theoremstyle{remark}
\newtheorem{remark}{Remark}

\theoremstyle{definition}

\newtheorem{assumption}{Assumption}
\numberwithin{equation}{section}

\newcommand\bzero{\mbox{\boldmath${0}$}}

\newcommand\beg{\mbox{\boldmath${\eg}$}}

\newcommand\bC{{\bf C}}
\newcommand\bE{{\bf E}}
\newcommand\bF{{\bf F}}
\newcommand\bM{{\bf M}}
\newcommand\bu{{\bf u}}
\newcommand\bR{{\bf R}}

\newcommand\bfx{{\bf x}}
\newcommand\bx{{\bf x}}

\newcommand{\R}{\mathbb{R}}
\newcommand{\PR}{\mathbb{P}}

\newcommand{\E}{\mathbb{E}}

\newcommand{\diag}{\operatorname{diag}}
\newcommand{\N}{\mathbb{N}}
\newcommand{\conp}{\stackrel{\mathbb{P}}{\to}}

\newcommand{\rev}{\textcolor{black}}

\newcommand\bfW{{\mathbf W}}
\newcommand\bPhi{{\boldsymbol \Phi}}

\newcommand \dg{\delta}

\newcommand \eg{\varepsilon}

\newcommand \la{\lambda}

\newcommand \sg{\sigma}

\newcommand\lip{\left\langle}
\newcommand\rip{\right\rangle}

\def\convd{\stackrel{d}{\rightarrow}}
\def\convP{\stackrel{P}{\rightarrow}}

\newcommand{\nn}{\nonumber}

\title{Monitoring for a Phase Transition in a Time Series of Wigner Matrices}
\author{Nina D{\"o}rnemann\textsuperscript{1}\thanks{Nina D{\"o}rnemann, Department of Mathematics, Aarhus University, Ny Munkegade 118
Building 1535, 
8000 Aarhus C, Denmark. \href{mailto:ndoernemann@math.au.dk}{ndoernemann@math.au.dk}}
\quad Piotr Kokoszka\textsuperscript{2}\\[0.4em]
Tim Kutta\textsuperscript{1}\quad Sunmin Lee\textsuperscript{2}\\[1em]
\small \textsuperscript{1}Department of Mathematics, Aarhus University, Denmark\\
\small \textsuperscript{2}Department of Statistics, Colorado State University, USA}
\date{}

\begin{document}
\maketitle

\begin{abstract}
We develop methodology and theory for the detection of a phase transition in a time series of high-dimensional random matrices. 
At each time point \( t = 1,2,\ldots \), we observe a deformed Wigner matrix \( \mathbf{M}_t \), where the unobservable deformation represents a latent signal. 
This signal is detectable only in the supercritical regime, and our objective is to detect the transition to this regime in real time as new matrix-valued observations arrive. 
Our approach is based on a partial sum process of extremal eigenvalues of \( \mathbf{M}_t \), and its theoretical analysis combines state-of-the-art tools from random matrix theory and Gaussian approximations. 
The resulting detector is self-normalized, which ensures appropriate scaling for convergence and a pivotal limit without any additional parameter estimation. \textcolor{black}{We also show that self-normalized detectors can be used for monitoring streams of random matrices beyond the Wigner model.}
Simulations show excellent performance across varying dimensions. 
Applications to pollution monitoring and social interactions in primates illustrate the usefulness of our approach.
\end{abstract}


\defaultbibliographystyle{apalike}
\defaultbibliography{ntB1.bib}

\begin{bibunit}
\section{Introduction} \label{s:i}
We study the problem of monitoring a time series of deformed Wigner
matrices, also known as spiked Wigner matrices,
for the emergence of a detectable signal. Unlike
extensive  research focusing on detecting
a signal in a single matrix, we consider a sequence (stream) of such matrices \rev{of dimension $n\times n$},
i.e.
\begin{align} \label{eq_def_M}
	\mathbf{M}_t:= s _t \cdot \bfx_t\bfx_t^\top + \bfW_t/\sqrt{n},
	\qquad t=1,2, \ldots ~ .
\end{align}
Our goal is to identify in real time the transition to
a regime where the signal  $s _t  \bfx_t\bfx_t^\top$
becomes detectable. Detailed background and problem formulation
are provided below.
Our approach is related  to research on temporal graphs or dynamic networks,
which has recently attracted increasing attention in telecommunications
and information systems, for example, in detecting network intrusions
or other anomalies; see, e.g.,
\cite{wang:localizing:2015,shin:2017,eswaran:2018,kiouche:2020,jiang:liu:2022}, 
among a variety of engineering and computer
science papers focusing on algorithms for specific applications.
In contrast to the analysis of static graphs or networks,
the analysis of streaming network data more closely reflects
both the structure of the data and the purpose of the analysis.
For example, an anomaly (a structural change) at a particular
time point can be more effectively detected by comparing
the current network state to a previous anomaly-free state rather
than searching for an unusual pattern in a snapshot of the network.

\cite{ye:xu:wang:2023} also consider
model \eqref{eq_def_M},   but focus on a
different problem. In the initial period,  there is no signal,
$\bx_t = \bzero$, and then a known
signal emerges at an unknown time $\tau$, and the objective is to sequentially
determine $\tau$ balancing the expected detection delay and the rate of
false alarms. In our setting,  the signal is unknown and the objective is
to detect the transition to a regime where it can be detected. The success criteria
are large probability of correct phase transition identification, if it exists,
and small probability of false alarms, if there is no phase transition.
The problem we study is related to, but quite different from,
sequential detection of change points in dynamic networks.
\cite{yu:2024} study the detection of a change in the expected
value of the adjacency matrix of an inhomogeneous Bernoulli network.
Extending this work, \cite{xu:dubey:yu:2024} consider networks
with missing values and temporal dependence quantified by $\phi$-mixing,
as we do.  The statistical problem studied in these papers is related,
but different; Bernoulli or dot product graphs vs. deformed Wigner matrices; no training
vs. a training period, as we elaborate later.  The main difference, however, lies
in the statistics under investigation. The aforementioned methods
rely on change point statistics for individual entries that
are aggregated across the network. In contrast, we focus on the
leading eigenvalue, which is typically associated with the most
influential direction/node of a  network.
If the leading eigenvalue exceeds a critical threshold,
a phase transition occurs to a regime in which a signal becomes detectable.
To develop a suitable statistical
framework, we combine Gaussian approximations with
state-of-the-art random matrix theory.
A more detailed discussion  of related work is given in Section~\ref{sec_lit}.
We now outline relevant properties of Wigner matrices and our contributions.

{\textbf{Deformed Wigner matrices}\label{Wignerm}}
A Wigner matrix $\bfW \in \mathbb{R}^{n \times n}$ is a random matrix
whose entries $(w_{ij})_{1 \leq i,j \leq n}$ are independent,
up to the symmetry constraint $w_{ij}=w_{ji}$, and satisfy the conditions
\begin{align*}
	\E [w_{ij}] = 0, \quad \E[ w_{ij}^2] = \sigma^2 + \delta_{ij} c,
	\quad \max_{1\leq i,j \leq n} \E | w_{ij}|^p \leq c_p,
\end{align*}
for all integers $p\geq 3$, where $c, c_p $ denote some constants.
The $\dg_{ij}$ are Kronecker deltas.
Wigner matrices are a standard ensemble in random matrix theory and their spectral properties are well-understood; see our literature review in Section \ref{sec_lit}.
In our model, Wigner matrices represent the noise, while the signal is given by a (random) rank one matrix. More precisely,
denoting by $s\ge 0$ the signal strength and by $\bfx$ a vector of unit length (both possibly random), the signal is given by $s \cdot \bfx\bfx^\top$ and the deformed Wigner matrix by
\begin{align} \label{e:def_M}
	\mathbf{M}:= s \cdot \bfx\bfx^\top + \bfW/\sqrt{n}.
\end{align}
The scaling $1/\sqrt{n}$ is standard in random matrix theory,
making noise and signal in an appropriate sense proportional
to each other. It is well-known that in this model only signals
with sufficiently large magnitude $s$ are detectable.
To explain this phenomenon, let us momentarily assume
that $s$ is deterministic.  Then, if $s<\sigma$, the signal
is not detectable in the sense that the maximum eigenvalue
$\lambda=\lambda(\mathbf{M}) $ of the matrix $\mathbf{M}$
satisfies $\lambda \overset{P}{\to} 2 \sigma$ (no information on $s$),
while the largest eigenvector $\mathbf{u} = \mathbf{u} (\mathbf{M})$
is asymptotically orthogonal to the direction of the signal,
$\langle \mathbf{u}, \mathbf{x}\rangle \overset{P}{\to}0$
(no information on $\mathbf{x}$). For $s>\sigma$, however,
both the maximum eigenvalue and its corresponding eigenvector
contain information about the signal. More precisely,
they satisfy $\la \convP s + \sg^{2}s^{-1}$
and $\lip \bu, \bx \rip^2 \convP 1 - \sg^{2}s^{-2}$
\citep{capitaine_et_al_2009,pizzo_et_al_2013,Lee2015}.
Hence, the largest eigenvalue of $\mathbf{M}$ in
\eqref{e:def_M} undergoes a phase transition,
and we refer to $s<\sigma$ as the \textit{subcritical} case,
while $s>\sigma$ is called the \textit{supercritical} case.
This phenomenon is a variant of the BBP-phase transition,
which refers to the  pivotal work of \cite{baik:ben:peche:2005}
on the eigenvalues
of a complex sample covariance matrix. For  model \eqref{e:def_M},
it is a crucial statistical task to determine whether a detectable signal
exists or not, see \cite{jung2024detection} for a recent comprehensive
review.

\textbf{Statistical model}\label{model} We consider the problem
of sequentially monitoring a weakly dependent time series of deformed Wigner matrices for the emergence of a detectable signal. Let $(\mathbf{W}_{n,t})_t$  be a triangular array of Wigner matrices in $\mathbb{R}^{n \times n}$, and $(s_t,\bx_{n,t})_t$ a triangular array of random vectors in $\mathbb{R}^{n+1}$.
We assume that both arrays are independent of each other,
and that for a fixed row (that is, for fixed $n$), the time series of Wigner matrices is strictly stationary in $t$. In particular, all Wigner matrices $\mathbf{W}_{n,t}$ are characterized by the same variance $\sigma^2$ and the same constants $c, c_p>0$.
In the following, we will usually drop dependence on $n$ in
both the vectors and the matrices to reduce the notational burden.

We are interested in monitoring the distribution of the $s_t$ in
\eqref{eq_def_M} for a possible transition from the subcritical to the supercritical
regime. The signals
$s_t$ are not directly observable; we only have access to
the deformed Wigner matrices $\mathbf{M}_t$. Our task is to detect the transition of the  signal $s_t$ to the supercritical regime based solely on the  matrix-valued observations $\mathbf{M}_t$.

To formulate a statistical model for this task, let $\mathbb{P}^{(1)}$ and $\mathbb{P}^{(2)}$ \textcolor{black}{be two probability measures, each with compact support, that is contained in the intervals $[0,\sigma)$ and  $(\sigma, \infty)$, respectively.} \rev{The measures $\mathbb{P}^{(1)}$, $\mathbb{P}^{(2)}$ capture the randomness of $s_t$, which is not necessarily ordered or increasing.} If $s_t \sim \mathbb{P}^{(1)}$, then the model describes the subcritical regime, which is assumed to be the initial state.
In contrast, if $s_t$ is distributed according to $\PR^{(2)}$ for some $t$, then a transition to the supercritical regime has taken place.
To detect a possible transition from $\PR^{(1)}$ to $\PR^{(2)}$, we assume
that the initial $m$ observations $\mathbf{M}_1,\ldots ,\mathbf{M}_m$ correspond  to matrices in the subcritical case. That is, if  $\mathcal{L}(s)$ denotes the law of a random variable $s$, then we have $\mathcal{L}(s_1)=\mathcal{L}(s_2)=\ldots =\mathcal{L}(s_m) = \PR^{(1)}$.  We call $\mathbf{M}_1,\ldots ,\mathbf{M}_m$ the \textit{training sample}.
Thus, we assume that the signal is not detectable during the (anomaly-free) training period. After the training, the \textit{monitoring period} begins, where we sequentially test the hypotheses pair
\begin{align}
	& \mathsf{H}_0: \PR^{(1)}=\mathcal{L}(s_1)=\mathcal{L}(s_2)= \ldots ,\\
	& \mathsf{H}_1: \PR^{(1)}=\mathcal{L}(s_1)= \ldots =\mathcal{L}(s_{m+k^*}) \neq   \PR^{(2)} =  \mathcal{L}(s_{m+k^*+1})=  \mathcal{L}(s_{m+k^*+2}) = \ldots . \label{e:alt}
\end{align}
The location $k^*$ of the phase transition is unknown, if it exists.
This means that under the null hypothesis, the signal is not detectable over the entire period, while under the alternative hypothesis, at some unknown time $k^*$ in the monitoring period, a detectable signal emerges. The distributions $\PR^{(1)}$, $\PR^{(2)}$, the unit vectors $\bx_{t}$ and the parameters $\sg^2$ and $c$
are unknown. Asymptotics are formulated for $m \to \infty$ (longer training period) and $n \to \infty$ (high dimension).

\textbf{Main contributions} Our contributions can be summarized  as follows.

\begin{enumerate}
	\item We develop a statistical test for the emergence of a supercritical signal in a dependent time series of high-dimensional matrices. To the best of our knowledge,  no such  test has been proposed yet. Existing statistical methods in random matrix theory typically confine their analysis to the study of a single matrix,  and do not include  temporal evolution.
	\item We prove that the level of the  test is controlled asymptotically, and that it is power consistent for the alternative $\mathsf{H}_1$. The relation between length of the training period and dimension is kept very general  (they have to be polynomially bounded by each other), and we allow the observations to be weakly dependent over time.
	\item   Our theoretical analysis lies at the intersection of time series analysis and random matrix theory. We derive novel Gaussian approximations for the partial sum process of maximum eigenvalues corresponding to deformed Wigner matrices.
	The derivations of these approximations are challenging, since there do not exist, at present, moment bounds for our objects of interest, i.e. the largest eigenvalues of high-dimensional matrices. \textcolor{black}{We also discuss extensions of our results beyond the Wigner case.}
	\item
	The proposed test statistic is based on the principle of self-normalization that rids our procedure of all nuisance parameters and implies convergence to a simple, pivotal limiting distribution. The proof of convergence is based on a new, finite sample normalization strategy, since (given current theory) it is unclear whether a theoretical long-run variance even exists. In this respect, our work broadens the scope of self-normalizations to non-standard scenarios.
\end{enumerate}

The remainder of the paper is organized as follows. Section \ref{sec_lit}
places our work in a broader perspective. Both sequential testing and random
matrix theory are well-established areas of research, so we can review only
the most closely related papers we know. The new statistical methodology
and the accompanying theory are presented in Section \ref{s:mt}. 
An assessment of the finite-sample behavior of our monitoring method is given
in Section \ref{s:fs}.
\textcolor{black}{All technical details and further numerical and theoretical results are deferred to the Supplement. }
The proofs are gathered in Section \ref{s:p}. 
\textcolor{black}{An extension of our theory to H\"older-weighted scan statistics is given in Section \ref{sec_extension_holeder}.} \textcolor{black}{Section \ref{sec:beyond:wigner} discusses extensions of our methodology beyond the Wigner case.}
\textcolor{black}{Further simulations and a comparison to an entry-wise monitoring scheme are discussed in Section \ref{sec:add:sim}. Details on the data pre-processing applied in Section \ref{s:fs} are given in Section \ref{s:fire}.}

\section{Related research} \label{sec_lit}
In this section,  we review related research on sequential change detection  and
random matrix theory. Our objective is to place  our work within broader
context, rather than provide an exhaustive review.

\rev{
\textbf{Monitoring} 
\cite{aue:kirch:2024} provide a review of change point monitoring
with an initial training sample (a setup  used in this work) and comparisons
to sequential methods where the initial parameter values are known. 
Monitoring is a form of sequential change point detection
where in-control parameters are not specified (e.g. by the production process
requirements), but are estimated from an initial training sample, 
whose length is often denoted by $m$. 
This approach was originally proposed by \cite{chu:stinchcombe:white:1996} 
and developed in many directions, see e.g. 
\cite{horvath:kokoszka:huskova:steinebach:2003}, \cite{berkes:gombay:horvath:kokoszka:2004},  \cite{aue:horvath:huskova:kokoszka:2006},  
\cite{fremdt:2015},  \cite{kirch:weber:2018}, \cite{gosmann:kley:dette:2021} and 
\cite{kutta:doernemann:2025}.  
The above papers focus on time series of scalars or vectors. 
As we do, \cite{he:et:al:2024} study matrix-valued time series, but  of the
form $\bM_t = \bR \bF_t \bC^\top + \bE_t$,  with the objective of monitoring
for a change in the row loadings matrix $\bR$.
There are monitoring schemes for functional and manifold-valued data, see e.g. 
\cite{aue:hormann:horvath:huskova:2014} and  
\cite{kutta:jach:kokoszka:2024}.
Alternative notions of sequential monitoring have been explored, 
and we refer to  \cite{wu:wang:yan:shao:2022,chen:wang:samworth:2022} for recent contributions.\\} %
Existing approaches (without a training period) typically aggregate change point detection statistics
across the entries of random matrices and analyze them using  tools
from the statistical theory of high-dimensional vectors. Our work differs
importantly from this pattern. To the best of our knowledge it is the first
monitoring approach
based on  fluctuations of  the leading eigenvalue. The developed theory
also relies on an entirely different foundation than previous works,
building on recently developed tools in random matrix theory to analyze
convergence and concentration of extremal eigenvalues. A further
difference is the use of an anomaly-free training period for our
algorithms that we combine with a self-normalization to obtain
a completely pivotal procedure. This contrasts with most existing
procedures that involve the (often challenging) choice of regularization
parameters.
Up to this point, self-normalization has been used primarily
in {\it a posteriori} change point detection, see e.g.
\cite{zhang:lavitas:2018}, with a recent exception
in \cite{chan:ng:yau:2021}. Self-normalization leads to
robust inference methods that avoid the difficult estimation
of the long-run variance. Proofs typically rely on weak convergence
of the partial sum process, combined with a continuous mapping theorem,
where the long-run variance cancels out in a carefully constructed limit.
Our approach is different because we cannot even be sure of the existence
of the (theoretical) long-run variance. For this reason, we propose a new
proof strategy  that leverages finite sample Gaussian approximations to
cancel out a finite sample variance, extending the scope of self-normalization
to settings where the long-run variance may not exist.

\textbf{Random matrix theory}
The noise in our model is a sequence of  Wigner matrices; the Wigner matrix is a  fundamental object in random matrix theory. This paper draws on the well-developed literature on spectral behavior of Wigner matrices and we review some of it  here.
For more in-depth studies, we refer  to \cite{erdos:yau:yin:2012} on the concentration of the bulk eigenvalues, \cite{erdos:peche:ramirez:2010} on bulk universality, or  \cite{knowles:yin:2013} on the eigenvector distribution. Broader surveys can be found in \cite{erdHos2011universality} or \cite{arous:guionnet:2011}.

Deformed Wigner matrices are a generalization of classical Wigner matrices, where the random matrix is deformed by introducing a deterministic component or altering the distribution of entries. A common example is the additive deformation, where a  symmetric matrix  is added to the Wigner matrix. Typically, the deformation  of the Wigner matrix is assumed to be of finite (fixed) rank or even of rank one in case of a hidden signal. It is well-known that the leading eigenvalues of deformed Wigner matrices undergo a phase transition under which their asymptotic behavior fundamentally changes (see also our introduction). This phenomenon is known as the  BBP-phase transition due to a pivotal work of \cite{baik:ben:peche:2005}. It has also  been studied in
\cite{feral:peche:2007,benaych-georges:nadakuditi:2009,pizzo_et_al_2013,Lee2015,Knowles2017}, among many others. It is known that the behavior of the extremal eigenvalues of many random matrix ensembles is universal in the sense that the limiting distribution does not depend on the distribution of the entries of the underlying random matrix.
For deformed Wigner matrices, however, the behavior of the leading eigenvalues in the supercritical regime is not universal, see  e.g. \cite{capitaine_et_al_2009,knowles:yin:2014}.

To the best of our knowledge, the problem of detecting the emergence of a signal in a time series of random matrices has not been addressed yet. While our method is placed in a sequential framework, where the user sequentially observes a stream of deformed Wigner matrices, existing methods confine their analysis to the study of a single random matrix \citep{ding2024two, dornemann2024detecting}.
We are, however, not even aware of a significance test for the existence of
a detectable signal even in an individual deformed Wigner matrix.
In particular, we emphasize that the problem of detecting spiked eigenvalues, which has been extensively studied for covariance matrices, see, e.g., \cite{ding2022tracy,onatski_et_al_2014,johnstone_onatski_2020}, and recently for deformed Wigner matrices as in \cite{el_alaoui_et_al,chung2019weak,jung2024detection}, is inherently different from our approach. In those papers, the authors are interested in testing for the number of spiked eigenvalues of any size, no matter whether they fall into the subcritical or supercritical regime. However, not every spike corresponds to a supercritical (detectable) signal in the sense of this paper.
A test for the distinct problem of determining the asymptotic regime of the leading eigenvalue of a (single) covariance matrix was proposed by \cite{doernemann:lopes:2025}.

\section{Detecting the transition to the supercritical regime}\label{s:mt}

\subsection{Notation and mathematical preliminaries}
\label{sec_not_prelim}

\textbf{\(\phi\)-mixing time series} Let \((X_{t,n})_{t \in \mathbb{Z}}\) be for each $n \in \N$ a stationary time series of random variables taking values in a metric space \((\mathcal{X}, d)\). To quantify dependence along $t$, we introduce the \(\phi\)-mixing coefficients, defined for all \(r \geq 0\) as
\[
\phi(r) := \sup_n \sup \Big\{|\mathbb{P}(B | A) - \mathbb{P}(B)|:A \in \mathcal{F}_{-\infty}^{0}(n),\,\, B \in \mathcal{F}_{r}^{\infty}(n),\,\, \mathbb{P}(B)>0 \Big\},
\]
where \(\mathcal{F}_{-\infty}^{0}(n)\) is the \(\sigma\)-algebra generated
by \((X_{t,n})_{t \leq 0}\) and \(\mathcal{F}_{r}^{\infty}(n)\)
is the \(\sigma\)-algebra generated by \((X_{t,n})_{t \geq r}\).
The triangular array \((X_{t,n})_{t \in \mathbb{Z}}\) is called \(\phi\)-mixing
if \(\phi(r) \to 0\), as $r \to \infty$, and the decay rate quantifies the strength of dependence.

\textbf{The Tracy-Widom distribution} The Tracy-Widom distribution arises in the study of the extremal eigenvalues of random matrices and hence plays a key role in describing the transition from subcritical to supercritical regimes. \rev{Since this distribution cannot be described in a simple closed form, it is often characterized as a limiting object.}
For instance, it describes the limiting distribution of the largest eigenvalue of a random matrix drawn from the Gaussian Orthogonal Ensemble  (GOE) as the matrix size tends to infinity (for a detailed description, see \cite{Tracy2000}).
Formally, let $ \lambda_{\max} $ be the largest eigenvalue of a GOE matrix of size \(n \times n\). Then, under appropriate scaling, the distribution of the centered and normalized $\lambda_{\max} $ converges weakly to the Tracy-Widom distribution:
\begin{equation} \label{e:TW}
	\mathbb{P}\left( \frac{\lambda_{\text{max}} - 2\sqrt{n}}{n^{-1/6}} \leq x \right) \to F_{TW}(x),  \quad \text{as } \, n \to \infty \quad \text{for } x\in\R.
\end{equation}
Here \(F_{TW}\) denotes the cumulative distribution function of the Tracy-Widom distribution. For details, we refer to
 \cite{tracy1994level,tracy1996orthogonal}.

\textbf{Additional notation} We will often compare sequences $(a_t)_t, (b_t)_t$ of positive numbers with each other. If there exists a positive constant $C>0$ such that $a_t \le C b_t$ for all $t$, we write $a_t \lesssim b_t$. If the sequences satisfy $a_t \lesssim b_t \lesssim a_t$, we  write $ a_t \asymp b_t$. Importantly, when using these notations in our below asymptotics, the implied constants are always independent of $n$ and $m$ (the objects of our asymptotics).
\subsection{Statistical methodology}\label{ss:sm}
\textbf{A first change point detector} In order to test the hypotheses pair $\mathsf{H}_0 - \mathsf{H}_1$, we construct a change point detector. As a first step, we define the largest eigenvalue of the deformed Wigner matrix $\mathbf{M}_t$ at time $t$ by $\lambda_t := \lambda(\mathbf{M}_t)$. Under the null hypothesis of no detectable signal, all eigenvalues $\lambda_t$ are concentrated around $ 2 \sigma$
because for a fixed $t$, $\lambda(\mathbf{M}_t)\convP 2\sg$, as $n\to \infty$,
as explained in the introduction.
Indeed, a more precise analysis  \citep{Knowles2017, Lee2015} shows that
\begin{align*}
	\frac{n^{2/3}}{\sigma}(\lambda_t-2\sigma)
	\convd \zeta, \quad n\to\infty,
\end{align*}
where $\zeta$ has the Tracy-Widom distribution defined in \eqref{e:TW}.
Notice that the Tracy-Widom distribution does not contain any  parameters depending on the distribution of $\bfW_t$ or $s_t$ - it is universal. If, however, the null hypothesis is false and $s_t \sim \PR^{(2)}$, the sample eigenvalue $\lambda_t$ becomes detached from the bulk, and it follows that
$$
	\frac{n^{2/3}}{\sigma}(\lambda_t-2\sigma)
	\overset{P}{\to} \infty.
$$
In the light of these convergence behaviors, the statistic $(n^{2/3}/\sigma)(\lambda_t-2\sigma) $ could be used to test for the presence of a signal at time $t$ -- supposing that the standard deviation $\sigma$ were known. We are interested in monitoring over many time points $t$,  and hence propose to aggregate all statistics $n^{2/3}(\lambda_t-2\sigma) $ available at some time $k$ in the monitoring period as
\begin{align}
	\textcolor{black}{D_m}(k):=\frac{\sqrt{m}}{m+k}
	\bigg(\sum_{t=m+1}^{m+k}n^{2/3}(\lambda_t-2\sigma)- \frac{k}{m}\sum_{t=1}^mn^{2/3}(\lambda_t-2\sigma)\bigg).
\end{align}
This aggregation method is motivated by the change point detector presented in \cite{horvath:kokoszka:huskova:steinebach:2003}
for multivariate time series. Intuitively, under the null hypothesis, 
\begin{align} \label{e:well}
	\textcolor{black}{D_m}(k)\approx  \frac{\sigma \sqrt{m}}{m+k}
	\bigg(\sum_{t=m+1}^{m+k} \zeta_t- \frac{k}{m}\sum_{t=1}^m \zeta_t\bigg),
\end{align}
for a sequence of Tracy-Widom variables $(\zeta_t)_t$,
and the right side converges weakly for $m \to \infty$ under
some mild dependence assumptions. Conversely, under $\mathsf{H}_1$,
it is easy to see that for suitable $k>k^*$,
$
\textcolor{black}{D_m}(k) \overset{P}{\to} \infty,
$
as the first sum in the definition of $\textcolor{black}{D_m}(k)$ blows up. These properties make $\textcolor{black}{D_m}(k)$ a potential candidate for a change point detector. Indeed $\textcolor{black}{D_m}(k) $ is even practically calculable, since \textcolor{black}{the parameter $\sigma$ in its definition cancels out, yielding the statistic}
\begin{align} \label{e:def:Dk}
	    \textcolor{black}{D_m}(k)=\frac{n^{2/3}\sqrt{m}}{ \,\,m+k}
	\bigg(\sum_{t=m+1}^{m+k}\lambda_t-\frac{k}{m}\sum_{t=1}^m\lambda_t\bigg).
\end{align}
However, both from a practical and a theoretical perspective,
$ \textcolor{black}{D_m}(k)$ cannot be used directly for a statistical test.
From a practical perspective, its limiting distribution is still influenced by
the unknown parameter $\sigma$ as well as (and more problematically) the time dependence of the sequence $(\mathbf{W}_t, \bfx_t,s_t)_t$. The latter introduces a difficult to estimate long-run variance parameter into the problem.
From a theoretical perspective, the analysis is challenging
because the approximation of  $D_m(k)$ by its well-behaved limit in
\eqref{e:well} is unclear,
and we can only show it for $n \to \infty$ with $m,k$ fixed.
Indeed, given current theory, it is not even known whether the moments
of $(n^{2/3}/\sigma)(\lambda_t-2\sigma) $ converge to those of the
Tracy-Widom distribution, making an asymptotic analysis where simultaneously $n, m \to \infty$ difficult. Fortunately, both problems can be addressed by the same  tool: self-normalization.

\textbf{A self-normalized detector}
In the context of time series analysis, self-normalization refers to the division of a statistic of interest (in our case the detector $ D_m$) by a "normalizer" that cancels out the long-run variance and hence makes the normalized statistic asymptotically pivotal. It is important to emphasize that self-normalization is not equivalent to long-run variance estimation, since the normalizer does not converge to a fixed number, but to a non-degenerate random variable that is proportional to the root of the long-run variance. Recently, self-normalization has been used to cancel out the long-run variance in sequential testing by \cite{chan:ng:yau:2021}.
Our application of the principle of self-normalization is novel because
we cannot even prove the existence of a long-run variance as $n,m \to \infty$. What we can show, however,
is that the self-normalization cancels out the finite sample variance, that always exists,
and scales the process automatically so that it can converge.
Such approximation properties in finite samples have not been studied
for self-normalized statistics before, but are known for
bootstrap, see, e.g., \cite{chen:kato:2020}.
For reviews of  self-normalization, we refer  to \cite{shao:2010,shao:2015}. We now define a normalizer, drawing on the initial training sample as follows
\begin{align} \label{e:def:Vm}
	V_m:= \frac{n^{2/3}}{m^{3/2}}\sum_{t=1}^m \bigg|\sum_{s=1}^t \lambda_s - \frac{t}{m}\sum_{s=1}^m \lambda_s \bigg|,
\end{align}
and therewith define the self-normalized detector
\begin{align} \label{e:def:gam}
	\Gamma_{n,m}(k):= \frac{D_m(k)}{V_m}.
\end{align}

\textbf{Test statistic} We are now in a position to formulate a statistical test
for the hypotheses pair $\mathsf{H}_0-\mathsf{H}_1$. We show  that under $\mathsf{H}_0$ and adequate technical assumptions (formulated in the next section),
\begin{align} \label{e:weakL}
	\sup_{1 \le k < \infty} \Gamma_{n,m}(k)
	\overset{d}{\to}\sup_{0 \le x < \infty} \frac{[B(1+x)-B(1)]-xB(1)}{(1+x)
		\cdot \int_0^1 |B(s)-s B(1)|ds},
	\quad m,n\to\infty,
\end{align}
where $B$ is the standard Brownian motion. Consequently, denoting by $q_{1-\alpha}$ the upper $\alpha$-quantile of the limiting distribution in \eqref{e:weakL}, we reject the hypothesis $\mathsf{H}_0$, if for some  $k\ge 1$,
\begin{align}\label{eq_test}
	\Gamma_{n,m}(k) >q_{1-\alpha}.
\end{align}
Denoting by $\hat{k}$ the smallest $k$ for which \eqref{eq_test} holds,
if $\mathsf{H}_0$ is rejected, we declare the emergence of a detectable
signal at time $m+ \hat{k}$.

\subsection{Theoretical analysis}
This section leads to Theorem \ref{theomain} that states, under appropriate
assumptions, the weak convergence  \eqref{e:weakL}. In a first step, we study the central building blocks of our statistic
- the maximum eigenvalues $\lambda_t$.
We leverage results from random matrix theory to
derive eigenvalue concentration which we use in our
proofs to truncate the eigenvalues. In the second step, we study
the partial sum process of the maximum eigenvalues
(the process $P_m$ defined in \eqref{e:def:Pm} below).
A central result is the Gaussian approximation in Theorem \ref{lem_3},
where we demonstrate that $P_m$ can be coupled to a Brownian motion.

\textbf{Concentration of eigenvalues}
As noted in the Introduction,
the largest eigenvalue of a deformed Wigner matrix undergoes a
phase transition if a signal becomes detectable.
To describe the results relevant to our monitoring problem,
we need the following assumption on the triple
$(\mathbf{W}, \bfx, s)$ in \eqref{eq_def_M}.

\begin{assumption}\label{ass_wigner}
	Suppose that $\bfW$ is a Wigner matrix independent of $(s,\bfx)$
	and that $\bfx$ is an $n$-dimensional random vector of length $1$.
\end{assumption}

We now state two lemmas
that help understand our approach; the proofs are deferred to Section \ref{ss:proof_rmt}.
The first lemma demonstrates that in the subcritical regime, the largest eigenvalue is highly concentrated around $2\sigma$.
The result follows from a combination of \cite{Knowles2017, Lee2015}. We also note that the well-known work \cite{knowles:yin:2013:a} contains a similar result but under the stronger assumption of uniform subexponential decay.

\begin{lemma} \label{lem_rmt}
	Suppose that the triple $(\mathbf{W}, \bfx, s)$ satisfies Assumption \ref{ass_wigner} and that $s \sim \PR^{(1)}$.
    For every $\varepsilon>0$ (sufficiently small) and $D>0$ (sufficiently large), there exists an integer $N=N(\varepsilon, D)$ such that
		\begin{align*}
			\PR ( | \lambda - 2 \sigma | > n^{\varepsilon-2/3} ) \leq n^{-D}
			\quad \forall \ n\geq N.
            \end{align*}
\end{lemma}
Next, we turn to the analysis in the supercritical case, which corresponds
to the alternative hypothesis.
In this regime, the largest eigenvalue  separates from the bulk.
The following result is a slight generalization of Theorem 1.1 in
\cite{pizzo_et_al_2013}, who consider the case where $s$ is deterministic.

\begin{lemma} \label{lem:pizzo}(Location of the largest eigenvalue in the supercritical case)
	Suppose that Assumption \ref{ass_wigner} holds and that $s \sim \mathbb{P}^{(2)}$. Then,
	\begin{align*}
		\lambda - \Big(s + \frac{\sigma^2}{s} \Big)\conp 0, \quad n\to\infty.
	\end{align*}
\end{lemma}

\textbf{Eigenvalue truncation} A key technique in our proofs is
the truncation of the eigenvalues to a region close to $2\sigma$.
More precisely, we define for a constant $\varepsilon>0$ the  event
\begin{align} \label{eq_def_Lambda_t}
	\Lambda_t := \{| \lambda_t-2\sigma|<n^{\varepsilon-2/3}\}.
\end{align}
In view of Lemma \ref{lem_rmt},  $\Lambda_t$  occurs  with high probability
for any $t$ in the subcritical regime. The theoretical constant $\varepsilon>0$
needs to be chosen appropriately in the proofs of our  convergence results.
We emphasize that while our theory involves the constant  $\varepsilon$, it does not have to be chosen in practice because the self-normalization asymptotically eliminates all parameters depending on $\varepsilon$.
Next, we specify the relation of the training period length $m$ and the dimension $n$ of the Wigner matrix.

\begin{assumption} \label{ass_2}
	The dimension   $n=n_m$ is a function of the training period length $m$
	and there is  a constant $\theta \in (0,\infty)$  such that $m \asymp n^\theta$.
\end{assumption}

\textcolor{black}{
\textbf{Change point model \& Gaussian approximation} Using the  concentration results for individual eigenvalues, we now focus on monitoring the stream of eigenvalues of deformed Wigner matrices.
To formalize the behavior of the time series pre- and post-change, we define the following extended model: For every $n$, there exists an infinite time series $(\mathbf{W}_{t}, \mathbf{x}_{t}^{(1)},s_t^{(1)}, \mathbf{x}_{t}^{(2)},s_t^{(2)})_{t \in \mathbb{Z}}$, and we define the deformed Wigner matrices for $j=1,2$ and the corresponding leading eigenvalues as
\begin{align} \label{e:model:cases}
\mathbf{M}_t^{(j)}:= s_t^{(j)}(\mathbf{x}_{t}^{(j)})(\mathbf{x}_{t}^{(j)})^{\top} +\mathbf{W}_{t}/\sqrt{n}, \quad \lambda_t^{(j)}:=\lambda(\mathbf{M}_t^{(j)}).
\end{align}
Here $j=1$ corresponds to the pre-change and $j=2$ to the post-change distribution, and the actually observed data are, as before, denoted by
\begin{align} \label{e:cases:j}
\mathbf{M}_t=\mathbf{M}_t^{(j)}, \quad \lambda_t = \lambda_t^{(j)},  \quad\textnormal{where} \quad j=\begin{cases}
    1, \quad 1 \le t \le m+k^*\\
    2, \quad \, t> m+k^*.
    \end{cases}
\end{align}
 Recall that $k^*$ denotes the change point location with $k^*=\infty$ meaning that no change occurs. A main difficulty is the analysis of the sequential test statistic under the null hypothesis.} We need to define some key objects, beginning with the partial sum process of eigenvalues, i.e.
\begin{align} \label{e:def:Pm}
	P_m(x):=\frac{1}{\sqrt{m \tau_n}}\sum_{t=1}^{\lfloor mx\rfloor}n^{2/3} [\lambda_t^{(1)}- b_{(n)}], \qquad x \in [0,T_m].
\end{align}
Here, $(T_m)_{m \in \mathbb{N}}$ is a sequence of positive numbers such that
$T_m \to \infty$,  as $m \to \infty$.  On the growing interval of length $T_m$,
we will approximate $P_m$ by a Gaussian process.
The constants $\tau_n$, $b_{(n)}$ are standardizing sequences
defined as follows. First, $b_{(n)}$ is the conditional mean
\begin{align} \label{eq_def_b_n}
	b_{(n)} := \mathbb{E}[\lambda_1^{(1)}|\Lambda_1^{(1)}],
\end{align}
where the event $\Lambda_1^{(1)}$ is defined as in  \eqref{eq_def_Lambda_t}. 
Second, setting $Y_t:=\textcolor{black}{n^{2/3}} \mathbb{I}\{\Lambda_t^{(1)}\}(\lambda_t^{(1)}-b_{(n)})$,  we define $\tau_n$ as the finite sample variance
\begin{align} \label{e:def:tau_n}
	\tau_{n}:=\mathbb{E}\bigg(\frac{1}{\sqrt{\lfloor mT_m\rfloor }}\sum_{t=1}^{\lfloor mT_m\rfloor } Y_t \bigg)^2.
\end{align}
By the construction of the $b_{(n)}$,  the variables $Y_t$ are  centered under the null hypothesis.
For any large enough $n$, the variance $\tau_{n}$ can be shown
to be close to the long-run variance
\begin{align}\label{e:def:tau^n}
	\tau^{(n)}:= \sum_{t \in \mathbb{Z}}\mathbb{E}[Y_t Y_0],
\end{align}
which is basically the long-run variance for fixed $n$ and $m \to \infty$.
Notice that  the objects $b_{(n)}, Y_t, \tau_n, \tau^{(n)} $
all depend on the events $\Lambda_t^{(1)}$, which in turn involve a truncation variable $\varepsilon>0$, which will be part of the following assumption.

\begin{assumption} 
The following conditions hold: \label{ass_3}
	\begin{itemize}
		\item[i)] \textcolor{black}{ For any  $(n,t,j)$, the triple $(\mathbf{W}_{t}, \mathbf{x}^{(j)}_{t},s_t^{(j)})$ satisfies Assumption \ref{ass_wigner}.}
		\item[ii)]\textcolor{black}{ For any $n$, the sequence $(\mathbf{W}_{t}, \mathbf{x}_{t}^{(1)},s_t^{(1)}, \mathbf{x}_{t}^{(2)},s_t^{(2)})_{t \in \mathbb{Z}}$ is strictly stationary and $\phi$-mixing with coefficients
		\[
		\phi(r) \le C a^{r},
		\]
		for all $r\ge 1$
		for some fixed constants $C>0$  and $a \in (0,1)$.}
		\item[iii)] There exist constants $\varepsilon_\tau, c_\tau>0$ such that
		\[
		\inf_{0<\varepsilon \le \varepsilon_\tau}\liminf_{n \to \infty} \tau^{(n)}>c_\tau.
		\]
	\end{itemize}
\end{assumption}

The first part of the assumption states that the data arriving at each time point follow the structure of a deformed Wigner matrix.
Part ii) guarantees that the time series under investigation is weakly dependent. In particular, it implies that $\tau^{(n)}<\infty$ for any $n$ and any $\varepsilon>0$. Part iii) of the assumption guarantees the non-degeneracy of the long-run variance for any small $\varepsilon>0$.

We now state a
Gaussian approximation for $P_m$ on the growing interval $[0,T_m]$,
which is a key theoretical result needed to establish the validity
of our method. The proof is presented in Section \ref{ss:p-Gauss}.

\begin{theorem}[Gaussian approximation] \label{lem_3}
	Suppose that Assumptions \ref{ass_2} and \ref{ass_3} hold and that
	$\mathsf{H}_0$ is true. Furthermore, suppose that
	$T_m\asymp m^{1/2-\rho}$ for some $\rho\in(1/6,1/2)$.
	Then, on a sufficiently rich probability space dependent on $m$,
	we can define versions of $P_m$ and a standard Brownian motion $B_m$
	such that
	\[
	\mathbb{P}\Big(
	\sup_{x\in[0,T_m]}|P_m(x)-B_m(x)|>d_m
	\Big)<d_m,
	\]
	where $(d_m)_{m\in\mathbb N}$ is a sequence of positive numbers
	converging to zero.
\end{theorem}

The Gaussian approximation can be used to approximate the denominator $V_m$ in the self-normalized statistic by a transformation of a Brownian motion
(see its definition in \eqref{e:def:Vm}) and similarly  for the numerator $D_m(k)$ for values of $k \le m T_m$.  To control the statistic for larger values of $k$, the following tail bound is necessary.

\begin{lemma}[Tail bound] \label{lem_4}
	Suppose that Assumptions \ref{ass_wigner}, \ref{ass_2} and \ref{ass_3}
	hold and that $\mathsf{H}_0$ is true. Then, it follows for any $\zeta>0$ that
	\[
	\sup_{m^{1+\zeta}< k<\infty } \Gamma_{n,m}(k)=\frac{-P_m(1)}{V_m/\sqrt{\tau_n}}+o_P(1).
	\]
	The $o_P(1)$ term is a negligible remainder,  as $m, n=n_m \to \infty$.
\end{lemma}
The proof given in Section \ref{ss:p-l3} is challenging and interesting.
For values of $k \le m T_m$ in the statistic $\Gamma_{n,m}(k)$ we use the Gaussian approximation from Theorem \ref{lem_3}. Now, for larger $k$ we have no Gaussian approximation, but we can still truncate the eigenvalues $\lambda_t$ with high probability using the event $\Lambda_t$
(see \eqref{eq_def_Lambda_t}) for $m T_m < k \le m^D$ and some large constant $D$.
For these $k$, we can then use bounds for sums of bounded, mixing random variables
(even though the bound depends on $n$). Finally, for values $k>m^D$, even these tight,
high-probability truncations are not available anymore; there is an infinite number
of these values and we cannot truncate all of them. So, for these very large $k$,
we use the fact that our statistic involves a temporal discounting
(the factor $1/(m+k)$ in the definition of
$D_m$;  \eqref{e:def:Dk}). Since $k$ is so large,
it effectively vanquishes the fluctuations of the remaining eigenvalues.
In the proof presented in Section \ref{ss:p-main},
we combine Theorem \ref{lem_3} and Lemma \ref{lem_4} to obtain
our main result.

\begin{theorem} \label{theomain}
	If Assumptions \ref{ass_wigner}, \ref{ass_2} and \ref{ass_3}
	hold and  $\mathsf{H}_0$ is true, then  convergence \eqref{e:weakL}
	holds.
\end{theorem}

Corollary \ref{cor_test} collects the theoretical guarantees for our monitoring
procedure under both the null and alternative hypothesis. As an immediate
consequence of Theorem \ref{theomain},
test decision  \eqref{eq_test} has  an asymptotically controlled
level under $\mathsf{H}_0.$ Moreover, we provide consistency under
$\mathsf{H}_1$ in  \eqref{e:alt}. To this end, we assume that a change
occurs at an unknown time point $m+k^*$ and we allow $k^*$ to
depend on $m$. More precisely, suppose that
\begin{align} \label{e:polyk}
	k^* \asymp m^D \ \ {\rm for \ some } \ D\ge 0.
\end{align}

\begin{corollary} \label{cor_test}
Suppose that Assumptions \ref{ass_2} and \ref{ass_3} hold.
	\begin{enumerate}
		\item[(a)] 
        Under $\mathsf{H}_0$, we have
		\begin{align*}
			\lim_{m\to\infty}  \PR \left ( \sup_{1 \leq k < \infty}
			\Gamma_{n,m}(k) >q_{1-\alpha} \right ) = \alpha .
		\end{align*}
		\item[(b)] 
        Under $\mathsf{H}_1$ with \eqref{e:polyk}, we have
		\begin{align*}
			\lim_{m\to\infty}  \PR \left ( \sup_{1 \leq k < \infty}
			\Gamma_{n,m}(k) >q_{1-\alpha} \right) =1.
		\end{align*}
	\end{enumerate}
\end{corollary}
The proof of Corollary \ref{cor_test} is given in Section \ref{ss:p-cor}.

\section{Finite sample performance and applications}\label{s:fs}
In this section, we first present the results of
a small simulation study aimed at assessing the performance
of our method in finite samples. Then,
we show how it can be used to detect a sudden change in
particulate pollution or the emergence of a signal in a social group
structure.
We want to show that the new theory developed in this paper
is likely to have practical impact. The code for running the experiments is available at {\tt{https://github.com/adelialee-source/npts}}.

\subsection{A small simulation study}\label{ss:sim}
Recall that the monitoring is based on the upper
$\alpha$th quantile, $q_{1-\alpha}$,  in \eqref{eq_test}.
The change point is declared  at the smallest
$k$ for which $\Gamma_{n,m}(k)>q_{1-\alpha}$.
For this study, we consider the significance levels
$\alpha = 0.05$ and $\alpha = 0.10$.

\textbf{Computation of the critical values}
To compute the critical value $q_{1-\alpha}$ in \eqref{eq_test},
we must simulate the limit in \eqref{e:weakL}. Since one cannot
simulate the Brownian motion on the infinite half-line $(0, \infty)$,
some approximations are needed. In the spirit of the Gaussian
approximation, the most direct approach is  to simulate
the limit distribution by replacing the eigenvalues in \eqref{e:def:Vm}
and \eqref{e:def:gam} by standard normal random variables.
The self-normalization
automatically eliminates the need for any scale adjustment. The limit
can then be simulated for any $m$ and $T$ relevant to a specific application.
As we will see, the quantiles are practically the same if $m$ and $T$
are sufficiently large.

Specifically, setting $W(l) = \sum_{j=1}^l N_j$, where the $N_j$ are
iid $N(0,1)$, we put
\begin{equation} \label{e:LmT}
	L(m,T) = \max_{k = 1, 2, \dots, T} \frac{m^2}{m+k}
	\frac{W(m+k)-\frac{m+k}{m}W(m)}
	{\sum_{t = 1}^m \left| W(t)
		- \frac{t}{m}W(m) \right|}.
\end{equation}
The distribution of $L(m,T)$ is an approximation to the distribution
of the limit in \eqref{e:weakL}. Figure \ref{fig:q} shows the
empirical distribution of $L(m,T)$ for $T=500$ and several values of $m$.
The critical values $q_{0.95}$ and $q_{0.90}$ are the same
for $m=200, 300, 400, 500$ up to two decimal places. In the
following, we thus use
\[
q_{0.90} = 4.57, \ \ \ q_{0.95} = 5.85.
\]

\begin{figure}[H]
	\includegraphics[width=0.98\textwidth]{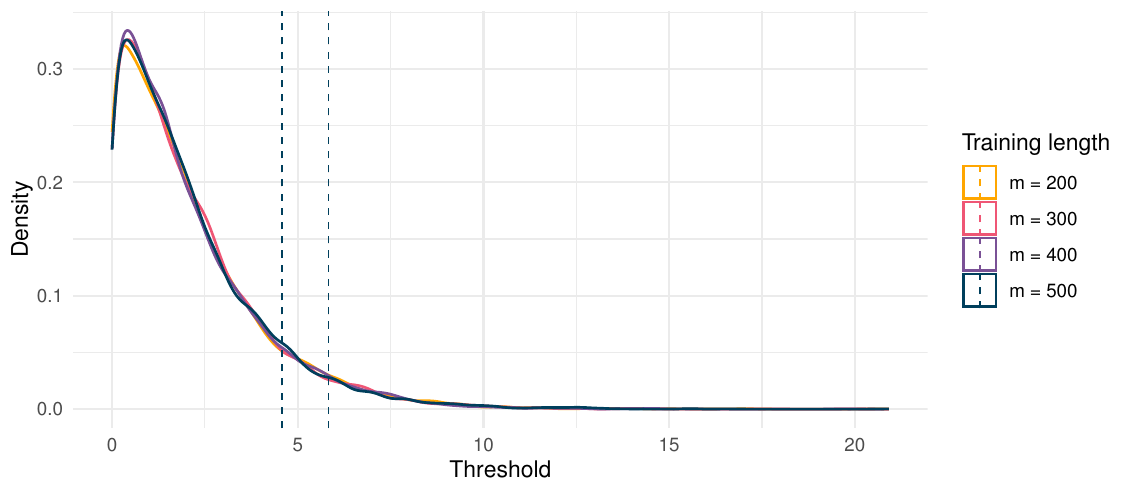}
	\caption{Empirical distribution of $L(m,T)$ based on
		10,000 replications for
		$T = 500$  and different training lengths $m$.
		The  vertical lines show the critical values
		$q_{0.90}$ and $q_{0.95}$ for $m=500$. The lines
		for other values of $m$ are visually indistinguishable.}
	\label{fig:q}
\end{figure}

\textbf{Data generating process}
Recall that the matrix observations $\mathbf{M}_t$
are assumed to follow model \eqref{eq_def_M}.
The signal direction $\bfx_t$ is constructed as follows. We first generate
a random vector $\mathbf{y}_t\in \R^n$ whose components are independent and
follow the standard normal distribution.
We then normalize $\mathbf{y}_t$ to the unit norm to get $\bfx_t$.
The noise matrices $\bfW_t$ form  a sequence of dependent
Wigner matrices. Each component is an autoregressive process of order
one with \rev{its own} autoregressive parameter. 
\rev{The parameters $\phi_{ij}$, $i\le j$, are drawn independently from the
uniform distribution on $(0,0.9)$, and we set $\phi_{ji}=\phi_{ij}$, so that
the resulting $n\times n$ matrix $\bPhi$ is symmetric. It is drawn once at the
beginning of each replication and held fixed over $t$.}
Each component of the $n\times n$
matrix $\bfW_t$ thus follows an autoregressive process with a different
parameter. Using the Hadamard (entrywise) product $\odot$,
our data generating process
can be written down as

\rev{
\[
\bfW_t
=
\bPhi \odot \bfW_{t-1}
+
\sqrt{\mathbf{1}-\bPhi\odot\bPhi}\odot\beg_t.
\]
Since $\bPhi$ is symmetric, so is each $\bfW_t$.}
The square root is applied entrywise to ensure that each  entry has unit variance.
{\em The critical threshold is thus $\sg =1$.}
The initial matrix $\bfW_1$ and the error matrices $\beg_t$ are Wigner matrices with independent upper triangular
and diagonal elements generated from the standard normal distribution.

\begin{table}[h]\centering
		\caption{Proportion of False Alarms for different training lengths $m$ and dimensions $n$. \rev{For 2,000 replications, the precision of the entries is at most $\pm 0.011$.}}
		\label{tb:type1}
        \begin{tabular}{@{}llrrrrr@{}}
			\hline
			\toprule[1.2pt]
			
			&& \multicolumn{2}{c}{Uniform(0,1)} &
			& \multicolumn{2}{c}{Beta(2,4)}\\
			
			\cmidrule{3-4} \cmidrule{6-7}
			
			$m$ & $n$
			& $\alpha=0.05$ & $\alpha=0.10$ &
			& $\alpha=0.05$ & $\alpha=0.10$ \\
			
			\midrule[1.2pt]
			
			\multirow{4}{*}{\textcolor{black}{300}}
			& \textcolor{black}{10} & \textcolor{black}{0.067} & \textcolor{black}{0.128} && \textcolor{black}{0.073} & \textcolor{black}{0.140} \\
			& \textcolor{black}{25} & \textcolor{black}{0.071} & \textcolor{black}{0.138} && \textcolor{black}{0.070} & \textcolor{black}{0.126} \\
			& \textcolor{black}{50} & \textcolor{black}{0.068} & \textcolor{black}{0.127} && \textcolor{black}{0.072} & \textcolor{black}{0.128} \\
			& \textcolor{black}{100} & \textcolor{black}{0.070} & \textcolor{black}{0.124} && \textcolor{black}{0.066} & \textcolor{black}{0.127} \\
			\addlinespace
			\multirow{4}{*}{\textcolor{black}{350}}
			& \textcolor{black}{10} & \textcolor{black}{0.072} & \textcolor{black}{0.118} && \textcolor{black}{0.064} & \textcolor{black}{0.118} \\
			& \textcolor{black}{25} & \textcolor{black}{0.046} & \textcolor{black}{0.108} && \textcolor{black}{0.067} & \textcolor{black}{0.116} \\
			& \textcolor{black}{50} & \textcolor{black}{0.050} & \textcolor{black}{0.114} && \textcolor{black}{0.061} & \textcolor{black}{0.115} \\
			& \textcolor{black}{100} & \textcolor{black}{0.069} & \textcolor{black}{0.119} && \textcolor{black}{0.053} & \textcolor{black}{0.117} \\
			\addlinespace
			\multirow{4}{*}{\textcolor{black}{400}}
			& \textcolor{black}{10} & \textcolor{black}{0.064} & \textcolor{black}{0.115} && \textcolor{black}{0.064} & \textcolor{black}{0.112} \\
			& \textcolor{black}{25} & \textcolor{black}{0.064} & \textcolor{black}{0.118} && \textcolor{black}{0.066} & \textcolor{black}{0.123} \\
			& \textcolor{black}{50} & \textcolor{black}{0.059} & \textcolor{black}{0.120} && \textcolor{black}{0.062} & \textcolor{black}{0.106} \\
			& \textcolor{black}{100} & \textcolor{black}{0.059} & \textcolor{black}{0.120} && \textcolor{black}{0.063} & \textcolor{black}{0.112} \\
			\addlinespace
			\multirow{4}{*}{\textcolor{black}{450}}
			& \textcolor{black}{10} & \textcolor{black}{0.056} & \textcolor{black}{0.118} && \textcolor{black}{0.052} & \textcolor{black}{0.100} \\
			& \textcolor{black}{25} & \textcolor{black}{0.052} & \textcolor{black}{0.098} && \textcolor{black}{0.059} & \textcolor{black}{0.108} \\
			& \textcolor{black}{50} & \textcolor{black}{0.059} & \textcolor{black}{0.115} && \textcolor{black}{0.060} & \textcolor{black}{0.112} \\
			& \textcolor{black}{100} & \textcolor{black}{0.062} & \textcolor{black}{0.116} && \textcolor{black}{0.059} & \textcolor{black}{0.121} \\
			\addlinespace
			\multirow{4}{*}{\textcolor{black}{500}}
			& \textcolor{black}{10} & \textcolor{black}{0.053} & \textcolor{black}{0.109} && \textcolor{black}{0.044} & \textcolor{black}{0.092} \\
			& \textcolor{black}{25} & \textcolor{black}{0.045} & \textcolor{black}{0.102} && \textcolor{black}{0.046} & \textcolor{black}{0.103} \\
			& \textcolor{black}{50} & \textcolor{black}{0.042} & \textcolor{black}{0.096} && \textcolor{black}{0.058} & \textcolor{black}{0.107} \\
			& \textcolor{black}{100} & \textcolor{black}{0.050} & \textcolor{black}{0.100} && \textcolor{black}{0.043} & \textcolor{black}{0.101} \\
			\bottomrule[1.2pt]
		\hline	
		\end{tabular}
\end{table}

\textbf{Empirical size and power}
To assess the empirical false
alarm rate, we need to specify the distribution  $\mathbb{P}^{(1)}$. %
Based on 2000 replications,
Table \ref{tb:type1} exhibits the empirical size, typically called
the Proportion of False Alarms (PFA) in the monitoring context.
Most entries are close to the nominal
sizes, and there appears to be no systematic dependence on the
dimension $n$ or the length of the training period $m$. 
\rev{We first specify $\mathbb{P}^{(1)}$ through $s_t = 0.95 Z_t$, 
where $Z_t$ is either Uniform$(0,1)$ or Beta$(2,4)$, both supported on $[0,1]$. 
To assess the power we must also specify $\mathbb{P}^{(2)}$. 
Suppose the transition occurs at the time point $t = m+k^*+1$, 
and define
\[
s_t=
\begin{cases}
0.95Z_t, & t\leq m+k^*,\\
1.05+\dg Z_t, & t>m+k^*,
\end{cases}
\]
so that $\mathbb{P}^{(2)}$ is the same distribution as $\mathbb{P}^{(1)}$, 
shifted and rescaled. 
The pre- and post-change supports are $[0,0.95]$ and
$[1.05,1.05+\dg]$, 
each separated from the critical value $\sigma=1$ by $0.05$.
Because the post-change signal stays at or above $1.05$ for every $\dg \ge 0$, 
a phase transition is present even in the limit $\dg \to 0$, 
and the power therefore does not tend to the size.
This is reflected in Figure~\ref{fig:power},
which presents power curves for the Uniform and Beta distributions and dimensions $n=25, 50$.
The two choices of $\mathbb{P}^{(1)}$ give nearly indistinguishable curves, 
so the procedure is not sensitive to the particular
distribution on $[0,0.95]$. 
The dimension $n=25$ is used in the data example in
Section~\ref{ss:pp}. 
We have presented the curves for fairly large $k^*$. If
$k^* \le 300$, the power is 1 (based on 2000 replications) even for small $\dg$.}

\begin{figure}[h]
  \centering
  \includegraphics[width=0.9\textwidth]{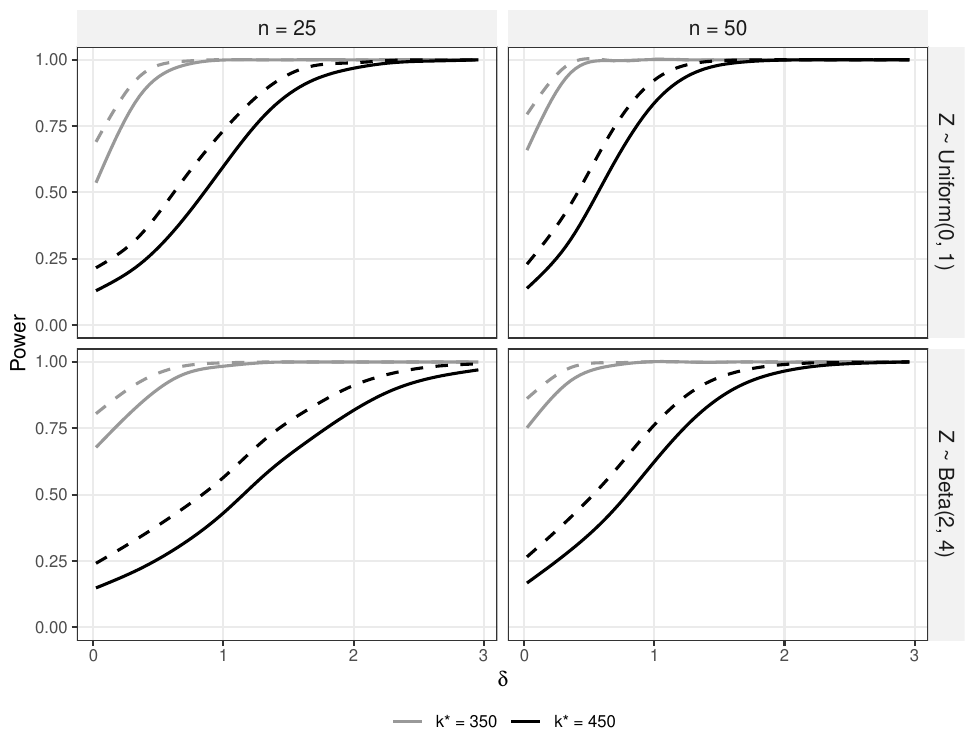}
  \caption{\rev{Empirical power curves for selected distributions and dimensions based on $2000$ replications with $m=400$ and $T=500$.
  Rows correspond to the distribution of $Z_t$, columns to the matrix dimension~$n$. 
  In each panel, the dashed and solid lines represent significance levels of $\alpha = 0.05$
  and $\alpha = 0.10$, respectively.
  }}
  \label{fig:power}
\end{figure}

\subsection{Detection of a transition  in regional particulate pollution levels}
\label{ss:pp}
In this section, we evaluate our monitoring procedure using
data from the United States Environmental Protection Agency.
We use $\text{PM}_{2.5}$ (particulate matter with a diameter
of up to 2.5 micrometers) concentration (in $\mu g/m^3$ LC)
from 2015 to 2020 (6 years).
Our purpose is to illustrate the application of our method
by detecting  the start of the Cameron
Peak Fire in Colorado using only $\text{PM}_{2.5}$ measurements,
which are automatically reported.  Deployment of  aircraft
or land crews to assess a fire is much more expensive.
This wildfire was first reported on August 13, 2020,
and officially contained on December 2, 2020. While it started as a small
local fire, it grew immensely over several months. It became  the first wildfire in
Colorado to exceed 200,000 acres, and the largest wildfire in
Colorado's history. Figure \ref{fig:map} shows its location in the mountains
west of Fort Collins.

\begin{figure}[h]
	\centering
	\includegraphics[width=0.9\textwidth]{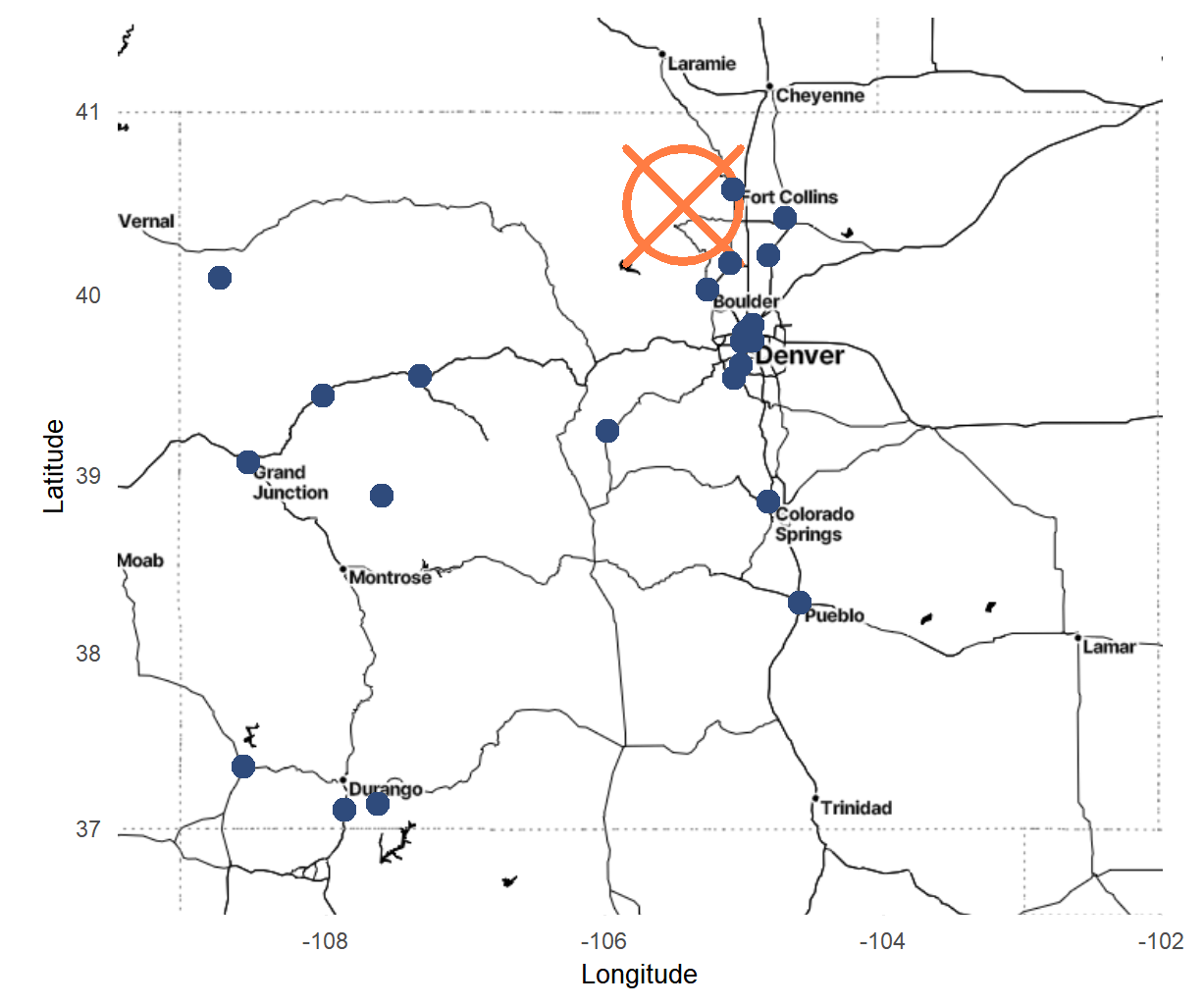}
	\caption{The crossed circle is the location of Cameron Peak Fire.
		The dots show the 25 locations at which $\text{PM}_{2.5}$
		concentration measurement are reported. \label{fig:map}}
\end{figure}

Let $\mathbf{U}_t$ denote the column vector whose components are the
$\text{PM}_{2.5}$ values observed on day $t$ at 25 locations shown in
Figure \ref{fig:map}. Some preprocessing of the raw data was necessary to
obtain the daily values. \rev{It is described in Section \ref{s:fire}
of the Supplementary Material, where links to the code 
and the raw data are also given.}  
Suppose $U_t(s_k)$ is the $\text{PM}_{2.5}$ value on day $t$ at location $s_k$.
The records from 2015 to 2018 are used to estimate the seasonal component.
The $\text{PM}_{2.5}$ concentration exhibits an annual pattern due to
changing temperature,  prevailing wind direction and  strength,  and the
occurrence
of small, seasonal fires. Over the study
area,  $\text{PM}_{2.5}$ concentrations are higher in the
second half of the calendar year.
The raw seasonal component is computed as
\[
S_t(s_k) = \frac{1}{4} \sum_{Y = 2015}^{2018} U_t^{(Y)}(s_k), \
\ t = 1, 2, \dots, 365.
\]
and it is smoothed by a moving average of 30 days.
The seasonal component is only used to compute  the
deseasonalized values $V_t(s_k) = U_t(s_k) - S_t(s_k)$ for
$t$ corresponding to days in 2019 and 2020. For example, if $t$
corresponds to May 15, 2020, then $U_t(s_k)$ is the pollution
on May 15, 2020,  and $S_t(s_k)$ is the value of the seasonal component
on May 15. The monitoring starts on April 29, 2020.
In Northern Colorado, snow storms occur until late April.
The training period is January 1, 2019 to April 28, 2020, $m=483$. 
\begin{figure}[h]
	\centering
	\includegraphics[width=0.9\textwidth]{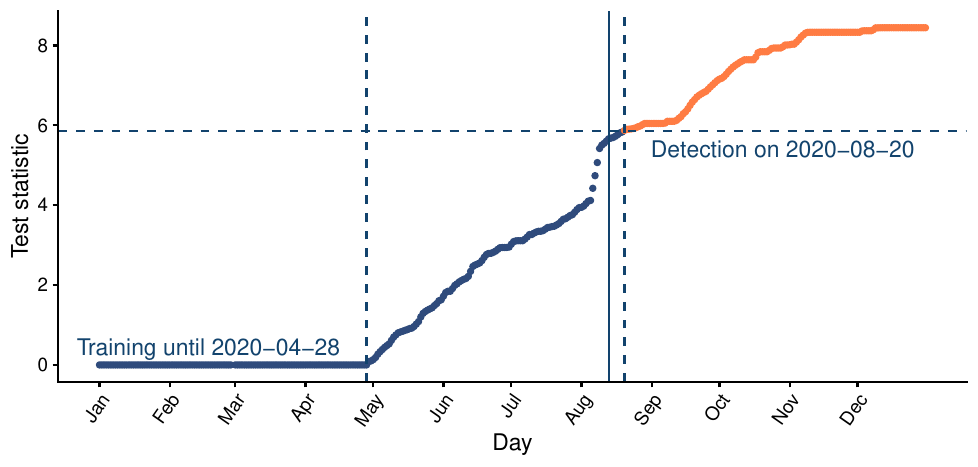}
	\caption{The detector $\max\Gamma_{n,m}(k)$ and the visualization of the detection
		process.  \label{fig:real} The horizontal dashed line corresponds to
		$q_{0.95} = 5.85$. The solid line represents the initial release date of the Cameron Peak Fire in Colorado.}
\end{figure}
\textcolor{black}{The data matrices $\bM_t$ are then given by the outer product of the deseasonalized vector
$\mathbf{V}_t:=(V_t(s_k))_k$, i.e. by $\bM_t = \mathbf{V}_t\mathbf{V}_t^\top$. Notice that these data matrices are of rank one and thus fit into the framework discussed in Section \ref{sec:beyond:wigner} of the Supplementary Material (note that applying the self-normalized detector to $\mathbf{M}_t$ is equivalent to applying it to $\mathbf{M}_t /n$).}
At the significance level  $\alpha = 0.05$, we use  the threshold $q_{0.95} = 5.85$.
Figure \ref{fig:real} shows that based only on the $\text{PM}_{2.5}$ measurements,
the wildfire is detected on August 20, 2020, a few days after the first
visual identification.

\subsection{Monitoring social interactions in primates} \label{sec:43}

As a second application, we discuss the monitoring of social interactions in a group of baboons. The data were gathered by \cite{gelardi:godard:paleressompoulle:claidiere:barrat:2020} (we refer to that paper for a more detailed data description). It is based on wearable sensors that track social interactions between physically close animals. Proximity-based analyses have been used extensively to study human behavior (e.g. using cellphone data) and are typically referred to as \textit{reality mining} \cite{eagle:pentland:2006}.
Applications to animals are quite new, but promise important insights into behavioral patterns without interfering with the animals' environment - a problem plaguing traditional studies where more intense human monitoring introduces an increasing sampling bias \citep{davis:crofoot:farine:2018}. %
Structural changes in primate interactions are important to understand because they have been associated with damage to the habitat or the introduction of pathogens (both often unintentional; \cite{capitanio:2012}).

We analyze the data of 13 baboons equipped with wearable sensors,
collected between 13 June and 10 July 2019. The sensors report
on interactions between two individuals if they are close to each other, with a recording resolution of $20$s. We aggregate the total duration of interactions between individuals over $30$-minute periods into a symmetric matrix. On the diagonal, we report total interaction time.  \textcolor{black}{Notice that the matrix entries are, by construction, not independent -- for instance, time allocated to interactions with one specific individual cannot be independent of total interaction time. In this sense, there exists a gap between this application and the Wigner model discussed in the main part of this paper. As we discussed in Section \ref{sec:beyond:wigner} of the Supplementary Material, some other matrix models are still compatible with our theory under certain concentration assumptions for the eigenvalues. Even though it is not a priori clear whether such concentration should hold for the matrices of primate interactions, it is at least plausible. So, with this important technical caveat, we will apply our new method in this context.}
Most social interactions are relatively short, lasting less than a minute, making interactions that overlap aggregation periods negligible. We used the first $4$ days of data, $192$ time points, to center the remaining matrices. %

\begin{figure}[h!]
    \centering
    \subfloat[Original data (no signal).]{%
        \includegraphics[width=0.49\textwidth]{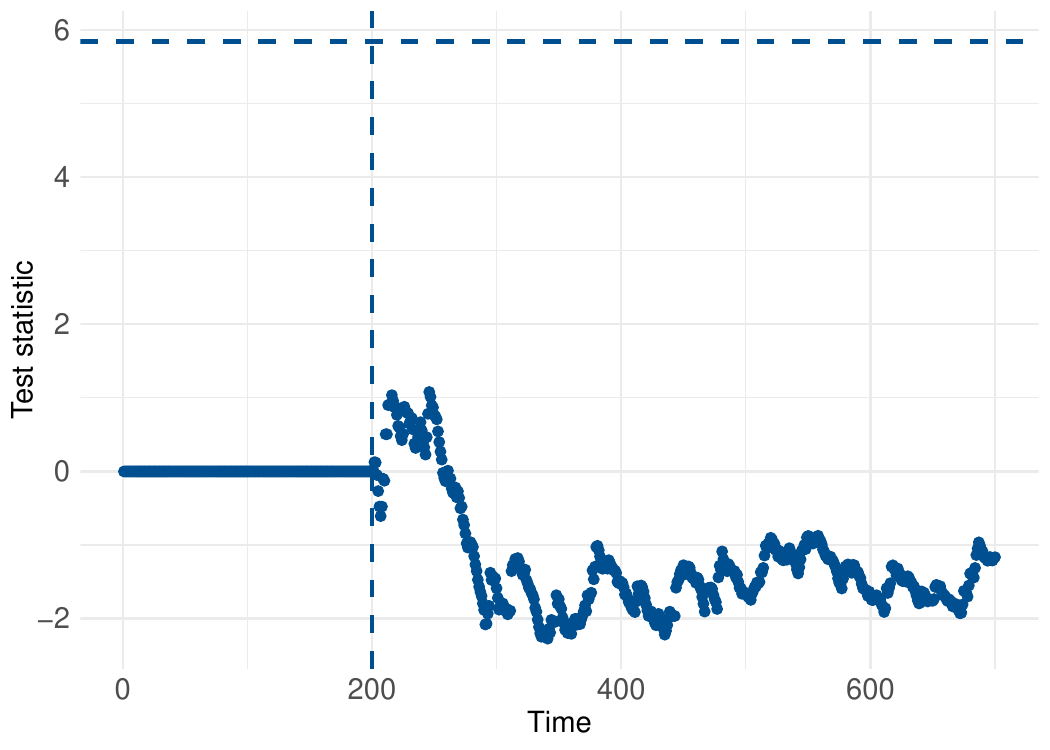}%
    }\hfill
    \subfloat[Data with a detectable signal added.]{%
        \includegraphics[width=0.49\textwidth]{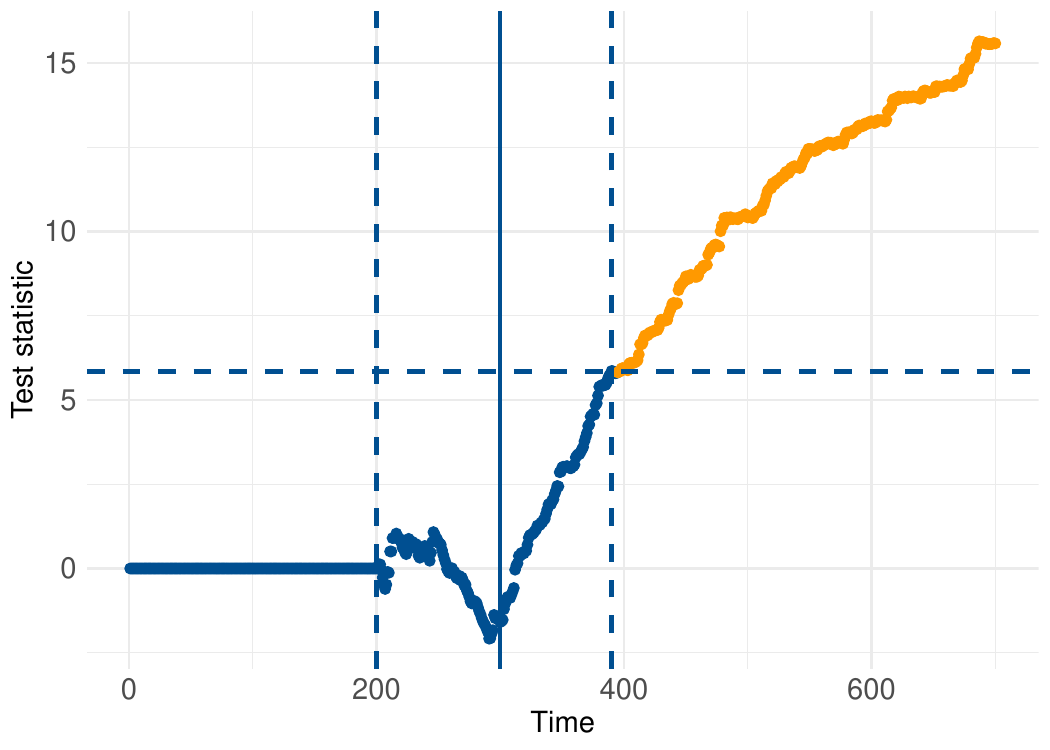}%
    }
    \caption{The detector $\Gamma_{n,m}(k)$ and the visualization of the detection
        process for the original data without a signal (left) and for an
        artificially included signal (right). The horizontal dashed line corresponds to
        $q_{0.95} = 5.85$. The solid line represents the time of signal insertion.
        The dashed vertical lines represent, first the end of the training period
        and then the time of the first signal detection.}
    \label{fig:power2}
\end{figure}

\begin{figure}[h!]
    \centering
    \subfloat[]{%
        \includegraphics[width=0.49\textwidth]{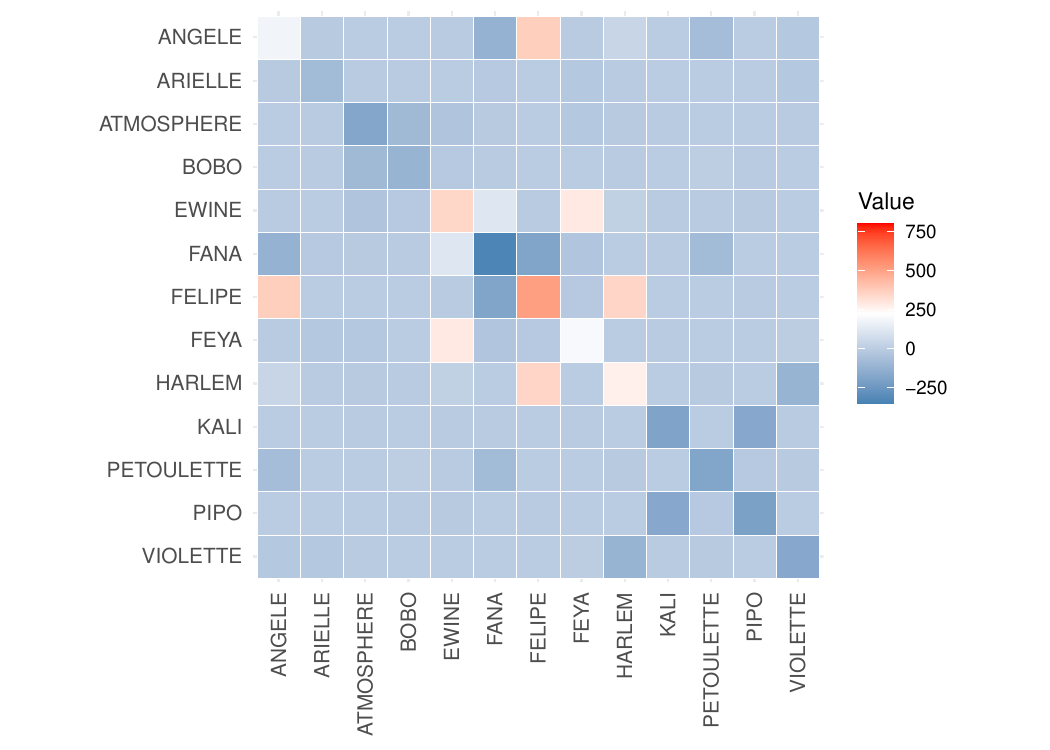}%
    }\hfill
    \subfloat[]{%
        \includegraphics[width=0.49\textwidth]{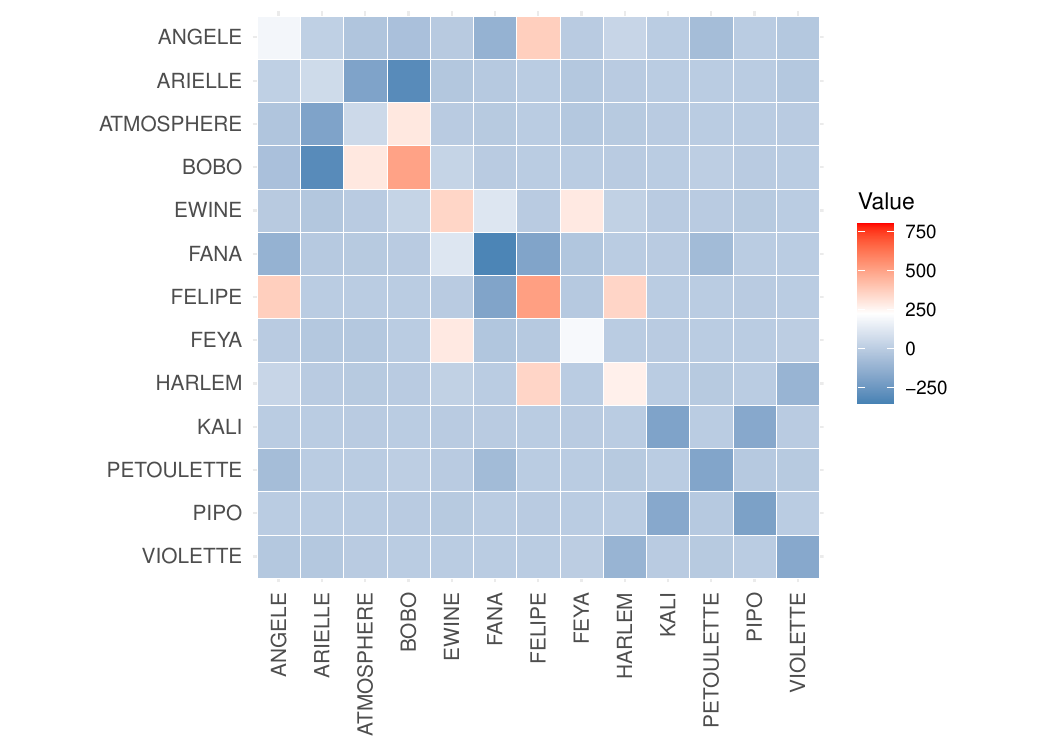}%
    }
    \caption{Example of original data matrix (terminal time $700$) without signal added (left) and with signal added (right). Names of individuals are taken from the original study.}
    \label{fig:matrix}
\end{figure}

Afterwards, we considered a time series of $700$ matrices for our statistical approach to monitor for unusual behavior patterns, using a training set of size $m=200$. During the observation period, social interactions in the baboon group were fairly stable, and no major breaks were observed using methods other than sensors \citep{gelardi:godard:paleressompoulle:claidiere:barrat:2020}. Thus, we expect not to detect any significant changes with our monitoring method either. In Figure \ref{fig:power2} (left), we display the values of our test statistic (time is given in the number of half-hour increments) and see that the statistic never crosses the $95\%$ quantile, depicted by a horizontal dashed line. Next, we introduce a synthetic
signal to see whether we can detect it. The signal is introduced at $k=100$ (time $300$ in the graphs) and uses a vector of the form
$\bx= (x_1,...,x_5,0,0, \cdots,0)$, where $(x_1,...,x_5)$ is uniformly distributed on the $5$-dimensional unit sphere. This corresponds to a change of interaction patterns in a subgroup. Different choices of the group size (anywhere between $2$ and $13$) did not affect the outcomes systematically.
The signal strength is given by $s =10^3$.
This signal strength is very moderate as illustrated in Figure \ref{fig:matrix}.
Signal directions $\bfx$ are i.i.d. after the change has occurred.
In Figure \ref{fig:power2}
(right), we display the statistic for the scenario of an added signal
(vertical solid line), which is detected less than two days after the structure of the social group has changed (dashed vertical line). Unreported simulations show that larger signals lead to faster detection.

\textbf{Funding}
The work of N. D{\"o}rnemann and T. Kutta was supported by the Aarhus
University Research Foundation (AUFF), project numbers
47221, 47222, 47331 and 47388. P. Kokoszka and S. Lee were supported
by United States National Science Foundation grant DMS--2412408.
\small\renewcommand{\baselinestretch}{1.2}
\putbib

\end{bibunit}
\newpage
\setcounter{page}{1}
\begin{bibunit}
\begin{center}
\LARGE \textbf{Supplement}
\end{center}

\section{Proofs and further theoretical results} \label{s:p}

\rev{ We first briefly recall weak convergence of stochastic processes in the function space under consideration, and then include a general remark on properties that will be useful in subsequent arguments. }

\rev{ \textbf{Convergence of stochastic processes} We consider for $T>0$ the space of
bounded functions
\[
\ell^\infty([0,T]) := \{f:[0,T] \to \mathbb{R}, \|f\|_\infty<\infty\},
\]
where $\|f\|_\infty = \sup_{x\in [0,T]} |f(x)|$.
The  space $(\ell^\infty([0,T]), \|\cdot\|_\infty)$ supports all stochastic processes of interest in this paper, but has the drawback of being non-separable. Hence, we frequently consider the separable subspace of continuous functions
\[
\mathcal{C}([0,T]) := \{ f:[0,T] \to \mathbb{R} \,\,\,\textnormal{is continuous}\}.
\]
To quantify the distance between stochastic processes, we introduce the
Prokhorov metric $\pi$, defined for two stochastic processes
$X,Y \in \ell^\infty([0,T])$ as follows
\begin{align}\label{eq_def_prokhorov_metric}
	\pi(X,Y) = \inf\{\varepsilon>0: \ & \mathbb{P}(X \in A) \le  \mathbb{P}(Y \in A_\varepsilon)+\varepsilon,  \\
	& \mathbb{P}(Y \in A) \le  \mathbb{P}(X \in A_\varepsilon)+\varepsilon,\,\, \forall A \,\, \textnormal{measurable}\}.\nn
\end{align}
Here, $A_\varepsilon$ is the open $\varepsilon$ environment of
a measurable set $A \subset \ell^\infty([0,T])$.
Since $\mathcal{C}([0,T])$ is a subspace of $\ell^\infty([0,T])$, $\pi$
in particular induces a distance between the laws of continuous processes
and since $\mathcal{C}([0,T])$ is separable, it even metrizes
weak convergence (e.g., \cite{huber:ronchetti:2009}, Chapter 2).\\ 
\textbf{Notation under $\mathsf{H}_0$} Results in the Supplement are overwhelmingly formulated for data under the null-distribution; that is for model \eqref{e:model:cases} with superindex $j=1$. The only exception is Section \ref{ss:p-cor}, where analyze the divergence of the test statistic under $\mathsf{H}_1$. To lighten notation we will in all other sections simply use $\lambda_t$ instead of $\lambda_t^{(j)}$, because they are of course identical for all $t \ge 1$ under $\mathsf{H}_0$. 
}
\begin{remark} \label{rem:early} We will find it useful to consider a variety
	of partial sum processes, over a variety of slightly modified random variables.
	For their definition, recall the definition of the set
	$\Lambda_t$ in \eqref{eq_def_Lambda_t} and the fact, proved in
	Lemma \ref{lem_p_t} below,  that for any $\varepsilon>0$,
	we have $p_{m^{D_1}} \lesssim m^{-D_2}$,  where
	\begin{align} \label{x1}
		p_{m^{D_1}}:= \mathbb{P}\bigg( \bigcup_{t=1}^{m^{D_1}} \Lambda_t^c\bigg).
	\end{align}
	Here, $D_1, D_2$ can be arbitrarily large. Define
	\begin{align} \label{e:defYt}
		Y_t:=&n^{2/3}  \big( \lambda_t\mathbb{I}\{\Lambda_t\} - b_t\big),\quad b_{t} := \frac{\mathbb{I}\{\Lambda_t\}\mathbb{E}[\lambda_t\mathbb{I}\{\Lambda_t\}]}{\mathbb{P}\{\Lambda_t\}} = \frac{\mathbb{I}\{\Lambda_t\}\mathbb{E}[\lambda_1\mathbb{I}\{\Lambda_1\}]}{\mathbb{P}\{\Lambda_1\}} .
	\end{align}
	The last identity above  holds under the null hypothesis.
	Now recall the definition of $b_{(n)}$ in \eqref{eq_def_b_n}.
	As a consequence of \eqref{x1}, it follows that
	\[
	n^{2/3}(\lambda_t-b_{(n)})=Y_t \qquad \forall \ 1 \le t \le m^{D_1}
	\]
	with probability converging to $1$.
	We call this the \textbf{early replacement property}. In particular, it entails
	\begin{align}\label{e:defPm'}
		P_m(x) = P_m'(x), \qquad
		P_m'(x):= \frac{1}{\sqrt{m \tau_n}}\sum_{t=1}^{\lfloor mx\rfloor}Y_t,
	\end{align}
	simultaneously for all $x \in [0,\lfloor m^{D_1-1}\rfloor ]$ and for any $D_1$ with probability decaying as $ m^{-D_2}$. In turn, this entails directly that
	\begin{align}\label{e:earlrep}
		\pi(\{P_m(x): x \in [0,\lfloor m^{D_1-1}\rfloor ]\},\{P_m'(x): x \in [0,m^{D_1-1}]\}) \lesssim m^{-D_2}.
	\end{align}
\end{remark}

\subsection{Proof of Theorem \ref{lem_3}} \label{ss:p-Gauss}
\textbf{Preliminaries} Throughout this proof, we will encounter a number
of  constants that for technical reasons have to be chosen sufficiently
close to $0$. We will denote them by $\varepsilon, \varepsilon_1, \varepsilon_2,\ldots$.
Often,  these numbers depend on each other and we make this clear,
when we introduce them for the first time. Recall, e.g. the number $\varepsilon>0$
from the definition of the event $\Lambda_t$ in \eqref{eq_def_Lambda_t}.
Using Assumption \ref{ass_2}, we have
\[
n^\varepsilon \lesssim m^{\varepsilon_1}, \qquad \textnormal{where} \quad \varepsilon_1 = \varepsilon_1(\varepsilon).
\]
A property of these small constants,
say $\varepsilon_K = \varepsilon_K(\varepsilon, \varepsilon_1,\ldots,\varepsilon_{K-1})$
will be that if all of the parameters they depend on are sufficiently small,
they can be made arbitrarily small themselves, or mathematically:
For any $\delta_1>0$ there exists a $\delta_2>0$, s.t. $\varepsilon_K$
can be chosen to be $\varepsilon_K<\delta_1$ if $\varepsilon, \varepsilon_1,\ldots,
\varepsilon_{K-1}<\delta_2$. We see this for $\varepsilon_1$,
which by Assumption \ref{ass_2} can be chosen as any value
$\varepsilon_1 \ge \varepsilon/\theta$.
Throughout the first part of this proof, we fix
\[
T=T_m\asymp m^{1/2-\rho}
\]
for some $\rho\in(1/6,1/2)$. All auxiliary $\varepsilon$-constants
introduced below are chosen sufficiently small depending on this fixed
value of $\rho$. Without loss of generality, we assume that $T$ is chosen
such that $Tm\in\mathbb N$.

\textbf{Variables and processes}  Recall the definition of the process $P_m'$ in
\eqref{e:defPm'}. For most of this proof we will study the process $P_m'$, and
only later make the connection to $P_m$.
We notice that random variables $Y_t$
defined in  \eqref{e:defYt} are measurable transforms
of the original time series $(\mathbf{W}_{t},
\mathbf{x}_{t},s_t)_{t \in \mathbb{Z}}$,  and hence $\phi$-mixing
with coefficients satisfying Assumption \ref{ass_3}, part ii).
We also notice that $Y_t$ is bounded. To see this, a small calculation is sufficient,
where in the first step, we add and subtract $2 \sigma \mathbb{I}\{\Lambda_t\}$,
in the second step use the definition of $\Lambda_t$, given in \eqref{eq_def_Lambda_t}, in the third step use that eventually $\mathbb{P}(\Lambda_t) \ge 1-p_{Tm}\ge 2/3$ for sufficiently large $n$ (again Lemma \ref{lem_p_t}) and in the final step exploit the relation $n \lesssim m^{\varepsilon_1}$, giving us
\begin{align} \label{e:bYt}
	&|Y_t| \le n^{2/3}|\lambda_t-2\sigma|\mathbb{I}\{\Lambda_t\}
	+ \frac{\mathbb{I}\{\Lambda_t\}\mathbb{E}[n^{2/3}|\lambda_t-2\sigma|\mathbb{I}\{\Lambda_t\}]}{\mathbb{P}\{\Lambda_t\}} \\\le & n^\varepsilon + \frac{n^\varepsilon}{\mathbb{P}\{\Lambda_t\}} \le 3 n^\varepsilon\lesssim m^{\varepsilon_1}. \nonumber
\end{align}
Moreover, it is clear that $\mathbb{E}Y_t=0$.
In the following, we define the variances
\begin{align} \label{eq_def_tau}
	\tau_{k,n}=\mathbb{E}\bigg(\frac{1}{\sqrt{k}}\sum_{t=1}^k Y_t \bigg)^2 \qquad \tau^{(n)}:= \sum_{t \in \mathbb{Z}}\mathbb{E}[Y_t Y_0].
\end{align}
From \eqref{eq_def_tau}, we get that the normalizing constant
$\tau_n$ can be written as $\tau_n :=\tau_{Tm,n}$.
Both variances are finite. For $\tau_{k,n}$ this is clear,
since the variables $Y_t$ are bounded and for $ \tau^{(n)}$ this
boundedness, together with the decay of the mixing coefficients in
Assumption \ref{ass_3}, part ii) implies that $ \tau^{(n)}<\infty$.
To be precise, we have
\begin{align} \label{e:intau(n)}
	\sum_{t \in \mathbb{Z}}|\mathbb{E}[Y_t Y_0]| \le \sum_{t \in \mathbb{Z}} 4 \phi(t) m^{2\varepsilon_1} \lesssim m^{2\varepsilon_1} \sum_{t \in \mathbb{Z}}  \phi(t),
\end{align}
where we have used \eqref{e:bYt} and an inequality for centered
mixing random variables $X,Y$ with mixing coefficient $\phi(X,Y)$:
If $X,Y$ are bounded by constants $C_X, C_Y$, then
$|\mathbb{E}[XY]| \le 4 \phi(X,Y) C_XC_Y$
(see Lemma 3.9 in \cite{dehling:mikosch:sorensen:2002}).
Next, we define the modified partial sum process (with $\lfloor Tm x \rfloor$ summands)
\begin{align} \label{e:defP''}
	P_m''(x) = \sqrt{m/(Tm)} \,\,P_m'(xT), \qquad x \in [0,1].
\end{align}
We also introduce the rescaled version $P_m''(v_m(x))$, where
\[
v_m(x):= \frac{1}{Tm }\min\{1 \le k \le Tm: k \tau_{k,n} \ge x Tm \tau_{Tm,n}\}.
\]
The function $v_m$ transforms the partial sum process and accounts for the evolving
variance $\tau_{k,n}$ along $k$ - if all $\tau_{k,n}$ (for all $k$) were equal,
it would reduce to the identity. Therefore, handling $v_m(x)$ involves showing
that the variances $\tau_{k,n}$  are (for all $k$ not too small) close to
a fixed value $\tau^{(n)}$ and we provide more details below.
Finally, we define the linearized version $(P_m'')^{lin}$  of $P_m''$
that linearly interpolates the points
\[
\left\{ \left( (k/(Tm)), P_m'' (k / (Tm) ) \right) : k=1,\ldots, Tm \right\}.
\]
We will now verify  three statements, where $c_T = o(T^{-1/2})$ is a sequence
of positive numbers, and then conclude the proof using them.
\renewcommand\theenumi{\alph{enumi}}
\renewcommand\labelenumi{(\theenumi)}
\begin{enumerate}
	\item \label{part_a}
	$ \pi(\{P_m''(x): x \in [0,1]\}, \{P_m''(v_m(x)): x \in [0,1]\})=\mathcal{O}(c_T) $.
	\item \label{part_b}
	$\pi(\{P_m''(v_m(x)): x \in [0,1]\}, \{B(x): x \in [0,1]\})=\mathcal{O}(c_T)$.
	\item \label{part_c}
	$ \pi(\{P_m''(x): x \in [0,1]\}, \{(P_m'')^{lin}(x): x \in [0,1]\})=\mathcal{O}(c_T) $.
\end{enumerate}

\textbf{Proof of \eqref{part_a}} We decompose the index set $[0,1]$ into
$x\le m^{-1/2+\varepsilon_3}$ and $x>m^{-1/2+\varepsilon_3}$, where
$\varepsilon_3>0$ is sufficiently small. Put
$x_0:=m^{-1/2+\varepsilon_3}$ and
$k_0:=\lfloor m^{1/2+\varepsilon_3}\rfloor$.

We first show that the index appearing in the definition of $v_m(x)$ is at
least $k_0$ whenever $x\ge x_0$. By \eqref{e:bYt}, the mixing inequality
used in \eqref{e:intau(n)}, and Assumption \ref{ass_3}, part ii), uniformly
in $k\ge1$,
$\tau_{k,n}\le\sum_{|h|<k}|\mathbb{E}[Y_0Y_h]|
\lesssim n^{2\varepsilon}$.
Moreover, Assumption \ref{ass_3}, part iii), and Lemma \ref{lem_var} imply
that $\tau_{Tm,n}$ is bounded away from zero. Hence, for $k\le k_0$,
\[
k\tau_{k,n}\lesssim k_0n^{2\varepsilon}
=o(Tk_0)\le xTm\tau_{Tm,n},
\]
where we choose $\varepsilon_1$ sufficiently small that
$2\varepsilon_1<1/2-\rho$. Consequently,
$Tm\,v_m(x)>k_0$ for every $x\ge x_0$ and all sufficiently large $m$.

We may therefore apply Lemma \ref{lem_var} uniformly to all indices involved
in the definition of $v_m(x)$. With $\varepsilon_2$ denoting the constant
$\eta$ in that lemma, we obtain
\begin{align}
\delta_m
&:=\sup_{k_0\le k\le Tm}
\left|\frac{\tau_{k,n}}{\tau_{Tm,n}}-1\right|
\lesssim n^{2\varepsilon}k_0^{-1+\varepsilon_2}
\lesssim
m^{(-1+\varepsilon_2)(1/2+\varepsilon_3)+2\varepsilon_1}
=:a_m=o(1).
\label{e:defam}
\end{align}
 Let $j_m(x):=Tm\,v_m(x)$. Thus, $j_m(x)$ is the smallest admissible
integer in the definition of $v_m(x)$, that is, the smallest integer $j$
such that
$j\tau_{j,n}\ge xTm\tau_{Tm,n}.$
As shown above, $j_m(x)\ge k_0$, so the definition of $\delta_m$ gives
$\tau_{j_m(x),n}\le(1+\delta_m)\tau_{Tm,n}$. Consequently,
\[
xTm\tau_{Tm,n}
\le j_m(x)\tau_{j_m(x),n}
\le j_m(x)(1+\delta_m)\tau_{Tm,n},
\]
and hence $v_m(x)=j_m(x)/(Tm)\ge x/(1+\delta_m)$.
For a reverse inequality, set
$q_m(x):=\lceil xTm/(1-\delta_m)\rceil$. Since
$x\ge x_0=m^{-1/2+\varepsilon_3}$, we have, for all sufficiently large $m$,
\[
q_m(x)
\ge xTm
\ge Tm^{1/2+\varepsilon_3}
\ge k_0.
\]
If $q_m(x)\le Tm$, it therefore follows from the definition of $\delta_m$
that
\[
q_m(x)\tau_{q_m(x),n}
\ge q_m(x)(1-\delta_m)\tau_{Tm,n}
\ge xTm\tau_{Tm,n}.
\]
Thus, $q_m(x)$ is admissible, and the minimality of $j_m(x)$ yields
\[
v_m(x)=\frac{j_m(x)}{Tm}
\le\frac{q_m(x)}{Tm}
\le\frac{x}{1-\delta_m}+\frac1{Tm}.
\]
If $q_m(x)>Tm$, we instead use the trivial bound $v_m(x)\le1$. Combining
the lower and upper bounds gives
\[
\frac{x}{1+\delta_m}
\le v_m(x)
\le \min\left\{1,\frac{x}{1-\delta_m}+\frac1{Tm}\right\}.
\]
It follows that, for some constant $C_v>0$ and all sufficiently large $m$,
\begin{align}
\sup_{m^{-1/2+\varepsilon_3}\le x\le1}
\left|v_m(x)-\frac{\lfloor Tmx\rfloor}{Tm}\right|
\le C\left(\delta_m+\frac1{Tm}\right)
\le C_v a_m.
\label{e:bv_m}
\end{align}
In the following, we will not always mention the requirement of $m$ being "large enough", to shorten our arguments.
Now, we have to introduce three constants: First let $q_0>2$ be a constant that later has to be chosen sufficiently large and consider values of $q>q_0$. It follows that
uniformly for $x>m^{-1/2+\varepsilon_3}$ that
\begin{align} \label{e:Pnb1}
	&  \{\mathbb{E} \sup_{x>m^{-1/2+\varepsilon_3}}|P_m''(x)- P_m''(v_m(x))|^q\}^{1/q}\\\le &\bigg\{\mathbb{E} \sup_{|k-j| 
    \le  C_v ( Tm ) a_m } \bigg|\frac{1}{\sqrt{ Tm}}\sum_{t=j+1}^k Y_t \bigg|^q \bigg\}^{1/q}\\
	\lesssim & ((Tm)^2 a_m)^{1/ q} \sup_{|k-j| \le  C_v ( Tm ) a_m }\bigg\{\mathbb{E}  \bigg|\frac{1}{\sqrt{ Tm}}\sum_{t=j+1}^k Y_t \bigg|^{q} \bigg\}^{1/q} \nonumber\\
	\lesssim  &\frac{m^{\varepsilon_4}(Tm)^{2/q_0}  a_m^{1/q_0} (Tm a_m)^{1/2}}{(Tm)^{1/2}} = a_m^{1/2} \{m^{\varepsilon_4}(Tm)^{2/q_0}  a_m^{1/q_0} \}\lesssim m^{\varepsilon_5} a_m^{1/2}.\nonumber
\end{align}
In the first inequality, we have used the definition of the partial sum process $P_m''$ (see   \eqref{e:defP''}), the fact that by construction $P_m''(x) = P_m''(\lfloor x Tm \rfloor/(Tm))$ and the inequality \eqref{e:bv_m}.
In the second inequality, we have used Lemma 2.2.2. from \cite{vaart:wellner:1996}, which is a standard way of bounding the expectation of the maximum of random variables. In the third inequality, we have used Theorem 1 in \cite{yoshihara:1978}, which is an inequality for mixing random variables. Its application is rather simple, because the random variables involved $Y_t$, are uniformly bounded by $3 n^\varepsilon \lesssim m^{\varepsilon_1}$. Moreover, we have exploited that $ q >q_0$ and that $(Tm)^{2}  a_m \to \infty$, which implies together that eventually $[(Tm)^{2}  a_m]^{1/ q} \le [(Tm)^{2}  a_m]^{1/ q_0}$.
Finally, for the last inequality, we have used that, for $q_0$ sufficiently
large,
$m^{\varepsilon_4}(Tm)^{2/q_0}a_m^{1/q_0}
\lesssim m^{\varepsilon_5},$
where $\varepsilon_5$ can be made arbitrarily small. Recalling
\eqref{e:defam}, we obtain
$m^{\varepsilon_5}a_m^{1/2}
=
m^{-1/4+\varepsilon_6},$
where
\[
\varepsilon_6
:=
\varepsilon_5-\frac{\varepsilon_3}{2}
+\frac{\varepsilon_2}{2}\Big(\frac12+\varepsilon_3\Big)
+\varepsilon_1
\]
can be made arbitrarily close to zero.
\\
Next, we consider
$|P_m''(x)- P_m''(v_m(x))| $ for $x \le m^{-1/2+\varepsilon_3}$. Recalling \eqref{e:bv_m}, we get by plugging in $x = m^{-1/2+\varepsilon_3}$
that
\[
v_m(m^{-1/2+\varepsilon_3}) \le m^{-1/2+\varepsilon_3} + C_v a_m \le 2m^{-1/2+\varepsilon_3}
\]
for all $m$ sufficiently large, such that
$C_v a_m \le m^{-1/2+\varepsilon_3}$ (this holds eventually for a sufficiently small choice of $\varepsilon_2$). This means, that we can use the following inequality
\begin{align*}
	\sup_{0 \le x \le  m^{-1/2+\varepsilon_3}}|P_m''(x)- P_m''(v_m(x))| \le 2 \sup_{0 \le  x \le  2 m^{-1/2+\varepsilon_3}}|P_m''(x)|.
\end{align*}
We can now combine Theorem 3.1 in \cite{moricz:serfling:stout:1982} (a maximum inequality for partial sums) with Theorem 1 in \cite{yoshihara:1978} (a moment bound for sums of mixing variables; such a bound is required for the application of \cite{moricz:serfling:stout:1982}) and we obtain
\[
\mathbb{E}[\sup_{0 \le x \le 2 m^{-1/2 + \varepsilon_3}}|P_m''(x)|^q]^{1/q} \lesssim n^{\varepsilon}  m^{-1/4+\varepsilon_3}.
\]
We can further upper bound the right side, using that $T\asymp m^{1/2-\rho}$ and that $n^\varepsilon \lesssim m^{\varepsilon_1}$,
yielding
\[
n^{\varepsilon}  m^{-1/4+\varepsilon_3} \lesssim \frac{m^{\varepsilon_1 +\varepsilon_3 +(1/2-\rho)/2-1/4}}{T^{1/2}
} = \frac{m^{\varepsilon_1+\varepsilon_3 -\rho/2}}{T^{1/2}}
=\frac{m^{-\varepsilon_7}}{T^{1/2}},
\]
where $\varepsilon_7:=\rho/2-\varepsilon_1-\varepsilon_3>0$ after choosing
$\varepsilon_1$ and $\varepsilon_3$ sufficiently small. Moreover, choose any
$0<\varepsilon_8<\rho/2-\varepsilon_6$. Then
\[
m^{\varepsilon_5}a_m^{1/2}
=
m^{-1/4+\varepsilon_6}
\lesssim
m^{-\varepsilon_8-(1/2-\rho)/2}
\asymp
T^{-1/2}m^{-\varepsilon_8}.
\]
Consequently, setting $\delta_0:=\min\{\varepsilon_7,\varepsilon_8\}$, we obtain,
for every $q>q_0$,
\[
\left\{\mathbb{E}\sup_{x\in[0,1]}
|P_m''(x)-P_m''(v_m(x))|^q\right\}^{1/q}
\lesssim T^{-1/2}m^{-\delta_0}.
\]
Since $\rho>1/6$, define $\Delta_\rho:=(6\rho-1)/16>0$ and choose
$0<\kappa<\min\{\delta_0,\Delta_\rho\}$. Increase $q$, if necessary, so that
$q(\delta_0-\kappa)$ is sufficiently large; then it follows from Markov's inequality that
\[
\mathbb{P}\left(
\sup_{x\in[0,1]}|P_m''(x)-P_m''(v_m(x))|
>T^{-1/2}m^{-\kappa}
\right)
\lesssim m^{-q(\delta_0-\kappa)}
\lesssim T^{-1/2}m^{-\kappa}.
\]
Thus, setting
\begin{align}
	c_T:=T^{-1/2}m^{-\kappa},
	\label{e:defc}
\end{align}
completes the verification of claim \eqref{part_a}.

\textbf{Proof of \eqref{part_b}}
We apply Theorem 2.3 of \cite{hafouta:2023}. To avoid confusion with the
matrix dimension $n$, denote the row length in that theorem by $N$ and set
$N=Tm$. Hafouta's triangular-array variables are $\xi_{t,N}=Y_t$,
$1\le t\le N$, and their total variance is
\begin{align}
	\sigma_N^2
	:=
	\operatorname{Var}\left(\sum_{t=1}^N Y_t\right)
	=
	N\tau_{N,n}
	=
	Tm\tau_{Tm,n}.
	\label{eq_def_sigma_N}
\end{align}
The variance time change in Hafouta's theorem coincides with the definition
of $v_m$. Thus, the process denoted by $W_N$ in that theorem is precisely
$\{P_m''(v_m(x)):x\in[0,1]\}$.
Fix constants $2<p_H<q_H$. For $r\in[1,\infty]$, define
$K_{r,N}:=\max_{1\le t\le N}\|Y_t\|_r$. By \eqref{e:bYt},
\begin{align}
	K_{r,N}\le K_{\infty,N}\le3n^\varepsilon
	\lesssim m^{\varepsilon_1}.
	\label{a0}
\end{align}
Let $\alpha_N(j)$ denote the $\alpha$-mixing coefficient of the row
$(Y_1,\ldots,Y_N)$ and define
$\Gamma_N(x):=\sum_{x\le j\le N}\alpha_N(j)^{1/p_H-1/q_H}$.
Since each $Y_t$ is a measurable transformation of the original process,
$\alpha_N(j)\le\phi(j)\le Ca^j$ for some $a\in(0,1)$. Consequently,
$\Gamma_N(x)$ decreases exponentially in $x$, uniformly in $N$, and
$\sup_N\Gamma_N(1)<\infty$.\\
In Hafouta's notation,
$r_N:=\max\{\Gamma_N^{-1}((8K_{q_H,N})^{-1}),1\}$, where
$\Gamma_N^{-1}(u)$ is the smallest positive integer $x$ such that
$\Gamma_N(x)\le u$. The exponential bound and \eqref{a0} imply
$r_N=\mathcal{O}(\log m)\lesssim m^{\varepsilon_9}$ for any fixed
$\varepsilon_9>0$. Therefore,
\begin{align}
	A_N
	:=
	18r_N(1+K_{2,N})
	\lesssim m^{\varepsilon_1+\varepsilon_9}.
	\label{a1}
\end{align}
Next, choose a fixed $j_0$ sufficiently large that $\phi(j_0)<1/4$.
The bound in Section 2.2 of \cite{hafouta:2023} gives
\begin{align}
	\beta_N
	\lesssim q_H(1+j_0K_{\infty,N})
	\lesssim m^{\varepsilon_1}.
	\label{a2}
\end{align}
By Assumption \ref{ass_3}, part iii), and Lemma \ref{lem_var},
$\tau_{Tm,n}$ is bounded away from zero for all sufficiently large $m$.
Hence $\sigma_N\gtrsim(Tm)^{1/2}\asymp m^{3/4-\rho/2}\to\infty$.
Choose $\varepsilon_1$ and $\varepsilon_9$ sufficiently small that
$2(\varepsilon_1+\varepsilon_9)<3/4-\rho/2$. It follows from \eqref{a1}
that $A_N^2/\sigma_N\to0$, and hence $A_N=o(\sigma_N)$.
Set $l_N:=\lfloor\sigma_N\rfloor$. For all sufficiently large $N$,
$l_N\le\sigma_N\le\sigma_N^2/(18A_N^2)$, so Hafouta's condition on
$l_N$ is satisfied. Theorem 2.3 of \cite{hafouta:2023} therefore yields
\begin{align*}
&\pi\left(
\{P_m''(v_m(x)):x\in[0,1]\},
\{B(x):x\in[0,1]\}
\right)\\
&\quad\lesssim
(1+K_{q_H,N})A_N\beta_N\Gamma_N(1)
\left[
\sigma_N^{-(p_H-2)/(2p_H)}|\log\sigma_N|^{3/4}
+
\mathfrak q_N^{p_H/(2p_H+4)}
|\log\mathfrak q_N|^{1/2}
\right],
\end{align*}
where
$\mathfrak q_N:=l_N^{1/2}\sigma_N^{-1}
+l_N\sigma_N^{-2(1-2/p_H)}+l_N^{-1/2}$.
Since $l_N\asymp\sigma_N$,
$\mathfrak q_N\lesssim\sigma_N^{-1/2}+\sigma_N^{-1+4/p_H}$.
Thus, for $p_H>8$, $\mathfrak q_N\lesssim\sigma_N^{-1/2}$. By choosing
$p_H$ sufficiently large and absorbing the logarithmic factors into an
arbitrarily small power, we obtain, for any sufficiently small
$\varepsilon_{12}>0$,
\[
\sigma_N^{-(p_H-2)/(2p_H)}|\log\sigma_N|^{3/4}
+
\mathfrak q_N^{p_H/(2p_H+4)}
|\log\mathfrak q_N|^{1/2}
\lesssim\sigma_N^{-1/4+\varepsilon_{12}}.
\]
Moreover, \eqref{a0}--\eqref{a2} and $\sup_N\Gamma_N(1)<\infty$ imply,
for an arbitrarily small $\varepsilon_{13}>0$, that
$(1+K_{q_H,N})A_N\beta_N\Gamma_N(1)\lesssim m^{\varepsilon_{13}}$.
Consequently,
\begin{align}
&\pi\left(
\{P_m''(v_m(x)):x\in[0,1]\},
\{B(x):x\in[0,1]\}
\right)
\lesssim
m^{\varepsilon_{13}}(Tm)^{-1/8+\varepsilon_{12}/2}.
\label{a3}
\end{align}
Since $T\asymp m^{1/2-\rho}$, the right-hand side of \eqref{a3} is bounded by
\[
T^{-1/2}
m^{-\Delta_\rho+\varepsilon_{13}
+(3/2-\rho)\varepsilon_{12}/2},
\qquad
\Delta_\rho:=\frac{6\rho-1}{16}>0.
\]
Recalling that $c_T=T^{-1/2}m^{-\kappa}$ and $\kappa<\Delta_\rho$, choose
$\varepsilon_{12}$ and $\varepsilon_{13}$ sufficiently small that
$\varepsilon_{13}+(3/2-\rho)\varepsilon_{12}/2+\kappa<\Delta_\rho$.
The right-hand side of \eqref{a3} is then bounded by $c_T$, proving
claim \eqref{part_b}.

\textbf{Proof of \eqref{part_c}} This part is the easiest of the entire proof.
By definition of $(P_m'')^{lin}$,
\[
\sup_{x \in [0,1]}|(P_m'')^{lin}(x)-P_m''(x)|
\le \sup_t \frac{|Y_t|}{\sqrt{Tm \tau_n}}
\lesssim
\frac{m^{\varepsilon_1}}{(Tm)^{1/2}} \le c_T,
\]
where we used \eqref{e:bYt} and again the fact that $\tau_n$ is bounded away from $0$ (see proof of part \eqref{part_a}).
Since this result is deterministic, in particular the bound from \eqref{part_c} follows.

\textbf{Concluding argument} We can now proceed to the final step of the proof.
Using the triangle inequality for the Prokhorov metric, we conclude from \eqref{part_a} and \eqref{part_b} that
\begin{align*}
	& \pi(\{P_m''(x): x \in [0,1]\}, \{B(x): x \in [0,1]\})=\mathcal{O}(c_T).
\end{align*}
Together with \eqref{part_c} and again the triangle inequality, we get
\[
\pi(\{(P_m'')^{lin}(x): x \in [0,1]\}, \{B(x): x \in [0,1]\})=\mathcal{O}(c_T).
\]
Now, since both $(P_m'')^{lin}$ and $ B$ live on the Banach space $\mathcal{C}([0,1])$, we can (after possibly changing the probability space) create a coupling such that
\begin{align} \label{e:approx1}
	\mathbb{P}(\sup_{x \in [0,1]}|(P_m'')^{lin}(x)-B(x)|>c_T)<c_T,
\end{align}
where we recall $c_T = o (T^{-1/2})$. The existence of the named coupling is a
consequence of Lemma \ref{lem:huber} stated in Subsection \ref{ss:au}.
On that probability space, we can use $(P_m'')^{lin}$ to create the variables $Y_1, Y_2,\ldots$ as increments of the process and therewith also construct the non-linearized partial sum process $P_m''$.
We have deterministically $Y_t \lesssim n^\varepsilon$ for all $t$,
which implies by the same reasoning as in part \eqref{part_c} of the proof that
\begin{align} \label{e:approx2}
	\sup_{x \in [0,1]}|(P_m'')^{lin}(x) - P_m''(x)|
	\lesssim \frac{m^{\varepsilon_1}}{(Tm)^{1/2}} \lesssim c_T.
\end{align}
Now, rescaling the process, we can consider $P_m'(x) = \sqrt{T} \,\,P_m''(x/T)$ on this probability space, too,  and we can define the scaled version of the Brownian motion  $\tilde B (x):=\sqrt{T}B(x/T)$.
Combining the results in \eqref{e:approx1} and \eqref{e:approx2}, we obtain for some large enough constant $C_0>0$ that
\[
\mathbb{P}\left (\sup_{x \in [0,T]}|P_m'(x)-\tilde B (x)|>C_0 \sqrt{T} c_T \right )<c_T.
\]
Since $\sqrt{T}c_T=m^{-\kappa}$ and $c_T\le m^{-\kappa}$, it follows
from Lemma \ref{lem:huber}, in the reverse direction, that
\[
\pi(P_m',\tilde B)=\mathcal{O}(m^{-\kappa}).
\]
This is a statement about the distributions of $P_m'$ and $\tilde B$ and
therefore also holds for $P_m'$ on the original probability space. By the
early replacement property \eqref{e:earlrep}, for arbitrarily large $D_2$,
\[
\pi(P_m,P_m')\lesssim m^{-D_2}.
\]
Choosing $D_2>\kappa$ and using the triangle inequality, we obtain
\[
\pi(P_m,\tilde B)=\mathcal{O}(m^{-\kappa}).
\]
Now let $P_m^{lin}$ be the version of $P_m$ that linearly interpolates the
points $\{(k/m,P_m(k/m)):k=1,\ldots,Tm\}$. Let $E_m$ denote the early
replacement event up to time $Tm$. By Remark \ref{rem:early},
$\mathbb{P}(E_m^c)\lesssim m^{-D_2}$ for arbitrarily large $D_2$. On $E_m$,
\[
\begin{aligned}
\sup_{x\in[0,T]}|P_m^{lin}(x)-P_m(x)|
&\le
\max_{1\le t\le Tm}
\frac{n^{2/3}|\lambda_t-b_{(n)}|}{\sqrt{m\tau_n}}\\
&=
\max_{1\le t\le Tm}\frac{|Y_t|}{\sqrt{m\tau_n}}
\lesssim m^{-1/2+\varepsilon_1}
\lesssim m^{-\kappa},
\end{aligned}
\]
where the last inequality follows by choosing
$\varepsilon_1<1/2-\kappa$. Consequently, again by Lemma
\ref{lem:huber} in the reverse direction,
\[
\pi(P_m^{lin},P_m)=\mathcal{O}(m^{-\kappa}).
\]
Together with the preceding bound and the triangle inequality, this gives
\[
\pi(P_m^{lin},\tilde B)=\mathcal{O}(m^{-\kappa}).
\]
Since $P_m^{lin}$ and $\tilde B$ take values in the separable space
$\mathcal{C}([0,T])$, Lemma \ref{lem:huber}, now in the direction
i) $\Rightarrow$ ii), gives a coupling such that, for some constant $C_1>0$,
\[
\mathbb{P}\left(
\sup_{x\in[0,T]}|P_m^{lin}(x)-\tilde B(x)|
>C_1m^{-\kappa}
\right)
\le C_1m^{-\kappa}.
\]
On this probability space, construct $P_m$ from the grid values of
$P_m^{lin}$. The preceding interpolation bound continues to hold outside
an event of probability $\mathcal{O}(m^{-D_2})$. Hence, after increasing
the constant if necessary, the union bound gives
\[
\mathbb{P}\left(
\sup_{x\in[0,T]}|P_m(x)-\tilde B(x)|
>C_2m^{-\kappa}
\right)
\le C_2m^{-\kappa}.
\]
Setting $d_m:=C_2m^{-\kappa}$ proves the desired result.
For subsequent use, notice that, when $\rho=1/4$, the preceding proof allows
us to choose constants $0<\nu<\kappa<1/32$. Since
$d_m=C_2m^{-\kappa}$, it follows that $d_m\le m^{-\nu}$ for all sufficiently
large $m$. This concludes the proof of Theorem \ref{lem_3}.

\subsection{Proof of Lemma \ref{lem_4}}\label{ss:p-l3}
Recall the definitions of $\Gamma_{n,m}(k)$ and $Y_t$ in \eqref{e:def:gam}
and \eqref{e:defYt}, respectively.
Using the early replacement property in Remark \ref{rem:early}, we get with probability going to $1$ that
\[
\sup_{m^{1+\zeta}\le k \le m^L} \Gamma_{n,m}(k) = \sup_{m^{1+\zeta}\le k \le m^L} \frac{D_m'(k)}{V_m/\sqrt{\tau_n}},
\]
where
\begin{align*}
	D_m'(k):=&\frac{\sqrt{m}}{ \sqrt{\tau_n}\,\,(m+k)}
	\bigg(\sum_{t=m+1}^{m+k}Y_t-\frac{k}{m}\sum_{t=1}^m Y_t\bigg) \\
	=& \frac{\sqrt{m}}{\sqrt{\tau_n} \,\,(m+k)}
	\sum_{t=m+1}^{m+k}Y_t-\frac{k}{\sqrt{\tau_n}(m+k)\sqrt{m}}\sum_{t=1}^m Y_t
	=:  D_{1,m}'(k)-D_{2,m}'(k).
\end{align*}
Here, $D_{1,m}', D_{2,m}'$ are defined in the obvious way. We will now demonstrate that
\begin{align}\label{e:bD1m}
	\sup_{m^{1+\zeta}\le k \le m^L} |D_{1,m}'(k)|=o_P(1).
\end{align}
As a preparatory step, we let $\zeta':= \zeta/(2(1+\zeta))$ and
note that $(1+\zeta)(1/2-\zeta') = 1/2$. Then, we find the following
upper bounds for $k\ge m^{1+\zeta}$:
\begin{align*}
	m+k & = (m+k)^{1/2+\zeta'} (m+k)^{1/2-\zeta'} \ge (m+k)^{1/2+\zeta'} k^{1/2-\zeta'}\\
    \ge& (m+k)^{1/2+\zeta'} m^{(1+\zeta)(1/2-\zeta')}  = (m+k)^{1/2+\zeta'} \sqrt{m}
	.
\end{align*}
Using these insights, we get
\begin{align*}
	& \sup_{m^{1+\zeta}\le k \le m^L} |D_{1,m}'(k)| = \sup_{m^{1+\zeta}\le k \le m^L}\bigg|\frac{\sqrt{m}}{\sqrt{\tau_n} \,\,(m+k)}
	\sum_{t=m+1}^{m+k}Y_t\bigg| %
	\\
	\lesssim  &
	\sup_{m^{1+\zeta}\le k \le m^L}\bigg|\frac{\sqrt{m}}{(m+k)^{1/2+\zeta'}\sqrt{m}}\sum_{t=m+1}^{m+k}Y_t\bigg| =
	\sup_{m^{1+\zeta}\le k \le m^L}\bigg|\frac{1}{(m+k)^{1/2+\zeta'}}\sum_{t=m+1}^{m+k}Y_t\bigg| \\
	\le &   \frac{n^\varepsilon}{m^{\zeta'/2}} \sup_{1\le k \le m^L}\bigg|\frac{1}{(m+k)^{1/2+\zeta'/2}}\sum_{t=m+1}^{m+k}\frac{Y_t}{n^\varepsilon}\bigg|
	\le  \frac{n^\varepsilon}{m^{\zeta'/2}} \sup_{k \ge 1}\bigg|\frac{1}{k^{1/2+\zeta'/2}}\sum_{t=1}^{k}\frac{Y_{t+m}}{n^\varepsilon}\bigg| \\
	\overset{d}{=} & \frac{n^\varepsilon}{m^{\zeta'/2}} \sup_{k \ge 1}\bigg|\frac{1}{k^{1/2+\zeta'/2}}\sum_{t=1}^{k}\frac{Y_{t}}{n^\varepsilon}\bigg| ,
\end{align*}
where we used Assumption \ref{ass_2}  for the second-to-last inequality, and the last equality holds because of the stationarity of the $Y_t$.
Now, the fact that
\[
\sup_{k \ge 1} \frac{1}{k^{1/2+\zeta'/2}} \bigg|\sum_{t=1}^k \frac{Y_t}{n^\varepsilon}\bigg| = \mathcal{O}_P(1)
\]
follows exactly as in the proof of Lemma C.1 of \cite{kutta:jach:kokoszka:2024} and is therefore not repeated here. Since we can make $\varepsilon$ arbitrarily small, we may assume that $\varepsilon<\theta\zeta'/2$ and we have
\[
\frac{n^\varepsilon}{m^{\zeta'/2}}  \asymp m^{\varepsilon/\theta-\zeta'/2}=o(1).
\]
This shows that \eqref{e:bD1m} holds.\\
Now, $V_m/\sqrt{\tau_n}$
converges in distribution to an a.s. positive random variable. To see this, notice that $V_m/\sqrt{\tau_n}$ is a continuous transform of the partial sum process $\{P_m(x): x \in [0,1]\}$ and hence, it follows by the continuous mapping theorem that
\begin{align} \label{e:Vmas}
	V_m/\sqrt{\tau_n}=\int_0^1 |P_m(x)-\frac{\lfloor xm\rfloor}{m}P_m(1)|
	dx \overset{d}{\to} \int_0^1 |B(x)-xB(1)|dx.
\end{align}
Combining this with \eqref{e:bD1m} yields
\begin{align} \label{e:Dpres1}
	\sup_{m^{1+\zeta}\le k \le m^L} \Gamma_{n,m}(k)
	= \sup_{m^{1+\zeta}\le k \le m^L} \frac{D_{m}'(k)}{V_m/\sqrt{\tau_n}}+o_P(1)=\sup_{m^{1+\zeta}\le k \le m^L} \frac{-D_{2,m}'(k)}{V_m/\sqrt{\tau_n}}+o_P(1).
\end{align}
Now, using the early replacement property in reverse
(replacing $Y_t$ by $n^{2/3}(\lambda_t-b_{(n)})$ for all $1 \le t \le m^L$,
which is possible with probability going to $1$; see Remark \ref{rem:early}), we obtain
\begin{align} \label{e:D2m}
	D_{2,m}'(k) =  \frac{kn^{2/3}}{\sqrt{\tau_n}(m+k)\sqrt{m}}\sum_{t=1}^m [\lambda_t-b_{(n)}] = \frac{k}{m+k}P_m(1)  = P_m(1)+o_P(1).
\end{align}
The above remainder $o_P(1)$ is independent of $k$ and thus the asymptotic equality is uniform over all $k \ge m^L$.
This shows the desired result
\begin{align} \label{e:Gamear}
	\sup_{m^{1+\zeta}\le k \le m^L} \Gamma_{n,m}(k)
	= \sup_{m^{1+\zeta}\le k \le m^L} \frac{-P_m(1)}{V_m/\sqrt{\tau_n}}+o_P(1),
\end{align}
uniformly over all $k$ between $m^{1+\zeta}$ and $m^{L}$.
Next, we consider values of $k>m^L$. Since we can freely choose $L>0$,
we choose it large enough such that
\[
\frac{L}{3}-\frac{1}{2}-\frac{2}{\theta}>0.
\]
Similarly as before, we decompose
\begin{align*}
	\Gamma_{n,m}(k) = &\frac{D_m''(k)}{V_m/\sqrt{\tau_n}},\qquad \textnormal{where}\quad D_m'' :=D_{1,m}'' - D_{2,m}'', \quad \textnormal{and}\\
	D_{1,m}''(k):= &\frac{\sqrt{m}}{\sqrt{\tau_n} \,\,(m+k)}
	\sum_{t=m+1}^{m+k}Y_t', \qquad Y_t':= n^{2/3}\big( \lambda_t - \mathbb{E}[\lambda_t]\big), \\
	D_{2,m}''(k):= &
	\frac{k}{\sqrt{\tau_n}(m+k)\sqrt{m}}\sum_{t=1}^m Y_t'.
\end{align*}
Combining Lemma \ref{lem:exp} in Subsection \ref{ss:au} and Remark \ref{rem:early},
we get  $ D_{2,m}''(k) =  D_{2,m}'(k)+o(1)$,
uniformly over $k$,  and we have analyzed $ D_{2,m}'$
in the first part of this proof already (see \eqref{e:D2m}).
Thus, we turn to the analysis of $D_{1,m}''$.
Recall $\tau^{(n)}$  defined in \eqref{e:def:tau^n}.
By Lemma \ref{lem_var} in Subsection \ref{ss:au},
$\tau_n = \tau^{(n)}+o_P(1)$ and
$\tau^{(n)}$  is bounded away from $0$
by Assumption \ref{ass_3}, which implies that $\tau_n$
(defined in \eqref{e:def:tau_n}) is bounded away from $0$.
Next,
recall that we consider $k>m^L$ and that $n\asymp m^{1/\theta}$. This implies
\begin{align*}
	\frac{\sqrt{m}}{\sqrt{\tau_n} \,\,(m+k)} \lesssim \frac{\sqrt{m}}{m+k} \le \frac{\sqrt{m}}{(m+k)^{2/3}k^{1/3}} \le \frac{m^{-L/3+1/2}}{(m+k)^{2/3}} \lesssim  \frac{m^{-L/3+1/2+5/(3\theta)}}{(m+k)^{2/3}n^{5/3}}.
\end{align*}
Accordingly, we have
\begin{align*}
	\sup_{k>m^L}|D_{1,m}''(k)| \lesssim m^{-L/3+1/2+5/(3\theta)}  \cdot \sup_{k>m^L}\bigg|\frac{1}{(m+k)^{2/3}}\sum_{t=m+1}^{m+k}\frac{Y_t'}{n^{5/3}}\bigg|.
\end{align*}
The first factor on the right  tends to $0$, by our choice of $L$. The supremum can once more be upper bounded by
\begin{align} \label{e:Dpres2}
	&\sup_{k>m^L}\bigg|\frac{1}{k^{2/3}}\sum_{t=m+1}^{m+k}\frac{Y_t'}{n^{5/3}}\bigg| \overset{d}{=}
	\sup_{k> m^L  }\bigg|\frac{1}{k^{2/3}}\sum_{t=1}^{k}\frac{Y_t'}{n^{5/3}}\bigg|
	\\
    \leq &\sup_{1 \le k <\infty}\bigg|\frac{1}{k^{2/3}}\sum_{t=1}^{k}\frac{Y_t'}{n^{5/3}}\bigg|=\mathcal{O}_P(1).\nonumber
\end{align}
Here, the $\mathcal{O}_P(1)$ rate follows again as in the proof of Lemma C.1 of \cite{kutta:jach:kokoszka:2024}. This approach relies on the fact that $\mathbb{E}|Y_t'/n^{5/3}|^q \le C_q<\infty$, where $C_q>0$ is fixed and only depends on $q$ (the proof of this result uses similar techniques as the proof of
Lemma \ref{lem:exp} and is omitted).  Again using the fact that
$V_m/\sqrt{\tau_n}$ converges weakly to an
a.s. positive random variable (see eq. \eqref{e:Vmas}),
we obtain
\[
\sup_{ k > m^L} \Gamma_{n,m}(k)= \sup_{ k > m^L} \frac{D_{m}''(k)}{V_m/\sqrt{\tau_n}}+o_P(1)=\sup_{k >m^L} \frac{-D_{2,m}'(k)}{V_m/\sqrt{\tau_n}}+o_P(1).
\]
As we have seen in \eqref{e:D2m}, it holds that uniformly in $k$ that $D_{2,m}'(k) = P_m(1)+o_P(1)$ and hence we get
\begin{align} \label{e:Gamlat}
	\sup_{ k > m^L} \Gamma_{n,m}(k) = \sup_{k \ge m^L} \frac{-P_m(1)}{V_m/\sqrt{\tau_n}}+o_P(1).
\end{align}
Thus, combining \eqref{e:Gamear} and \eqref{e:Gamlat}, we  have established that
\[
\sup_{k \ge m^{1+\zeta}} \Gamma_{n,m}(k) = \sup_{k \ge m^{1+\zeta}} \frac{-P_m(1)}{V_m/\sqrt{\tau_n}}+o_P(1),
\]
which proves the Lemma.

\subsection{Proof of Theorem \ref{theomain}} \label{ss:p-main}
We apply Theorem \ref{lem_3} with $\rho=1/4$ and, for definiteness, choose
$T_m=2m^{1/4}$, up to inconsequential integer rounding. Thus, the approximation
interval contains $[0,1+m^{1/4}]$ for all sufficiently large $m$. We then combine
Theorem \ref{lem_3} and Lemma \ref{lem_4}.
First, we decompose
\[
\sup_{k \ge 1}\Gamma_{n,m}(k) = \max \bigg\{ \sup_{1 \le k \le m^{5/4}}\Gamma_{n,m}(k), \sup_{k> m^{5/4}}\Gamma_{n,m}(k) \bigg\}.
\]
Using Lemma \ref{lem_4}, we have
\[
\sup_{k> m^{5/4}}\Gamma_{n,m}(k) = \frac{-P_m(1)}{V_m/\sqrt{\tau_n}}+o_P(1).
\]
Next, using the definition of $\Gamma_{n,m}(k)$ in \eqref{e:def:gam} and of the partial sum process $P_m$ in \eqref{e:def:Pm}, we have
\begin{align*}
	&\sup_{1 \le k \le  m^{5/4}}\Gamma_{n,m}(k) = \sup_{\substack{x = \frac{k}{m}, \\ k = 1, \ldots, \lfloor m^{5/4} \rfloor}}
	\frac{
		P_m(1 + x) - (x + 1) P_m(1)
	}{
		(1 + x) V_m / \sqrt{\tau_n}
	} \\
	=& \sup_{0 \le x \le m^{1/4}}
	\frac{
		P_m(1 + x) - (\lfloor xm \rfloor/m + 1) P_m(1)
	}{
		(1 + \lfloor xm \rfloor/m) V_m / \sqrt{\tau_n}.
	}
\end{align*}
Now, we move to the coupling from Theorem \ref{lem_3}. Consider the coupling of $(P_m, B_m)$ created in this lemma on a probability space $(\Omega_m, \mathcal{A}_m, \mathbb{P}_m)$. We can take the infinite product of these probability spaces and call it $(\Omega, \mathcal{A}, \mathbb{P})$.
Then we can write on this probability space
\[
\sup_{0\le x\le1+m^{1/4}}|P_m(x)-B_m(x)|=o_P(1).
\]
where $(P_m)_m$ is a sequence of variables, distributed as our partial sum process with training sample size $m$ and $B_m$ a sequence of standard Brownian motions.
It then follows by the above approximation that
\begin{align*}
	&\sup_{0 \le x \le m^{1/4}} \frac{(P_m(1+x)-(\lfloor xm \rfloor/m+ 1) P_m(1))}{(1+\lfloor xm \rfloor/m)V_m/\sqrt{\tau_n}}\\
	= &\sup_{0 \le x  \le m^{1/4}} \frac{B_m(1+x)-(x+1)B_m(1)}{(1+x) \cdot \int_0^1 |B_m(s)-s B_m(1)|ds} +o_P(1)=:L_{1,m}+o_P(1)
\end{align*}
and
\[
\frac{-P_m(1)}{V_m/\sqrt{\tau_n}}= \frac{-B_m(1)}{\int_0^1 |B_m(s)-s B_m(1)|ds}+o_P(1)=:L_{2,m}+o_P(1) .
\]
Here, we have used that formally
\[
V_m /\sqrt{\tau_n}= \int_0^1 |P_m(s)-\lfloor sm\rfloor/m P_m(1)|ds.
\]
Now, to finish the proof, we have to study the distribution of $\max\{L_{1,m}, L_{2,m}\}$.
To this end, it is convenient to replace $B_m$ by a fixed Brownian motion $B$, giving us $\max\{L^B_{1,m}, L^B_{2,m}\}$ (now the distribution only depends on $m$ through the supremum). Using monotone convergence of the test statistic, it follows directly that pathwise
\begin{align*}
L^B_{1,m}=& \sup_{0 \le x  \le m^{1/4}} \frac{B(1+x)-(x+1)B(1)}{(1+x) \cdot \int_0^1 |B(s)-sB(1)|ds}\\ \to &\sup_{0 \le x  <\infty} \frac{B(1+x)-(x+1)B(1)}{(1+x) \cdot \int_0^1 |B(s)-sB(1)|ds} =L.
\end{align*}
Moreover, we have
\[
\lim_{x \to \infty}\frac{B(1+x)-(x+1)B(1)}{(1+x) \cdot \int_0^1 |B(s)-sB(1)|ds} = \frac{-B(1)}{\int_0^1 |B(s)-sB(1)|ds} = \lim_{m \to \infty} L_{2,m}^B,
\]
which shows the desired result that
$\max\{L^B_{1,m}, L^B_{2,m}\} \overset{d}{\to} L$.

\subsection{Proof of Corollary \ref{cor_test}} \label{ss:p-cor}
\textcolor{black}{
Before starting this proof, recall that $\lambda_t$ denotes the $t$th observed eigenvalue and its definition through $\lambda_t^{(j)}$ is based on a distinction by cases;  we refer to eq. \eqref{e:cases:j}, for the precise definition. 
} 
We only prove the second part, since the first part
(level control under the null hypothesis) is directly implied by
Theorem \ref{theomain}. In the proof, we will use the fact that by Lemma \ref{lem_var} we have that
\[
\tau_n = \tau^{(n)}+o(1)
\]
and that
\[
0<\tau^{(n)} \lesssim n^{2\epsilon}
\]
(see \eqref{e:intau(n)}). Notice that a condition of said lemma is that data are generated under the null-distribution, \textcolor{black}{which is in our model always true of the data from the training sample, where no detectable signal exists.}  
\\
Now, we consider the case $k^* \asymp m$ and without loss of generality assume that $k^* \ge m$. All other cases work exactly analogously and focusing on this specific case merely simplifies the presentation of the result. We have that
\begin{align}
	\sup_{k \ge 1} \Gamma_{n,m}(k) \ge  \Gamma_{n,m}(2k^*) = \frac{D_m(2k^*)/\sqrt{\tau_n}}{V_m/\sqrt{\tau_n}}.
\end{align}
We first briefly discuss the denominator. From the proof of Theorem \ref{theomain}, we have seen that $V_m/\sqrt{\tau_n} $ converges to a non-degenerate, a.s. positive random variable and hence
\[
(V_m/\sqrt{\tau_n})^{-1} = \mathcal{O}_P(1)
\]
(again, notice that the distribution of $V_m$ is the same under the null hypothesis and alternative because it only depends on the training sample).
We can hence focus our analysis on the numerator $D_m(2k^*)/\sqrt{\tau_n}$ and we begin by decomposing $D_m(2k^*)$
as follows
\begin{align} \nonumber
	D_m(2k^*)=  & \frac{n^{2/3}\sqrt{m}}{ \,\,m+2k^*}
	\bigg(\sum_{t=m+1}^{m+2k^*}\lambda_t-\frac{2k^*}{m}\sum_{t=1}^m\lambda_t\bigg)\\ = &\bigg\{ \frac{n^{2/3}\sqrt{m}}{ \,\,m+2k^*}
	\bigg(\sum_{t=m+1}^{m+k^*}\{\lambda_t^{(1)}-2\sigma\}-\frac{2k^*}{m}\sum_{t=1}^m\{\lambda_t^{(1)}-2\sigma\}\bigg)\bigg\}\nonumber\\
	& + \bigg\{ \frac{n^{2/3}\sqrt{m}}{ \,\,m+2k^*}
	\bigg(\sum_{t=m+k^*+1}^{m+2k^*}\Big\{\lambda_t^{(2)}-\Big(s_t^{(2)} + \frac{\sigma^2}{s_t^{(2)}} \Big)\Big\}\bigg)\bigg\}\nonumber\\
	& + \bigg\{\frac{n^{2/3}\sqrt{m}}{ \,\,m+2k^*}
	\bigg(\sum_{t=m+k^*+1}^{m+2k^*}\Big(s_t^{(2)} + \frac{\sigma^2}{s_t^{(2)}} \Big)-k^*(2\sigma)\bigg)\bigg\}=:R_1+R_2+R_3.\nonumber
\end{align}
The terms $R_1, R_2, R_3$ on the right are defined in the obvious way. The study of $R_1, R_2$ is related. For $R_1$, one uses the concentration result in Lemma \ref{lem_rmt} and hence, we focus on the more challenging term $R_2$. For this analysis, we use the following fact that will be validated at the end of this proof. For any small enough constant $\chi>0$
\begin{align} \label{e:id:knowles}
	\mathbb{P}\bigg(\Big|\lambda_t^{(2)} -\Big(s_t^{(2)} + \frac{\sigma^2}{s_t^{(2)}} \Big)\Big|<n^{-1/2+\chi}, \,\,\,\forall t=m+k^*+1,\ldots,m+2k^* \bigg)=1-o(1).
\end{align}
On this event, we have that
\begin{align*}
	|R_2| \le \frac{n^{2/3}\sqrt{m}}{ \,\,m+k^*} n^{-1/2+\chi}k^* \asymp m^{1/2}n^{1/6+\chi}.
\end{align*}
Using the fact that $1/\sqrt{\tau_n}$ is bounded, we get $R_2/\sqrt{\tau_n}=\mathcal{O}_P(m^{1/2}n^{1/6+\chi})$.\\
Finally, we turn to $R_3$, the dominating term. Recall that the probability measure $\mathbb{P}^{(2)}$ (supercritical signal under the alternative hypothesis) has support inside $(\sigma,\infty)$. The support is closed and hence its leftmost endpoint $x_*$ satisfies
\begin{align}
	\label{eq_lower_bound_x_star}
	x_* > \sigma + \eta
\end{align}
for some $\eta >0.$
Note that $s_t^{(2)} + \sigma^2 / s_t^{(2)}$ denotes the (random) approximate of $\lambda_t^{(2)}$ in the supercritical case (Lemma \ref{lem:pizzo}), while, in the subcritical case,  $\lambda_t^{(1)}$ approaches $2\sigma$ (Lemma \ref{lem_rmt}).
Thus, $R_3$ essentially consists of the difference of these two limits, and our task is to show that they are well-separated under our assumptions. To this end, consider the function $g(x)= x + \sigma^2 / x - 2 \sigma, ~x \geq \sigma.$ Then, it holds $g(\sigma)=0$, and g is monotonically increasing for $x \geq \sigma$.  Thus, using also \eqref{eq_lower_bound_x_star}, we obtain, for some $\eta' > 0$,
\begin{align*}
	s_t^{(2)} + \frac{\sigma^2}{s_t^{(2)}} - 2 \sigma = g(s_t^{(2)}) \geq \eta' \quad \textnormal{almost surely}.
\end{align*}
Consequently, we get the almost sure lower bound
\begin{align*}
	R_3 \ge \frac{n^{2/3}\sqrt{m}}{ \,\,m+k^*} (k^* \eta')\asymp n^{2/3}m^{1/2}.
\end{align*}
Since $\tau_n = \tau^{(n)}+o(1)$ and $\tau^{(n)} \lesssim n^{2\epsilon}$ (see beginning of this proof), we have that almost surely
\[
n^{2/3-\epsilon}m^{1/2}\lesssim R_3/\sqrt{\tau_n}.
\]
The same concentration and union-bound argument based on Lemma
\ref{lem_rmt} gives
$R_1/\sqrt{\tau_n}=\mathcal{O}_P(m^{1/2}n^\epsilon)$.
Consequently, for $\epsilon<1/3$ and $\chi+\epsilon<1/2$,
\[
\frac{|R_1|+|R_2|}{\sqrt{\tau_n}}
=
\mathcal{O}_P\bigl(m^{1/2}n^\epsilon
+m^{1/2}n^{1/6+\chi}\bigr)
=
o_P\bigl(m^{1/2}n^{2/3-\epsilon}\bigr)
=
o_P\bigl(R_3/\sqrt{\tau_n}\bigr),
\]
where the last step follows from the preceding lower bound on $R_3$.
Since $ R_3/\sqrt{\tau_n} \to \infty$ as $m \to \infty$, we obtain
\begin{align*}
	\frac{D_m(2k^*)/\sqrt{\tau_n}}{V_m/\sqrt{\tau_n}} \ge \frac{R_3/\sqrt{\tau_n}}{V_m/\sqrt{\tau_n}}+o\bigg(\frac{R_3/\sqrt{\tau_n}}{V_m/\sqrt{\tau_n}}\bigg)\overset{P}{\to} \infty.
\end{align*}
Thus, it is left to show that \eqref{e:id:knowles} holds. For this result to hold, we use the union bound and get the upper bound %
\begin{align*}
	& \mathbb{P}\bigg(\Big|\lambda_t^{(2)} -\Big(s_t^{(2)} + \frac{\sigma^2}{s_t^{(2)}} \Big)\Big|<n^{-1/2+\chi}, \,\,\, \forall t=m+k^*+1,\ldots,m+2k^* \bigg)
	\\ & \geq
	1-k^*  \mathbb{P}\bigg(\Big|\lambda_1^{(2)} -\Big(s_1^{(2)} + \frac{\sigma^2}{s_1^{(2)}} \Big)\Big|\ge n^{-1/2+\chi} \bigg) \\ & \geq 1- C m \mathbb{P}\bigg(\Big|\lambda_1^{(2)} -\Big(s_1^{(2)} + \frac{\sigma^2}{s_1^{(2)}} \Big)\Big|\ge n^{-1/2+\chi} \bigg) =:1- Cm \mathbb{P}_m,
\end{align*}
where we used that we study the case $k^*\asymp m$.
We claim that $m \mathbb{P}_m$ converges to $0$. To see this, define the conditional probability
\[
f_m(s_1^{(2)}, \mathbf{x}_1^{(2)}):=\mathbb{P}\bigg(\Big|\lambda_1^{(2)} -\Big(s_1^{(2)} + \frac{\sigma^2}{s_1^{(2)}} \Big)\Big|\ge n^{-1/2+\chi} \Big| (s_1^{(2)}, \mathbf{x}_1^{(2)})\bigg)
\]
and rewrite
\begin{align*}
	& m \mathbb{P}_m  = m\int_{f_m(s_1^{(2)}, \mathbf{x}_1^{(2)})<1/m^2}  f_m(s_1^{(2)}, \mathbf{x}_1^{(2)}) + m \int_{f_m(s_1^{(2)}, \mathbf{x}_1^{(2)}) \ge 1/m^2}  f_m(s_1^{(2)}, \mathbf{x}_1^{(2)}) \\
	\le &  1/m + m \int_{f_m(s_1^{(2)}, \mathbf{x}_1^{(2)}) \ge 1/m^2}  f_m(s_1^{(2)}, \mathbf{x}_1^{(2)}).
\end{align*}
Now suppose on the contrary that $m \mathbb{P}_m$ does not converge to $0$. If this were so, there must be a sequence of values $(s_1^{(2)}, \mathbf{x}_1^{(2)})$ (indexed in $n$ and thus $m$), such that $f_m(s_1^{(2)}, \mathbf{x}_1^{(2)})>1/m^2 \ge n^{-c}$ for some constant $c>0$. 
This, however, is at odds with the Theorem 3.4 of \cite{knowles:yin:2014}\footnote{\rev{The result of \cite{knowles:yin:2014} is stated under the stronger assumption
of uniform subexponential decay of the Wigner entries.
Under our moment bounds, the required 
bound for $f_m(s_1^{(2)},\mathbf{x}_1^{(2)})$ follows from
\cite[Theorem 2.15]{bloemendal:erdos:knowles:yau:yin:2014} by a standard
application of the matrix determinant lemma (see, e.g. proof of Theorem 6.3 in \cite{knowles:yin:2013:a}).%
}}, which states that for any sequence  $(s_1^{(2)}, \mathbf{x}_1^{(2)})$ (indexed in $n$) we have faster than polynomial decay of $f_m(s_1^{(2)}, \mathbf{x}_1^{(2)})$ eventually. Consequently  $m \mathbb{P}_m\to 0$ is established.
Thus \eqref{e:id:knowles} holds, and the proof concludes.

\subsection{Proofs of Lemmas \ref{lem_rmt} and \ref{lem:pizzo}}
\label{ss:proof_rmt}

\noindent {\sc Proof of Lemma \ref{lem_rmt}:}
For the sake of notational convenience, we omit the index $1$ throughout this proof, writing $\bfx_1 = \bfx$, $s_1 =s$ and $\bfW_1 = \bfW.$
To begin with, we want to show that the diagonalization $\diag (s, 0, \ldots, 0)$ of the rank-one matrix $s \bfx \bfx^\top$ satisfies the regularity assumptions of \cite{Lee2015}, who consider the fluctuations and eigenvalue rigidity of a Wigner matrix under a diagonal deformation. We will then use universality results from \cite{Knowles2017} to allow for general rank-one perturbations. Note that similar results on eigenvalue rigidity, but under stronger assumptions, can also be found in \cite{knowles:yin:2013:a}.

Note that \cite{Lee2015} consider the random matrix $\mathbf{H} = \lambda_0 \mathbf{V} + \frac{1}{\sqrt{n}} \bfW,$ where $\bfW$ is a Wigner matrix with variance $1$ and $\mathbf{V}$ is a diagonal matrix. To align this definition with our framework \eqref{eq_def_M}, we set $\lambda_0 = 1$ and $\mathbf{V}=\diag(s/\sigma,0,\ldots,0) \in \R^{n\times n}$ such that $\sigma \mathbf{H} \stackrel{\mathcal{D}}{=} \diag(s,0, \ldots, 0) + \frac{1}{\sqrt{n}} \bfW$. By noting that, if the results of \cite{Lee2015} are valid for $\mathbf{H}$, they can also applied to the scaled matrix $\sigma \mathbf{H}$, we want to show that $\mathbf{H}$ satisfies Assumptions 2.1 and 2.2 of the aforementioned work.

For this purpose, denote $\lambda = \frac{s}{\sigma}$. Then the random variable $\lambda$ has a compact support in $[0,1)$. Since $s \sim \PR^{(1)}$, we may choose $0 \leq \lambda_l \leq \lambda_r < 1 $ independent of $n$ such that the support of $\lambda$ is contained in $[\lambda_l, \lambda_r]$.
Moreover, let
\begin{align*}
	F^{\mathbf{V}} = \frac{1}{n} \delta_\lambda + \frac{n-1}{n} \delta_0
\end{align*}
denote the (random) empirical spectral distribution of $\mathbf{V}.$ Then, the weak limit of $F^{\mathbf{V}}$, the so-called limiting spectral distribution of $\mathbf{V}$, is given by the point-mass distribution $\delta_0$ in the origin almost surely. It is straightforward to see that Assumption 2.1 and (2.5) of Assumption 2.2 in \cite{Lee2015} are satisfied. In order to verify (2.6) of Assumption 2.2 in the aforementioned work, it remains to show that
\begin{align}\label{eq_aim_schnelli}
	\inf_{x\in [\lambda_l,\lambda_r]} \int \frac{1}{(v-x)^2} dF^{\mathbf{V}} (v) \geq 1+\varepsilon
\end{align}
for sufficiently large $n$ and for some $\varepsilon>0$ with probability $1$. Note that the left-hand side of \eqref{eq_aim_schnelli} might be infinite for some $x$. In this case, this choice of $x$ may be neglected for showing the condition \eqref{eq_aim_schnelli}.
Using the definition of $F^{\mathbf{V}},$ we have almost surely
\begin{align*}
	\int \frac{1}{(v-x)^2} dF^{\mathbf{V}} (v) = \frac{1}{n} \frac{1}{(x-\lambda)^2} +  \frac{n-1}{n} \frac{1}{x^2} \geq \frac{n-1}{n} \frac{1}{\lambda_r}.
\end{align*}
Since $\lambda_r < 1 $ is fixed with respect to $n$, we find some constant $\varepsilon>0$ such that $(n-1)/( n \lambda_r) \geq 1 + \varepsilon$ for all large $n$.  Thus, we conclude that \eqref{eq_aim_schnelli} holds true.

In summary, we have shown the Assumptions 2.1 and 2.2 of \cite{Lee2015} are satisfied for $\mathbf{H}.$
Then, the assertion follows by combining  \cite[Lemma 4.3]{Lee2015} with the universality result \cite[Theorem 12.4]{Knowles2017}.
 To conclude, note that \cite{Knowles2017} has the assumption that $s \cdot \bfx \bfx^\top$ is deterministic. It can, however, be easily extended to the possibly random case considered in this work (see Remark 2.12 in \cite{knowles:yin:2013:a}).

\medskip

\noindent{\sc Proof of Lemma \ref{lem:pizzo}:}
According to Theorem 1.1 in \cite{pizzo_et_al_2013},
it holds for any (deterministic) sequence $(\mathbf{x}^{(n)})_n$ with $\mathbf{x}^{(n)} \in \mathbb{R}^n$ a vector of unit length, and $s>\sigma$ fixed with respect to $n$ that
\[
r_n := \lambda\Big(s  \cdot \bfx^{(n)}(\bfx^{(n)})^\top + \bfW/\sqrt{n}\Big) -\Big(s+\frac{\sigma^2}{s}\Big)\overset{\mathbb{P}}{\to}0.
\]
\textcolor{black}{Here,} $\lambda(\mathbf{A})$ denotes the largest eigenvalue of a symmetric matrix $\mathbf{A}.$
We will show that this result is actually uniform over the $\mathbf{x}^{(n)}$ in the following sense: Let $S^{(n)} \subset \mathbb{R}^n$ be the Euclidean unit sphere, then for any $c>0$
\begin{align}
	\label{z1}
	\limsup_n \sup_{\mathbf{x}^{(n)} \in S^{(n)}} \mathbb{P}(|r_n|>c)=0.
\end{align}
If this were not so, and the right side was equal to some $\eta>0$, it would immediately follow that we could select a sequence $(\mathbf{x}^{(n,\star)})_n$ such that
\[
\limsup_n\mathbb{P}\bigg(\Big|\lambda\Big(s  \cdot \bfx^{(n,\star)}(\bfx^{(n,\star)})^\top+\bfW/\sqrt{n}\Big)  -\Big(s+\frac{\sigma^2}{s} \Big)\Big|>c\bigg) \ge \eta/2,
\]
contradicting the original theorem. \\
Next, consider our setup, \textcolor{black}{in which} both $s, \mathbf{x}$ are random. In this case, we get
\begin{align*}
	&\mathbb{P}\bigg(\Big|\lambda\Big(s  \cdot \bfx\bfx^\top+ \bfW/\sqrt{n}\Big)-\Big(s+\frac{\sigma^2}{s} \Big) \Big|>c\bigg)\\
	= &\mathbb{E}\bigg[\mathbb{P}\bigg(\Big|\lambda\Big(s  \cdot \bfx\bfx^\top + \bfW/\sqrt{n}\Big) -\Big(s+\frac{\sigma^2}{s} \Big)\Big|>c \Big| s, \mathbf{x}\bigg)\bigg]\\
	\le & \mathbb{E}\bigg[\sup_{\bfx \in S^{(n)}}\mathbb{P}\bigg(\Big|\lambda\Big(s  \cdot \bfx\bfx^\top + \bfW/\sqrt{n}\Big) -\Big(s+\frac{\sigma^2}{s} \Big)\Big|>c\Big| s \Bigg)\bigg] =: \mathbb{E}[f_n(s)].
\end{align*}
Now, for any fixed $s'>\sigma$, it follows from \eqref{z1} that $f_n(s') \to 0$. This convergence is pointwise and $f_n$ is dominated by the constant $1$-function. Accordingly, by the dominated convergence theorem, we have
\[
\lim_{n \to \infty}\mathbb{E}[f_n(s)] = \lim_{n \to \infty}\int f_n(t) d\mathbb{P}^s(t) = \int \lim_{n \to \infty}f_n(t) d\mathbb{P}^s(t)=0.
\]

\subsection{Auxiliary lemmas} \label{ss:au}
Recall the definition of the Prokhorov metric $\pi$ in
\eqref{eq_def_prokhorov_metric}.

\begin{lemma}[Theorem 2.13 of \cite{huber:ronchetti:2009}] \label{lem:huber}
	Let $X,Y$ be two random variables on a separable metric space $(\mathcal{M},d)$.
	Then the following statements are equivalent:
	\begin{itemize}
		\item[i)] $\pi(X,Y)\le  \delta$.
		\item[ii)] There exists a vector $(\tilde X,\tilde Y)$,  s.t. $\mathcal{L}(\tilde X)
		= \mathcal{L}( \textcolor{black}{X})$, $\mathcal{L}(\tilde Y) = \mathcal{L}(Y)$ defined on some suitable probability space, s.t.  $\mathbb{P}(d(\tilde X,\tilde Y)\le \delta)\ge 1-\delta$.
	\end{itemize}
	The implication from i) to ii) also holds without the assumption of separability.
\end{lemma}

\begin{lemma}
	\label{lem_p_t}
	Consider any numbers $D_1, D_2>0$ (arbitrarily large) and any $\varepsilon>0$ (arbitrarily small) and let $A_m \asymp m^{D_1}$ with $A_m \in \mathbb{N}$ for all $m$. Then, it holds for the probability
	\[
	p_{A_m}:= \mathbb{P}\bigg( \bigcup_{t=1}^{A_m} \Lambda_t^c\bigg), \qquad \textnormal{that} \qquad p_{A_m} \lesssim m^{-D_2}.
	\]
	Here the event $\Lambda_t$ is defined in \eqref{eq_def_Lambda_t} and depends on the constant $\varepsilon$.
\end{lemma}

\begin{proof}
	The proof is a simple consequence of Lemma \ref{lem_rmt} together with Assumption \ref{ass_2}, which states that $m\asymp n^\theta$, and hence $n\asymp m^{1/\theta}$. In said lemma, choose
	\[
	D=\theta(D_1+D_2).
	\]
	Together with stationarity of the time series, this implies
	\begin{align*}
		p_{A_m}
		\leq \sum_{t=1}^{A_m}\mathbb{P}(\Lambda_t^c)
		\leq A_m\mathbb{P}(\Lambda_1^c)
		\lesssim A_m n^{-D}
		\asymp A_m m^{-D/\theta}.
	\end{align*}
	Now, using $A_m\lesssim m^{D_1}$, we obtain
	\[
		A_m m^{-D/\theta}
		\lesssim m^{D_1-D/\theta}
		= m^{D_1-(D_1+D_2)}
		= m^{-D_2}.
	\]
	This concludes the proof.
\end{proof}

\begin{lemma} \label{lem:exp}
	Suppose that Assumptions \ref{ass_wigner}, \ref{ass_2} and \ref{ass_3} hold. Then, it follows that
	$\mathbb{E}[\lambda_t] = \mathbb{E}[\lambda_t|\Lambda_t]+Rem,$
	where the remainder satisfies $Rem = \mathcal{O}(m^{-D})$ for any $D>0$ and is independent of $t$.
\end{lemma}
\begin{proof}
	We notice that $\lambda_t = \lambda_t\mathbb{I}\{\Lambda_t\}$ except on a set with probability $\lesssim m^{-D_2}$ (see Lemma \ref{lem_p_t} above), where $D_2$ can be chosen arbitrarily large. Thus, we get
	\[
	\mathbb{E}[\lambda_t] = \mathbb{E}[\lambda_t\mathbb{I}\{\Lambda_t\}] +\mathbb{E}[\lambda_t\mathbb{I}\{\Lambda_t^c\}].
	\]
	Let us now consider $\mathbb{E}[\lambda_t\mathbb{I}\{\Lambda_t^c\}]$, where we observe the upper bound
	\begin{align*}
		\mathbb{E}[\lambda_t\mathbb{I}\{\Lambda_t^c\}] \le \left( \mathbb{E}[\lambda_t^2] \right) ^{1/2} \mathbb{P}\{\Lambda_t^c\}^{1/2} \lesssim \left( \mathbb{E}[\lambda_t^2] \right) ^{1/2}  m^{-D_2/2}.
	\end{align*}
	Since $\lambda_t$ it the biggest eigenvalue of $\mathbf{M}_t$, it is upper bounded by %
	its spectral norm $\|\mathbf{M}_t\|_2$. We thus obtain
	\begin{align*}
		\mathbb{E}[\lambda_t^2] \le \sum_{i,j}   4\mathbb{E}[|(s \cdot \bfx\bfx^\top)_{i,j}|^2 + w_{i,j}^2] \lesssim n^2 \lesssim m^{2/\theta},
	\end{align*}
	where we used Assumption \ref{ass_2} for the last estimate.
	Thus, choosing $D_2$ large enough, we get that $\mathbb{E}[\lambda_t\mathbb{I}\{\Lambda_t^c\}]$ decays at an (arbitrarily fast) polynomial rate in $m$. Using a similar argument, we get that
	\[
	\mathbb{E}[\lambda_t\mathbb{I}\{\Lambda_t\}] = \mathbb{E}[\lambda_t|\Lambda_t]+Rem,
	\]
	where the remainder $Rem$ is decaying at an arbitrarily fast polynomial rate.
\end{proof}

We conclude this section with a Lemma on the convergence of the long-run variance.
Recall the definitions of $\tau_{k,n}$ and $\tau^{(n)}$ in \eqref{eq_def_tau} and the parameter $\varepsilon$ from \eqref{eq_def_Lambda_t}.
\begin{lemma} \label{lem_var}
	Under the conditions of Theorem \ref{lem_3}, it holds for any fixed $\eta \in (0,1)$ and for some universal constant $C=C(\eta)>0$ that
	\[
	|\tau_{k,n}-\tau^{(n)}| \le C n^{2\varepsilon} k^{-(1-\eta)}.
	\]
\end{lemma}

\begin{proof}
	Let $\eta \in (0,1)$ be some constant.  Using stationarity, we can rewrite
	\begin{align*}
		&\tau_{k,n} = \sum_{|h|<k}(1-|h|/k)\mathbb{E}[Y_0Y_h] \\
        =& \tau^{(n)} + \mathcal{O}\Big(  \sum_{|h|<k^\eta} (|h|/k)|\mathbb{E}[Y_0Y_h]|\Big) +  \mathcal{O}\Big(  \sum_{|h|\ge k^\eta} |\mathbb{E}[Y_0Y_h]|\Big).
	\end{align*}
	Since $\mathbb{E}Y_t^2 = \mathcal{O}(n^{2\varepsilon})$ by \eqref{e:bYt}, it follows for the second term on the right that
\begin{align*}
    \sum_{|h|<k^{\eta}} (|h|/k)|\mathbb{E}[Y_0Y_h]|
    \lesssim k^{-1+\eta} n^{2\varepsilon}.
\end{align*}
Next, using the mixing property of the time series, we obtain with the mixing inequality from \cite{dehling:mikosch:sorensen:2002} (\rev{3.17}) for bounded variables that
\[
    |\mathbb{E}[Y_0Y_h]| \lesssim n^{2\varepsilon}\phi(|h|).
\]
Since $\phi(r)\lesssim a^r$ for $a\in(0,1)$ (see Assumption \ref{ass_3}, part ii)), we obtain for the other term
\begin{align*}
    \sum_{|h|\ge k^\eta}|\mathbb{E}[Y_0Y_h]|
    \lesssim n^{2\varepsilon}
    \sum_{|h|\ge k^\eta}a^{|h|}
    \lesssim n^{2\varepsilon}k^{-1+\eta}.
\end{align*}
	This concludes the proof of Lemma \ref{lem_var}.
\end{proof}

\textcolor{black}{
\subsection{Extension to H\"older-weighted rolling-window scans} \label{sec_extension_holeder} 
In this paper, we have developed sequential tests for the emergence of a
supercritical signal, based on a self-normalized version of the CUSUM
aggregation scheme introduced by \cite{horvath:kokoszka:huskova:steinebach:2003}.
This sequential CUSUM statistic has been adapted to many different settings.
Its main advantages are its versatility, power and low computational cost. However, there exist circumstances when its detection delays, i.e. the time between the
emergence and the detection of a change, are longer than necessary.\\
To address this issue, \cite{kutta:doernemann:2025} recently proposed to replace the
standard CUSUM aggregation by a scan statistic with H\"older weights. The idea
is to scan over many local windows and to upweight short windows that contain
the most recent information. This weighting is designed to improve the
detection of recent changes, leading to shorter delays. The technical
prerequisites of this approach are higher than for the classical CUSUM
statistic, since one has to control moduli of continuity of the partial-sum
process on an unbounded time domain. Below, we adapt this idea to the problem
of detecting the emergence of a supercritical signal.\\
 In the entire discussion of this section, we will call the current time $k$, where the data $\{\mathbf{M}_t\}_{t=1,...,k}$ are available. This is a slightly different scaling than in the main body of this paper, but it proves to make the presentation much easier in the following discussion.
Now let us give the construction of the new scan statistic. Suppose we are in the monitoring period, i.e. that $k\ge m+1$ and consider a window
length $h$. We compare the average over the rolling monitoring window
$k-h+1,\ldots,k$
with the average over the full training period $1,\ldots,m$. In terms of the
eigenvalues this gives
\[
\sum_{t=k-h+1}^{k} \lambda_t 
-\frac{h}{m} \sum_{t=1}^{m} \lambda_t.
\]
This quantity is small under the null hypothesis, but it becomes positive
and large if a supercritical signal emerges in the rolling monitoring window.
To be precise, we scan over all admissible rolling-window lengths
$ 1 \le h \le k-m$.
The restriction $h\le k-m$ ensures that the rolling window lies entirely in
the monitoring period. For $\beta\in[0,1/2)$, we then define
\[
\gamma_m^{(\beta)}(h,k)
:= \frac{n^{2/3}}{\sqrt{m}}
\frac{\sum_{t=k-h+1}^{k} \lambda_t 
-\frac{h}{m} \sum_{t=1}^{m} \lambda_t 
}{
(k/m)^{1-\beta}(h/m)^\beta
}.
\]
The factor $(h/m)^\beta$ is the H\"older weight. With a choice of $\beta$ closer to $1/2$, shorter
windows receive more weight, but not so much that fluctuations over
small intervals become dominant. The discounting factor $(k/m)^{1-\beta}$ plays the same
role as in the classical weighted sequential CUSUM statistic: it prevents
late monitoring times from dominating the statistic under the null
hypothesis.
A resulting H\"older-weighted rolling-window scan statistic at time
$k$ is
\[
\Gamma_m^{(O,\beta)}(k)
:=
\max_{1\le h\le k-m}
\gamma_m^{(\beta)}(h,k),
\]
and we can define the self-normalized version as follows:
\[
\Gamma_m^{(SN,\beta)}(k)
:=
\frac{\Gamma_m^{(O,\beta)}(k)}{V_m}.
\]
Therewith, the first time of signal detection is defined as
\[
\widehat k_{\alpha}^{(\beta)}
:=
\inf\left\{
k\ge m+1:
\Gamma_m^{(SN,\beta)}(k)
>
q_{\alpha,\beta}
\right\},
\]
with the convention that $\inf(\emptyset):=\infty$ (no signal is ever detected). In the above formula,  $q_{\alpha,\beta}$ denotes the
$(1-\alpha)$-quantile of the pivotal limiting distribution
\[
\frac{
T_\beta
}{
\displaystyle\int_0^1 |B(s)-sB(1)|\,ds
},
\]
with
\[
T_\beta
:=
\sup_{\substack{y>1,\\0<x\le y-1}}
\frac{
B(y)-B(y-x)-xB(1)
}{
y^{1-\beta}x^\beta
}.
\]
The next result shows that this test decision yields an asymptotic level-$\alpha$
sequential test under the null hypothesis. Recall the definition of $\tau_n$ in \eqref{e:def:tau_n}.
\begin{theorem}
\label{thm:holder_scan}
Suppose that Assumptions \ref{ass_2} and \ref{ass_3}
hold, that $\mathsf H_0$ is true and let $\beta\in[0,1/2)$.
Then
\[
\sup_{k\ge m+1}
\frac{\Gamma_m^{(O,\beta)}(k)}{\sqrt{\tau_n}}
\overset{d}{\to}
T_\beta,
\]
where
\[
T_\beta
:=
\sup_{\substack{y>1,\\0<x\le y-1}}
\frac{
B(y)-B(y-x)-xB(1)
}{
y^{1-\beta}x^\beta
},
\]
and $B$ is a standard Brownian motion. Moreover, $T_\beta<\infty$
almost surely. Consequently, 
\[
\sup_{k\ge m+1}\Gamma_m^{(SN,\beta)}(k)
=
\sup_{k\ge m+1}\frac{\Gamma_m^{(O,\beta)}(k) / \sqrt{\tau_n}} {V_m/\sqrt{\tau_n}}
\overset{d}{\to}
\frac{
T_\beta
}{
\displaystyle \int_0^1 |B(s)-sB(1)|\,ds
}.
\]
\end{theorem}
The scan also suggests a natural estimator of the change location. If the
detector stops at $\widehat k_{\alpha}^{(\beta)}<\infty$, let
\[
\widehat h
\in
\arg\max_{1\le h\le \widehat k_{\alpha}^{(\beta)}-m}
\gamma_m^{(\beta)}(h,\widehat k_{\alpha}^{(\beta)}).
\]
The estimated change location is then defined as
\begin{align} \label{e:est:loc}
\widehat\vartheta
:=
\widehat k_{\alpha}^{(\beta)}-\widehat h.  
\end{align}
Equivalently, the estimated first affected observation is
$\widehat\vartheta+1$ and this corresponds to the $\widehat\vartheta-m$ time in the monitoring period. It is well known that, if the change is relatively large, $\widehat k_{\alpha}^{(\beta)}$ will be close to the true change point location already. However, for smaller changes, more delay is expected, and the change point estimate $\widehat\vartheta$ becomes more useful.
\begin{proof} 
Throughout the proof, constants denoted by $C$ may change from line to line
and do not depend on $m$. We use the notation $P_m(x)$ also for values of
$x$ outside the approximation interval of Theorem \ref{lem_3}, always
understood through the same partial-sum definition as in \eqref{e:def:Pm}.\\
Fix $\rho=1/4$ in Theorem \ref{lem_3}. Hence the Gaussian approximation is
available on an interval of length $T_m\asymp m^{1/4}$. We use the following
consequence of the proof of Theorem \ref{lem_3}: after choosing the
truncation parameter $\varepsilon_1>0$ sufficiently small, there exists
$\nu_1>0$ such that the coupling error satisfies
\[
d_m\le m^{-\nu_1}
\]
for all sufficiently large $m$. Moreover, $\nu_1$ can be chosen uniformly
for all $\varepsilon_1\in(0,\varepsilon_{\max}]$, provided
$\varepsilon_{\max}>0$ is chosen sufficiently small.\\
Let $\nu_2,\nu_3>0$ be such that
\[
\nu_3<1/4,
\qquad
\beta\nu_2<\nu_1.
\]
If $\beta=0$, the second restriction is void. The set of all permissible pairs $(h,k)$ over which we scan is denoted by 
\[
\mathcal I_m
:=
\left\{
(h,k): k\ge m+1,\ 1\le h\le k-m
\right\}.
\]
Fix temporarily a constant $D>1+\nu_3$ (we will talk about the precise choice later). The role of $D$ is to help decompose the set $\mathcal I_m$; more precisely to separate large scales
from very large scales. Define
\[
\begin{aligned}
\mathcal S_{1,m}
&:=
\left\{
(h,k)\in\mathcal I_m:
h\ge m^{1-\nu_2},\ k\le m^{1+\nu_3}
\right\},\\
\mathcal S_{2,m}
&:=
\left\{
(h,k)\in\mathcal I_m:
h<m^{1-\nu_2},\ k\le m^{1+\nu_3}
\right\},\\
\mathcal S_{3,m}(D)
&:=
\left\{
(h,k)\in\mathcal I_m:
m^{1+\nu_3}<k\le m^D
\right\},\\
\mathcal S_{4,m}(D)
&:=
\left\{
(h,k)\in\mathcal I_m:
k>m^D
\right\}.
\end{aligned}
\]
Then
\[
\mathcal I_m
=
\mathcal S_{1,m}
\cup
\mathcal S_{2,m}
\cup
\mathcal S_{3,m}(D)
\cup
\mathcal S_{4,m}(D).
\]
The proof now follows this decomposition; we consider the maximum over each set of scales. The most important one, that dominates the rest, is the maximum over $\mathcal S_{1,m}$.  These are early times $k$ with longer scans $h$. We will see later that the maximum over the other sets (essentially short scans $h$ or late times $k$) are negligible. 
We also notice the important rewriting
\[
\frac{\gamma_m^{(\beta)}(h,k)}{\sqrt{\tau_n}} = 
\frac{
P_m(k/m)-P_m((k-h)/m)-\frac{h}{m}P_m(1)
}{
(k/m)^{1-\beta}(h/m)^\beta
}.
\]
We first focus on $\mathcal S_{1,m}$. On the coupling space from Theorem
\ref{lem_3}, let $B_m$ denote the Brownian motion coupled to $P_m$. Define
\[
\widetilde\gamma_m^{(\beta)}(h,k)
:=
\frac{
B_m(k/m)-B_m((k-h)/m)-\frac{h}{m}B_m(1)
}{
(k/m)^{1-\beta}(h/m)^\beta
}.
\]
Since $k\le m^{1+\nu_3}$ and $\nu_3<1/4$, all time points appearing in
$\mathcal S_{1,m}$ lie in the interval of the Gaussian approximation for all
sufficiently large $m$. Furthermore, on $\mathcal S_{1,m}$,
\[
(k/m)^{1-\beta}\ge1,
\qquad
(h/m)^\beta\ge m^{-\beta\nu_2}.
\]
Using the triangle inequality, we thus get a uniform bound
\begin{align*}
 &  \bigg| \frac{
[P_m(k/m)-P_m((k-h)/m)]-[B_m(k/m)-B_m((k-h)/m)]
}{
(k/m)^{1-\beta}(h/m)^\beta
}\bigg|\\
\le &
2m^{\beta\nu_2}
\sup_{0\le x\le m^{\nu_3}}
|P_m(x)-B_m(x)|.
\end{align*}
Similarly, we have, using $h\leq k$ from the definition of $\mathcal{I}_m$, 
\[
\bigg|\frac{(h/m)P_m(1)-(h/m)B_m(1)}
{(k/m)^{1-\beta}(h/m)^\beta}\bigg|
=
\left(\frac{h}{k}\right)^{1-\beta}
|P_m(1)-B_m(1)|
\le
|P_m(1)-B_m(1)|.
\]
Hence
\[
\sup_{(h,k)\in\mathcal S_{1,m}}
\left|
\gamma_m^{(\beta)}(h,k) / \sqrt{\tau_n}
-
\widetilde\gamma_m^{(\beta)}(h,k)
\right|
=
\mathcal{O}_P(m^{-\nu_1+\beta\nu_2})+\mathcal{O}_P(m^{-\nu_1})
=o_P(1).
\]
In the last step, we have exploited that $\nu_2$ is so small such that $\beta\nu_2<\nu_1$. This means that we have shown so far the Gaussian approximation 
\begin{align} \label{e:rep1}
    \sup_{(h,k)\in\mathcal S_{1,m}}
\gamma_m^{(\beta)}(h,k)/\sqrt{\tau_n}
 = 
\sup_{(h,k)\in\mathcal S_{1,m}}\widetilde\gamma_m^{(\beta)}(h,k)+o_P(1).
\end{align}
\\
Next we show that the maximum over $\mathcal S_{2,m}$ is negligible. For
$(h,k)\in\mathcal S_{2,m}$,
\begin{align} \label{n1}
|\gamma_m^{(\beta)}(h,k)/\sqrt{\tau_n}|
\le
\frac{
\left|P_m(k/m)-P_m((k-h)/m)\right|
}{
(k/m)^{1-\beta}(h/m)^\beta
}
+
\left(\frac{h}{k}\right)^{1-\beta}|P_m(1)|.
\end{align}
Since $h<m^{1-\nu_2}$ and $k\ge m+1$,
\[
\left(\frac{h}{k}\right)^{1-\beta}
\le
Cm^{-\nu_2(1-\beta)}.
\]
Moreover, $P_m(1)=O_P(1)$, and hence the training  correction satisfies
\begin{align} \label{n2}
 \sup_{(h,k)\in \mathcal{S}_{2,m}}   \left(\frac{h}{k}\right)^{1-\beta}|P_m(1)| = o_P(1).
\end{align}
It remains to control the increment term, i.e., the first summand on the right hand side of \eqref{n1}. By the early replacement property
in Remark \ref{rem:early}, we may replace increments of $P_m$ by increments
of the truncated variables $Y_t$ with probability tending to one. Since
$\tau_n$ is bounded away from zero, it is enough to show that, for every
$c>0$,
\[
\PR\left(
\max_{\substack{1\le i\le j\le C m^{1+\nu_3}\\
j-i+1\le m^{1-\nu_2}}}
\frac{
\left|\sum_{t=i}^j Y_t\right|
}{
\sqrt m\,((j-i+1)/m)^\beta
}
>c
\right)
\to0.
\]
Writing $\ell=j-i+1$, we have, since $\ell\le m^{1-\nu_2}$ and
$\beta<1/2$,
\[
\sqrt m\,(\ell/m)^\beta
=
\sqrt{\ell}\,m^{1/2-\beta}\ell^{\beta-1/2}
\ge
\sqrt{\ell}\,m^{(1/2-\beta)\nu_2}.
\] %
Thus, for every $q\ge1$,
\[
\begin{aligned}
&\PR\left(
\max_{\substack{1\le i\le j\le C m^{1+\nu_3}\\
j-i+1\le m^{1-\nu_2}}}
\frac{
\left|\sum_{t=i}^j Y_t\right|
}{
\sqrt m\,((j-i+1)/m)^\beta
}
>c
\right)
\\
&\quad\le
C m^{2(1+\nu_3)}
m^{-q(1/2-\beta)\nu_2}
\sup_{1\le i\le j\le C m^{1+\nu_3}}
\mathbb E
\left|
\frac{\sum_{t=i}^jY_t}{\sqrt{j-i+1}}
\right|^q .
\end{aligned}
\]
In the second line, the first factor $C m^{2(1+\nu_3)}$ comes from an application of the union bound; the other factors come from an application of the Markov bound with $q$th power.
By the moment inequality for $\phi$-mixing variables from Theorem 1 in
\cite{yoshihara:1978}, and using that the variables $Y_t$ are bounded by
$Cm^{\varepsilon_1}$, the last supremum is bounded by
$Cm^{q\varepsilon_1}$. Consequently,
\[
\PR\left(
\max_{(h,k)\in\mathcal S_{2,m}}
|\gamma_m^{(\beta)}(h,k)/\sqrt{\tau_n}|>c
\right)
\le
C m^{2(1+\nu_3)+q\varepsilon_1-q(1/2-\beta)\nu_2}+o(1).
\]
Notice that the $o(1)$ term on the right comes from exploiting the early replacement property and \eqref{n2}.
Choose $q$ sufficiently large and then $\varepsilon_1>0$ sufficiently small
so that
\[
q(1/2-\beta)\nu_2>2(1+\nu_3)+q\varepsilon_1.
\]
Then
\[
\max_{(h,k)\in\mathcal S_{2,m}}
\gamma_m^{(\beta)}(h,k)/\sqrt{\tau_n}=o_P(1).
\] 
We now turn to the large-scale regions. First consider
$\mathcal S_{3,m}(D)$. The contribution of the monitoring increment 
\begin{align}
    \label{eq_monitoring_incr}
    \mathcal{M}_m(h,k) := \frac{
P_m(k/m)-P_m((k-h)/m)}{
(k/m)^{1-\beta}(h/m)^\beta
}.
\end{align}
is
negligible. Indeed, for $(h,k)\in\mathcal S_{3,m}(D)$ we have
$k>m^{1+\nu_3}$ and $h\le k$. Since $\beta<1/2$,
\[
\begin{aligned}
\sqrt m\,(k/m)^{1-\beta}(h/m)^\beta
&=
\sqrt h\, k^{1-\beta}h^{\beta-1/2}m^{-1/2}
\ge
\sqrt h\, k^{1-\beta}k^{\beta-1/2}m^{-1/2}
\\
&=
\sqrt h\,(k/m)^{1/2}
\ge
\sqrt h\,m^{\nu_3/2}.
\end{aligned}
\]
Therefore, for every $c>0$ and every $q\ge1$, we get by using a union bound for the maximum and then again Markov for $q$th moments
\[
\begin{aligned}
&\PR\left(
\max_{(h,k)\in\mathcal S_{3,m}(D)}
 \left| \mathcal{M}_m(h,k) \right| 
>c
\right)
\\
&\quad\le
C m^{2D}
m^{-q\nu_3/2}
\sup_{1\le i\le j\le C m^D}
\mathbb E
\left|
\frac{\sum_{t=i}^jY_t}{\sqrt{j-i+1}}
\right|^q
+o(1)
\\
&\quad\le
C m^{2D+q\varepsilon_1-q\nu_3/2}+o(1).
\end{aligned}
\]
For the fixed value of $D$, choose $q$ sufficiently large and then
$\varepsilon_1>0$ sufficiently small so that
\[
q\nu_3/2>2D+q\varepsilon_1.
\]
Thus the monitoring increment $\mathcal{M}_m(h,k)$
is $o_P(1)$ uniformly over
$(h,k) \in \mathcal S_{3,m}(D)$ and hence we get uniformly over $S_{3,m}(D)$  that
\begin{align} \label{e:gamma1}
\gamma_m^{(\beta)}(h,k) /\sqrt{\tau_n}
=
-\left(\frac{h}{k}\right)^{1-\beta}P_m(1)+o_P(1).
\end{align}
We will comment on the surviving part on the right below. 
Now, for $\mathcal S_{4,m}(D)$ we begin by using a different representation of the monitoring statistic, namely
\[
\frac{\gamma_m^{(\beta)}(h,k)}{\sqrt{\tau_n}} = 
\frac{
\bar P_m(k/m)-\bar P_m((k-h)/m)-\frac{h}{m}\bar P_m(1)
}{
(k/m)^{1-\beta}(h/m)^\beta
},
\]
where
\[
	\bar P_m(x):=\frac{1}{\sqrt{m \tau_n}}\sum_{t=1}^{\lfloor mx\rfloor} \bar Y_t, \quad  \textnormal{for} \quad \bar Y_t:=n^{2/3}(\lambda_t-\mathbb E\lambda_t).
\]
We cannot go back to the (well-behaved) $Y_t$ here, because we consider a domain beyond what is covered by the early replacement property. We have seen a similar problem in the proof of Lemma \ref{lem_4} and the solution is also similar.
We may decompose into two terms, namely
\[
\frac{
\bar P_m(k/m)-\bar P_m((k-h)/m)
}{
(k/m)^{1-\beta}(h/m)^\beta
}
-\frac{\frac{h}{m}\bar P_m(1)
}{
(k/m)^{1-\beta}(h/m)^\beta.
}
\]
We will first focus on the first term
\begin{align*}
  \overline{\mathcal{M}}_n(h,k) :=  \frac{
\bar P_m(k/m)-\bar P_m((k-h)/m)
}{
(k/m)^{1-\beta}(h/m)^\beta
},
\end{align*}
which is more difficult to analyze and, in fact, is asymptotically negligible. To see this we use some crude moment bounds, defining $Z_t:=\frac{\overline Y_t}{n^2}$, and we see that for every $q\ge1$
\[
\sup_t \mathbb E|Z_t|^q\le C_q.
\]
This follows from Assumption \ref{ass_wigner} and the trace-norm bound for Wigner matrices.
Together with the same $\phi$-mixing moment inequality as before, this implies
\[
\sup_{1\le i\le j}
\mathbb E
\left|
\sum_{t=i}^j Z_t
\right|^q
\le
C (j-i+1)^{q/2}.
\]
Now choose \(D\) large. Recall that \(k>m^D\) on \(\mathcal S_{4,m}(D)\), and thus
\[
m^{-1}\ge k^{-1/D}.
\]
Also, since by assumption \(m\asymp n^\theta\), it follows that \(n\le m^{2/\theta}\) eventually. Therefore,
\[
n\le m^{2/\theta}\le k^{2/(\theta D)}.
\]
Hence, using $h\ge1$, eventually
\[
\frac{n^2}{\sqrt m\,(k/m)^{1-\beta}(h/m)^\beta}
\le
\frac{\,k^{\beta+1/D+4/(\theta D)-1}}{\sqrt m}.
\]
Put
\[
a_D:=1-\beta-\frac1D- \frac{4}{\theta D}.
\]
Since $\beta<1/2$, we can choose $D$ so large that $a_D>1/2$. Then, it follows that %
\[
\begin{aligned}
&\PR\left(
\max_{(h,k)\in\mathcal S_{4,m}(D)}
\left| \overline{\mathcal{M}}_n(h,k) \right| 
>c
\right)
\\
& \quad \leq C
\sum_{j\ge Cm^D}\sum_{i=1}^j
\frac{j^{-qa_D}}{(m \tau_n)^{q/2}}(j-i+1)^{q/2}
\\
&\quad\le
C
\sum_{j\ge Cm^D}
j^{1+q/2-q a_D}.
\end{aligned}
\]
Choose $q$ so large that
\[
q(a_D-1/2)>2.
\]
Then the last series tends to zero as $m \to \infty$. Hence the monitoring increment $\overline{\mathcal{M}}_n(h,k)$ is also
$o_P(1)$ uniformly over $\mathcal S_{4,m}(D)$ and we get  uniformly over $\mathcal S_{4,m}(D)$
\begin{align} \label{e:gamma2}
\gamma_m^{(\beta)}(h,k) /\sqrt{\tau_n}
=
-\left(\frac{h}{k}\right)^{1-\beta}\bar P_m(1)+o_P(1).
\end{align}
It is not difficult to show (using the early replacement property) that $\bar P_m(1)= P_m(1)+o_P(1)$ and
consequently (from \eqref{e:gamma1} and \eqref{e:gamma2}) we get
uniformly over
$\mathcal S_{3,m}(D)\cup\mathcal S_{4,m}(D)$ that
\[
\gamma_m^{(\beta)}(h,k) /\sqrt{\tau_n}
=
-\left(\frac{h}{k}\right)^{1-\beta}P_m(1)+o_P(1).
\]
Since $0\le (h/k)^{1-\beta}\le1$, it follows that 
\[
\max_{(h,k)\in\mathcal S_{3,m}(D)\cup\mathcal S_{4,m}(D)}
\gamma_m^{(\beta)}(h,k) /\sqrt{\tau_n}
\le
\{-P_m(1)\}_{+}+o_P(1).
\]
This possible tail contribution is already attained, asymptotically, inside
$\mathcal S_{1,m}$. Indeed, taking $k_m=\lfloor m^{1+\nu_3}\rfloor$ and
$h_m=k_m-m$, we have $(h_m,k_m)\in\mathcal S_{1,m}$ for all sufficiently
large $m$, and
\[
\widetilde\gamma_m^{(\beta)}(h_m,k_m)
=
-B_m(1)+o_P(1).
\]
Taking instead $h_m=\lceil m^{1-\nu_2}\rceil$ and
$k_m=m+h_m$ gives
\[
\widetilde\gamma_m^{(\beta)}(h_m,k_m)=o_P(1).
\]
Therefore
\[
\max_{(h,k)\in\mathcal S_{1,m}}
\widetilde\gamma_m^{(\beta)}(h,k)
\ge
\{-B_m(1)\}_{+}+o_P(1).
\]
Since $P_m(1)-B_m(1)=o_P(1)$ on the coupling space, the large-scale regions
are asymptotically dominated by the Gaussian maximum over $\mathcal S_{1,m}$.
Combining the four regions, we obtain
\[
\frac{1}{\sqrt{\tau_n}}\sup_{k\ge m+1}\Gamma_m^{(O,\beta)}(k)
=
\max_{(h,k)\in\mathcal S_{1,m}}
\widetilde\gamma_m^{(\beta)}(h,k)
+o_P(1).
\]
It remains to identify the limit of the Gaussian maximum over
$\mathcal S_{1,m}$. Since each $B_m$ is a standard Brownian motion, it
suffices to work with a single standard Brownian motion $B$. Define
\[
F(y,x)
:=
\frac{
B(y)-B(y-x)-xB(1)
}{
y^{1-\beta}x^\beta
},
\qquad
0<x\le y-1.
\]
For fixed $A>1$, set
\[
T_{\beta,A}
:=
\sup_{\substack{1<y\le A\\0<x\le y-1}}
F(y,x).
\]
The function $F$ extends continuously to $x=0$ with $F(y,0)=0$ (this is a simple consequence of Lévy's modulus of continuity theorem) and thus the extension is continuous on the compact set
\[
K_A
:=
\left\{
(y,x):1\le y\le A,\ 0\le x\le y-1
\right\}.
\]
Since $K_A$ is compact, $F$ is even uniformly continuous on it (almost surely for a fixed outcome).
 Since the lattice points
$(k/m,h/m)$ with $(h,k)\in\mathcal S_{1,m}$ and $k/m \le A$ are asymptotically dense in
$K_A$, it follows that
\[
\max_{\substack{(h,k)\in\mathcal S_{1,m}\\ k/m\le A}}
\widetilde\gamma_m^{(\beta)}(h,k)
\to
T_{\beta,A}
\]
almost surely for every fixed $A>1$. Notice that we have that $T_{\beta,A} \le T_\beta<\infty$ almost surely, by the main result in \cite{chevallier:2023}. 
Now, we have for every fixed $A>1$,
\[
\liminf_{m\to\infty}
\max_{(h,k)\in\mathcal S_{1,m}}
\widetilde\gamma_m^{(\beta)}(h,k)
\ge
T_{\beta,A}
\]
almost surely, while
\[
\limsup_{m\to\infty}
\max_{(h,k)\in\mathcal S_{1,m}}
\widetilde\gamma_m^{(\beta)}(h,k)
\le
T_\beta
\]
almost surely. Since $T_{\beta,A}\uparrow T_\beta$ almost surely as $A\to\infty$, we get
\[
\max_{(h,k)\in\mathcal S_{1,m}}
\widetilde\gamma_m^{(\beta)}(h,k)
\to
T_\beta
\]
almost surely, and hence also in distribution. Therefore, 
\[
\sup_{k\ge m+1}\Gamma_m^{(O,\beta)}(k) / \sqrt{\tau_n}
\overset{d}{\to}
T_\beta.
\]
Finally, from the proof of Theorem \ref{theomain},
\[
\frac{V_m}{\sqrt{\tau_n}}
=
\int_0^1
\left|
P_m(s)-\frac{\lfloor ms\rfloor}{m}P_m(1)
\right|\,ds
\overset{d}{\to}
\int_0^1 |B(s)-sB(1)|\,ds.
\]
The limiting denominator is almost surely positive. The preceding Gaussian
approximation and the convergence of $V_m/\sqrt{\tau_n}$ hold jointly,
with the same Brownian motion $B$. Hence the continuous mapping theorem
yields 
\[
\sup_{k\ge m+1}\Gamma_m^{(SN,\beta)}(k)
=
\sup_{k\ge m+1}\frac{\Gamma_m^{(O,\beta)}(k) / \sqrt{\tau_n}}{V_m/\sqrt{\tau_n}}
\overset{d}{\to}
\frac{
T_\beta
}{
\displaystyle \int_0^1 |B(s)-sB(1)|\,ds
}.
\]
 $ {}\qquad\qquad\qquad\qquad\qquad\qquad\qquad\qquad\qquad\qquad\qquad\qquad\qquad\qquad\qquad\qquad\qquad$ \end{proof}
} %

\textcolor{black}{
\section{Monitoring beyond the Wigner model} \label{sec:beyond:wigner}
One of the merits of the proposed monitoring procedure is that
the self-normalization makes the method robust to certain deviations from 
model assumptions. In particular, the procedure is not fundamentally tied to the
 Wigner model considered in the previous sections. 
In this section, we formulate high-level assumptions under which the proposed
detector has the same weak limit as in \eqref{e:weakL}. Notice that the proposed theory and notation are developed here under the null hypothesis only, and therefore, we do not introduce in this section the index $j$ that distinguishes between null and alternative distributions. 
The assumptions imposed in the following are sufficiently general to cover random
matrix models beyond the Wigner setting. This will be demonstrated by subsequently considering two specific examples, namely covariance matrices and rank-one matrices. 
For now, let $n\in \N$ be the dimension, and consider a sequence $(\mathbf{A}_{t})_{t\in \mathbb{Z}}$ of generic $n \times n$ real symmetric matrices $\mathbf{A}_{t}.$ Denote by $\lambda(\mathbf{A}_{t})$ the largest eigenvalue of $\mathbf{A}_{t}$. As for $\mathbf{M}_t$, we suppress the dependence of
$\mathbf{A}_t$ on $n$ in the notation.
The next assumption states that an individual matrix $\mathbf{A}_{t}$ satisfies a concentration result akin to Lemma~\ref{lem_rmt}. This type of concentration result is one main ingredient needed in our proof. Importantly, the assumption is required only at the level of an individual matrix, which is the form in which such results are typically available in random matrix theory. Our contribution is to combine this individual-matrix concentration result with Gaussian approximation techniques in order to obtain results for sequential streams of matrices.
\begin{assumption} \label{ass:beyond_1}
There exists $\varpi \in \R$ and a nonrandom sequence $(c_n)_{n\in\N}$ with $0 < \inf_{n\in\N} c_n\leq \sup_{n\in\N} c_n <\infty $ such that the following concentration result holds:  
   For every $\varepsilon>0$ (sufficiently small) and $D>0$ (sufficiently large), there exists an integer $N=N(\varepsilon, D)$ such that 
		\begin{align*}
			\PR ( n^\varpi \left|\lambda(\mathbf{A}_{1}) - c_n \right| \geq n^\varepsilon ) \leq n^{-D}
			\quad \forall \ n\geq N.
            \end{align*}
\end{assumption} 
The remaining assumptions are the counterparts of Assumption \ref{ass_3} for the generic matrix model $\mathbf{A}_t$. To this end, we define the analogue of $\tau^{(n)}$ in \eqref{e:def:tau^n}, which requires additional notation. Let \begin{align*} \Lambda_t^{\mathbf{A}} = \{ n^\varpi \left|\lambda(\mathbf{A}_{t}) - c_n \right| < n^\varepsilon \}\quad \textnormal{
and} \quad b_{(n)}({\mathbf{A}}) := \mathbb{E}[\lambda(\mathbf{A}_1)|\Lambda_1^{\mathbf{A}}].
\end{align*}
Also, we define the random variables 
\begin{align*}
    Y_t^{\mathbf{A}}:= n^\varpi \mathbb{I}\{\Lambda_t^{\mathbf{A}}\}(\lambda(\mathbf{A}_t)-b_{(n)}({\mathbf{A}})) \quad \textnormal{and the parameter} \quad \tau_{\mathbf{A}}^{(n)}:= \sum_{t \in \mathbb{Z}}\mathbb{E}[Y_t^{\mathbf{A}} Y_0^{\mathbf{A}}].
\end{align*}
We are now in the position to formulate the remaining assumption. Recall that the matrices $\mathbf{A}_{t}$ depend on $n$, even though this dependence is suppressed in our notation.
 \begin{assumption} \label{ass:beyond_2} $ $\\[-5ex]
    \begin{enumerate}
        \item For any $n$, the sequence $(\mathbf{A}_{t})_{t \in \mathbb{Z}}$ is strictly stationary. The triangular array $(\mathbf{A}_{t})_{t \in \mathbb{Z}}$ is $\phi$-mixing with coefficients $\phi^{\mathbf{A}}(r) \le C a^{r},$ 
		for all $r\ge 1$
		for some fixed constants $C>0$  and $a \in (0,1)$.
		\item There exist constants $\varepsilon_\tau, c_\tau>0$ such that
		\[
		\inf_{0<\varepsilon \le \varepsilon_\tau}\liminf_{n \to \infty} \tau^{(n)}_{\mathbf{A}}>c_\tau.
		\]
        \item   For every $q\geq 1$, there exists a constant $C_q>0$ such that
\[
    \sup_{n\in\mathbb N}
    \max_{1\leq i,j\leq n}
    \mathbb E\left|(\mathbf A_1)_{ij}\right|^q
    \leq C_q.
\]
    \end{enumerate}
\end{assumption}
We define $\Gamma_{n,m}^{\mathbf{A}}(k)$ as the analogue of \eqref{e:def:gam}, with $\lambda_t$ being replaced by $\lambda(\mathbf{A}_{t}).$ 
The next theorem provides a weak convergence result for $\Gamma_{n,m}^{\mathbf{A}}(k)$ analogous to \eqref{e:weakL}.
\begin{theorem} \label{thm:beyond_wigner}
Suppose that Assumptions \ref{ass_2}, \ref{ass:beyond_1} and \ref{ass:beyond_2} 
are satisfied. Then, 
\begin{align*}
    	\sup_{1 \le k < \infty} \Gamma_{n,m}^{\mathbf{A}}(k)
	\overset{d}{\to}\sup_{0 \le x < \infty} \frac{[B(1+x)-B(1)]-xB(1)}{(1+x)
		\cdot \int_0^1 |B(s)-s B(1)|ds},
	\quad m,n\to\infty
\end{align*}
\end{theorem}
The proof is  analogous to that of Theorem \ref{theomain} and is therefore omitted.
\\ 
In the Wigner case $\mathbf{A}_{t}=\mathbf{M}_t$ for all $t\in \mathbb{Z},$ Assumption \ref{ass:beyond_1} is satisfied under the assumptions of Lemma \ref{lem_rmt}. Thus, under $\mathsf{H}_0$, we recover Theorem \ref{theomain} in this special case. We illustrate the broader applicability of Theorem \ref{thm:beyond_wigner} by
discussing two examples beyond the Wigner model. In these examples, we explain how the concentration result required in Assumption~\ref{ass:beyond_1} can be verified.\\
\noindent\textit{Monitoring sample covariance matrices.} 
An important class of random matrix models is given by sample covariance
matrices. Let $N=N_n$ and suppose that
\[
    \mathbf{A}_t=\frac{1}{N}\sum_{i=1}^{N}
    \mathbf{y}_{i,t}\mathbf{y}_{i,t}^{\top},
    \qquad t\in\mathbb{Z}.
\]
Here, $\mathbf{y}_{1,t}, \ldots, \mathbf{y}_{N,t}\in\mathbb{R}^n$ are independent random vectors of the form $\mathbf{y}_{i,t} = \mathbf{\Sigma}_t^{1/2} \mathbf{z}_{i,t}$, where the components of
$\mathbf{z}_{i,t}$ are independent, centered,  normalized, and have uniformly bounded moments. Also, $\mathbf{\Sigma}_t$ is an $n\times n$ population covariance matrix. We assume
that the dimension $n$ and the sample size $N$ corresponding to an individual sample covariance matrix are of proportional order. This setting is of particular interest since
the largest eigenvalue of $\mathbf{A}_1$ is known to undergo the BBP-phase transition
\citep{baik:ben:peche:2005}. 
 For instance, in the standard spiked model $\mathbf{\Sigma}_1=\operatorname{diag}(\lambda(\mathbf{\Sigma}_1), 1, \ldots, 1)$, the critical threshold for $\lambda(\mathbf{\Sigma}_1)$ is given by $1+\sqrt{n/N}$. If $\lambda(\mathbf{\Sigma}_1)$ falls well below this threshold, the leading sample eigenvalue fluctuates around 
$
    c_n=( 1+\sqrt{n/N})^2.
$
Generalizations of this result are available under regularity assumptions on
$\mathbf{\Sigma}_1$ \citep{elkaroui:2007, lee:schnelli:2016, ding:yang:2021,
Knowles2017}. 
 In particular, eigenvalue rigidity results for single sample covariance
matrices imply that Assumption
\ref{ass:beyond_1} is satisfied with $\varpi=2/3$ in the subcritical regime under moment assumptions on $\mathbf{z}_{i,1}$ and regularity assumptions on $\mathbf{\Sigma}_1$ (see, e.g., Theorem 3.12 in \cite{Knowles2017}). Since our aim here is only to illustrate the applicability of
Theorem \ref{thm:beyond_wigner}, we do not reproduce the precise frameworks and
technical assumptions, and instead refer the reader to these works for detailed
statements.
Consequently, the self-normalized detector $\Gamma_{n,m}^{\mathbf{A}}$ can be applied to a sequence of
weakly dependent sample covariance matrices $\mathbf{A}_1, \mathbf{A}_2, \ldots$ in order to monitor whether the leading eigenvalue
continues to behave according to the subcritical regime. \\
\noindent\textit{Monitoring rank-one matrices.}
As a second example, consider the rank-one matrix model
\[
    \mathbf{A}_t =  \frac{1}{n}\mathbf{y}_t \mathbf{y}_t^\top,
    \qquad t\in\mathbb{Z},
\]
where $\mathbf{y}_t\in\mathbb{R}^n$ is a random vector. In this setting, the largest eigenvalue simplifies to
$
    \lambda(\mathbf{A}_t) = (1/n) \mathbf{y}_t^\top \mathbf{y}_t.
$
The proposed monitoring procedure can still be applied, although there is
no natural notion of subcritical or supercritical spectral behavior as in the
spiked Wigner or sample covariance models. Instead, the detector $\Gamma_{n,m}^{\mathbf{A}}$ monitors
changes in the norm of the vector $\mathbf{y}_t$, or equivalently changes in the
trace of its covariance matrix.
To make this more explicit, suppose that
\[
    \mathbf{y}_1 = \mathbf{\Sigma}_1^{1/2}\mathbf{z}_1,
\]
where the components of \(\mathbf{z}_1\in\mathbb{R}^n\) are independent,
centered, standardized, and have uniformly bounded sub-Gaussian norm. 
Assume in addition that the spectral norm of $\mathbf{\Sigma}_1$ is uniformly bounded and that $c_n = (1/n) \operatorname{tr} \mathbf{\Sigma}_1$ satisfies $\inf_{n\in\N} c_n > 0$. Then, by the Hanson-Wright inequality, see for instance
Theorem 1.1 in \cite{rudelson:2013}, there exists a constant \(C>0\) such that
for all \(0<\varepsilon\le 1/2\)
\[
    \mathbb{P}\left(
    \left|
    \mathbf{z}_1^\top \mathbf{\Sigma}_1 \mathbf{z}_1
    -
    \operatorname{tr}(\mathbf{\Sigma}_1)
    \right|
    > n^{1/2+\varepsilon}
    \right)
    \le
    2\exp(-C n^{2\varepsilon}).
\]
Consequently, Assumption
\ref{ass:beyond_1} is satisfied with $\varpi=1/2$ in this rank-one setting. 
}

 \textcolor{black}{
\section{Additional simulations} \label{sec:add:sim}
In Section \ref{s:fs} we presented illustrative numerical experiments to demonstrate power and level approximation of our new monitoring statistic. Here, we present some additional numerical studies, with the focus on a comparison to a method that is based on individual matrix entries. Choosing a suitable benchmark is non-trivial. While various detection methods for high-dimensional networks exist (see our literature review in Section \ref{sec_lit}), most do not fit the scenario presented in our paper - typically either because they do not study the same type of  random matrices (e.g. dot products instead of Wigner matrices) or because they do not control the false alarm probability as our approach, but rather provide other guarantees, in terms of average run lengths. An apples-to-apples comparison is only possible if inferential aims are similar enough. Hence, we consider for comparison the approach of  \cite{gosmann:stoehr:heiny:dette:2022}, which is a sequential test for the mean change in a high-dimensional vector after an initial training on a sample of size $m$. After vectorizing the upper triangle of the observed matrices $\mathbf{M}_t$ (including the diagonal) and calling the result $\mathbf{v}_t:= vec_\triangle(\mathbf{M}_t)$, their method can be used in principle to detect some distributional shifts from $\mathbb{P}^{(1)}$ to $\mathbb{P}^{(2)}$, by testing mean changes in $(\mathbf{v}_t)_t$. 
We briefly validate that a change of the form $\mathsf{H}_0-\mathsf{H}_1$ in \eqref{e:alt} also entails a mean change of the sequence $(\mathbf{v}_t)_t$, which makes an application of the method by \cite{gosmann:stoehr:heiny:dette:2022} meaningful. 
Suppose (as we will do in our simulations below) that $s_t$ and $\bx_t$ are independent, then comparing the means of pre- and post-change vectors gives
\[
\mathbb{E}[\mathbf{v}_{m+k^*}] - \mathbb{E}[\mathbf{v}_{m+k^*+1}] = \mathbb{E}[s_{m+k^*}-s_{m+k^*+1}] vec_\triangle(\mathbb{E}[\bx_t\bx_t^\top]).
\]
Notice that $\mathbb{E}[\bx_t\bx_t^\top]\neq 0$ because
\[
\operatorname{Tr}[\mathbb{E}  \bx_t\bx_t^\top] =  \mathbb{E}  \operatorname{Tr}[\bx_t\bx_t^\top] = 1
\]
since $\bx_t$ has unit length. Moreover  
\[
\mathbb{E}[s_{m+k^*}-s_{m+k^*+1}] \neq 0
\]
because 
\[
\mathbb{E}[s_{m+k^*}] =\mathbb{E}_{s \sim \mathbb{P}^{(1)}}[s] <\sigma<\mathbb{E}_{s \sim \mathbb{P}^{(2)}}[s]=\mathbb{E}[s_{m+k^*+1}].
\]
This means that the change in our model $\mathsf{H}_0-\mathsf{H}_1$ entails a mean change for the stream of random vectors $(\mathbf{v}_t)_t$, making the application of \cite{gosmann:stoehr:heiny:dette:2022} meaningful. \\
Next, we have to deal with a technical detail of \cite{gosmann:stoehr:heiny:dette:2022}. In principle, the method developed by these authors requires weak spatial dependence, i.e. that the entries of $\mathbf{v}_t$ are weakly dependent, with dependence decaying with distance of indices. This is not realistic in our scenario, because all entries of $\mathbf{v}_t$ depend to some degree on the same variable $s_t$. However, it is clear from the proofs of \cite{gosmann:stoehr:heiny:dette:2022}, that the assumption of weak dependence is only needed for more efficient quantile approximation (using results from extreme value theory) and if the quantiles can be derived in a different way, the method remains statistically valid. Practically, in such general scenarios, quantile approximation for \cite{gosmann:stoehr:heiny:dette:2022} is still feasible, but the task then involves using a high-dimensional bootstrap (either parametric or non-parametric). Such bootstraps are extremely costly, because for a monitoring horizon of length $T$, \cite{gosmann:stoehr:heiny:dette:2022}  require the calculation of about $\mathcal{O}(n^2T^2B)$ maxima (where $B$ is the number of bootstrap repetitions). The $T^2$ comes from the double scan used by the monitoring statistic. 
To keep computational costs of our study manageable, we have thus decided to simply compute the true quantiles of the statistic by \cite{gosmann:stoehr:heiny:dette:2022} under the null hypothesis once for each scenario and to suppose they are known, making the comparison more favorable to \cite{gosmann:stoehr:heiny:dette:2022}. The thus obtained quantiles agreed for all scenarios up to two decimal places, and thus we used the same value of $q=5.85$ in all settings. Our own method is computationally cheap, because (due to self-normalization) it does not require a bootstrap.\\
The following simulation is concerned with comparative detection power rather than size control, which we have already studied in the main paper. Throughout the study we fix the size of the training sample at $m=200$, the number of monitored points at $T=500$, the change point location within the monitoring period at $k^*=400$, the dimension at $n=25$ and the nominal level at $\alpha=0.05$. For the Wigner matrices we consider the relatively simple scenario of 
\[
(\mathbf{W}_t)_{i,j} \sim \mathcal{N}(0, 1+ \delta_{i,j})
\]
with independent entries (up to symmetry) in $i,j$, independence across $t$ and with the definition of $\delta_{i,j}=1$ if $i=j$ and $0$ else. Accordingly $\sigma=1$ in the following. Next, for the directions $\mathbf{x}$, we consider three different scenarios, namely 
\begin{itemize}
    \item[X1)] $\mathbf{x}_t=(1,0,0,\ldots,0)$, i.e., fixed and ultra-sparse $\mathbf{x}_t$.
    \item[X2)] $\mathbf{x}=(1,1,1,\ldots,1)/\sqrt{n}$, i.e., fixed and dense $\mathbf{x}_t$.
    \item[X3)] $\mathbf{x}_t$ i.i.d. across $t$ and independent of everything else, drawn uniformly from the $n$-dimensional unit sphere, i.e., the case of random $\mathbf{x}_t$.
\end{itemize}
For $s_t$ we specify three scenarios, namely
\begin{itemize}
    \item[S1)] $s_t=0.7$, for $t \le m+k^*$ (pre change) and $s_{t} = 1+\delta $ for $t > m+k^*$ (post change).
    \item[S2)] $s_t =0.95 Z_t$, with $(Z_t)_t$ i.i.d. uniform on $[0,1]$, for $t \le m+k^*$ (pre change)  and  $s_t = 1.05 + \delta Z_t$ for $t > m+k^*$ (post change). %
    \item[S3)] $s_t = 0.95 Z_t $, with $(Z_t)_t$ i.i.d. Beta(2,4), for $t \le m+k^*$ (pre change) and $s_t = 1.05+\delta Z_t$ for $t > m+k^*$ (post change).
\end{itemize}
The $s_t$ are assumed  to be independent of each other, also before and after change and independent of everything else. Power curves for $\delta=0.1, 0.2,...,1.5$ are  reported for all $9$ combinations of scenarios in Figure \ref{fig:power_detection_3x3}, the spectral test being our new method and the max-norm test the alternative from \cite{gosmann:stoehr:heiny:dette:2022}. 
\begin{figure}[h]
    \centering  \includegraphics[width=\textwidth,height=0.88\textheight,keepaspectratio]{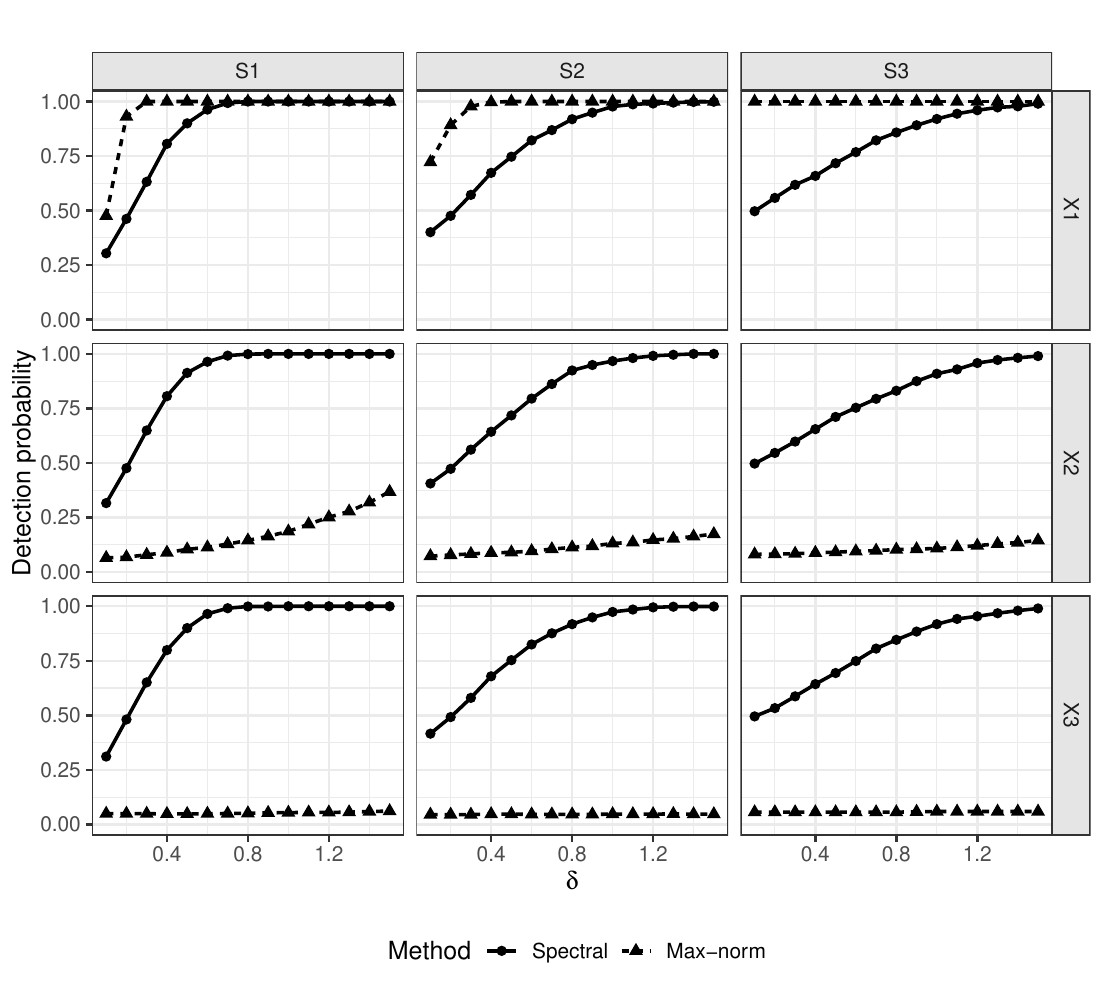}
    \caption{Detection probabilities for the spectral monitoring statistic (our approach) and the entrywise max-norm benchmark across the nine simulation scenarios. Rows correspond to the direction scenarios X1--X3, and columns correspond to the signal distributions S1--S3.}   \label{fig:power_detection_3x3}
\end{figure}
The power of our test is essentially independent of the scenario X1)-X3),  and only depends on the sampling of $s_t$, i.e. on the scenarios S1)-S3). In particular, our approach detects signals equally well, regardless of whether they are sparsely or densely distributed across the matrices. For the approach from \cite{gosmann:stoehr:heiny:dette:2022}, which is max-norm based, as may be expected, high power is attained for a sparse signal but little power when the signal is more dispersed. This means that the X-scenario matters more than the specific S-scenario. \\
Finally, we want to comment on some technical details. First, power for the statistic of  \cite{gosmann:stoehr:heiny:dette:2022} and our own approach do not mean exactly the same thing. Our statistic is specifically made to detect changes to a supercritical regime, while \cite{gosmann:stoehr:heiny:dette:2022} do not make such a differentiation and may have power even when a change happens from one subcritical signal to another one. This is desirable when any change is of interest, yet it may be less meaningful/helpful when the aim is specifically to search for the emergence of a supercritical signal. Second, this simulation study uses a moderate $T$ compared to $m$ ($T=(5/2)m$). Our method is valid for $T$ of any size, because our theory covers the open-end case $T=\infty$. In contrast, theory by \cite{gosmann:stoehr:heiny:dette:2022} only covers scenarios where $T$ is at most proportional to $m$, meaning that long time monitoring is not covered by their theory. It is, however, unclear whether this restriction is practically relevant, because already for moderate $T$ the max-norm statistic (and its bootstrap) becomes computationally very costly. 
\section{Data preprocessing for the Colorado wildfire
example}\label{s:fire} 
Some preprocessing of the raw data was necessary to obtain the daily values. Missing observations were first filled by carrying forward the most recent available value in the corresponding series. Any values that remained missing were replaced by the mean of the other retained measurement series on the same day.
To ensure a consistent 365-day structure across all years, the
observations for February~29 were removed from leap years. 
For the historical period (2015--2018), 
we applied a 30-day rolling mean to smooth out sub-monthly fluctuations, allowing us to
focus on seasonal patterns rather than day-to-day variability. 
This smoothed seasonal baseline was subtracted from the observations to deseasonalize the data. The remaining 483 days (all of 2019 and
approximately the first four months of 2020) serve as the in-sample training period for the monitoring procedure.
The \textsf{R} code for data preprocessing and the monitoring procedure is publicly available at {\tt{https://github.com/adelialee-source/npts}}.
The raw PM$_{2.5}$ data are freely available from the United States Environmental Protection Agency (EPA) Air Quality System (AQS) database at\\
\url{https://www.epa.gov/outdoor-air-quality-data/download-daily-data}.
}

\small\renewcommand{\baselinestretch}{1.2}
\putbib

\end{bibunit}
\end{document}